\documentclass[11pt,reqno]{amsart}
\usepackage{amssymb,mathrsfs,graphicx,subfigure, enumerate}
\usepackage{amsmath,amsfonts,amssymb,amscd,amsthm,bbm}
\usepackage{extpfeil}
\usepackage{graphicx,colortbl}
\usepackage{float}
\usepackage{epsfig}
\usepackage{caption}
\usepackage{bm}
\graphicspath{{Figure/}}

\usepackage[margin=1in]{geometry}

\title[Stability of Outflow Problem of 3D Barotropic Navier--Stokes systems]{Asymptotic stability of viscous shock for the outflow problem of 3D Navier-Stokes system}

\author[Kang]{Moon-Jin Kang}
\address[Moon-Jin Kang]{\newline Department of Mathematical Sciences \newline Korea Advanced Institute of Science and Technology, Daejeon  34141, Republic of Korea}
\email{moonjinkang@kaist.ac.kr}

\author[Kim]{Hyeonseop Kim}
\address[Hyeonseop Kim]{\newline Department of Mathematical Sciences \newline Korea Advanced Institute of Science and Technology, Daejeon  34141, Republic of Korea}
\email{hskim0740@kaist.ac.kr}

\author[Lee]{Hobin Lee}
\address[Hobin Lee]{\newline Department of Mathematical Sciences \newline Seoul National University, Seoul 08826, Republic of Korea}
\email{lcuh11@snu.ac.kr}

\newtheorem{theorem}{Theorem}[section]
\newtheorem{lemma}{Lemma}[section]

\newtheorem{proposition}{Proposition}[section]
\newtheorem{remark}{Remark}[section]

\newcommand{\beq}{\begin{equation}}
\newcommand{\eeq}{\end{equation}}

\newcommand{\bbr}{\mathbb R}

\newcommand{\eps}{\varepsilon }

\newcommand{\bq}{\begin{equation}}
\newcommand{\eq}{\end{equation}}
\newcommand{\e}{\varepsilon}

\newcommand{\pa}{\partial}

\newcommand{\tU}{\tilde{U}}
\newcommand{\tu}{\tilde{u}}

\newcommand{\tv}{\tilde{v}}

\newcommand{\pv}{p(v)}

\newcommand{\tpv}{p(\tilde{v})}

\newcommand{\di}{\displaystyle}

\newcommand{\bu}{\mathbf{u}}

\begin{document}

\date{\today}


\thanks{\textbf{Acknowledgment.} 
The authors are partially supported by  the National Research Foundation of Korea  (NRF-RS-2024-00361663), and Samsung Science and Technology Foundation under Project Number SSTF-BA2102-01. In addition, the author Hobin Lee is partially supported by BK21 SNU Mathematical Sciences Division.}
\maketitle

\begin{abstract} 
We establish the asymptotic stability of planar viscous shock waves for the
outflow problem of the three-dimensional barotropic compressible
Navier--Stokes equations in a half-space, with periodic boundary conditions
imposed in the transverse directions.
For a weak shock located sufficiently
far from the boundary, we prove that, under  small perturbations in $H^{2}$, the outflow problem admits a unique global-in-time
solution and that  the solution converges uniformly to the viscous shock, up to  a dynamical shift, whose velocity time-asymptotically decays.
This provides the first shock-stability result for the multidimensional
outflow problem.
Our proof combines the $a$-contraction method and higher-order
energy estimates adapted to the half-space boundary.
The boundary trace remaining in the zeroth-order estimate is controlled by
differentiating the perturbation system in the tangential directions and
exploiting the favorable sign of the outflow flux.
The highest-order normal derivatives are then recovered from the time and
tangential derivatives through the momentum equation, while the boundary
terms generated by the dynamical shift are controlled using the exponential
decay of the viscous shock tail.
\end{abstract}

\tableofcontents

\section{Introduction}\label{sec:1}
\setcounter{equation}{0}

We study the initial-boundary value problem for the three-dimensional barotropic compressible Navier-Stokes system:
\begin{equation}  \label{main}
\begin{cases} \partial_t\rho+\operatorname{div}_x(\rho \mathbf{u})=0, \qquad \qquad t>0, \ x\in  \mathbb{R}_+ \times \mathbb{T}^2,&  \\
\partial_t(\rho \mathbf{u})+\operatorname{div}_x(\rho \mathbf{u} \otimes \mathbf{u})+\nabla_xp(\rho)=\mu\Delta_x\bu+(\mu+\lambda)\nabla_x \operatorname{div}_x\bu,& \\
(\rho, \mathbf{u})|_{t=0}=(\rho_0(x), \mathbf{u}_0(x)),& \\ 
\mathbf{u}|_{x_1=0}= \mathbf{u}_-, \qquad \mathbf{u}_-=(u_-, 0, 0)\in \mathbb{R}^3.
\end{cases}
\end{equation}
In this system, $\rho=\rho(t,x)\in\mathbb{R}_+$ and $\mathbf{u}=\mathbf{u}(t,x)=(u_1,u_2,u_3)(t,x)\in\mathbb{R}^3$ are the mass density and the velocity field of a fluid, respectively. Under the barotropic condition, the pressure $p$ depends only on the density $\rho$ and follows the $\gamma$-law equation of state: $p(\rho)=b\rho^{\gamma}$ for constants $b>0$ and $\gamma>1$.
 In addition, $\mu$ and $\lambda$ are the viscosity coefficients that satisfy the physical constraints 
\[
\mu>0, \quad 2\mu+3\lambda\ge 0.
\]
Furthermore, the time variable is $t \in \mathbb{R}_+$, and the spatial domain is $\Omega := \mathbb{R}_+ \times \mathbb{T}^2$, where $x=(x_1, x_2, x_3) \in \Omega$ is the spatial variable. Here, the flat torus $\mathbb{T}^2 := \mathbb{R}^2/\mathbb{Z}^2$ naturally corresponds to periodic boundary conditions in the transverse directions. \\
We supplement the system with the far-field condition:
\begin{equation} \label{farfield condition}
(\rho(t,x), \mathbf{u}(t,x)) \to (\rho_{+}, \mathbf{u}_{+}), \quad as \quad x_1 \to +\infty , 
\end{equation}
where $\rho_{+}>0$ and $\mathbf{u}_{+}=(u_{+},0,0)\in \mathbb{R}^3$. 

\noindent Finally, the compatibility conditions are satisfied; that is, the initial data $(\rho_0(x), \mathbf{u}_0(x))$ satisfy both \eqref{main}$_4$ and \eqref{farfield condition}.

It is widely recognized that time-asymptotic profiles are related to the associated Riemann problem for 3D  compressible Euler system:
\begin{align*}
\begin{cases} \partial_t\rho+\operatorname{div}_x(\rho \mathbf{u})=0, \qquad \qquad t>0, \ x\in  \mathbb{R}\times\mathbb{T}^2,&  \\
\partial_t(\rho \mathbf{u})+\operatorname{div}_x(\rho \mathbf{u} \otimes \mathbf{u})+\nabla_xp(\rho)=0,& \\
(\rho, \mathbf{u})|_{t=0}=
\begin{cases}(\rho_-, \mathbf{u}_-),& x_1<0, \\
(\rho_+, \mathbf{u}_+),& x_1>0.
\end{cases}
\end{cases}   
\end{align*}
This problem has two basic types of self-similar solutions, namely, rarefaction waves and shock waves. We focus only on shock waves in this paper. In fact, the planar shock waves are the solution of this 1D systems:
\begin{align*}
\begin{cases} \partial_t\rho+\partial_{x_1}(\rho u_1)=0, \qquad \qquad t>0, \ x_1\in  \mathbb{R},&  \\
\partial_t(\rho u_1)+\partial_{x_1}(\rho u_1^2+p(\rho))=0,& \\
(\rho, u_1)|_{t=0}=
\begin{cases}(\rho_-, u_-),& x_1<0, \\
(\rho_+, u_+),& x_1>0,
\end{cases}
\end{cases}   
\end{align*}
which has a solution of the form:
\begin{equation*}
(\rho, u_1)(t, x_1)=
\begin{cases}
(\rho_-, u_-),& x_1<\sigma t, \\     
(\rho_+, u_+),& x_1>\sigma t.
\end{cases}    
\end{equation*}
Therein, $\sigma$ is determined by Rankine-Hugoniot condition:
\begin{equation}\label{rhoucoorRHcon}
\begin{cases}
-\sigma(\rho_+-\rho_-)+(\rho_+u_+-\rho_-u_-)=0,& \\
-\sigma(\rho_+u_+-\rho_-u_-)+(\rho_+u_+^2-\rho_-u_-^2)+(p(\rho_+)-p(\rho_-))=0.
\end{cases}    
\end{equation}
Moreover, since we consider a 2-shock (an outgoing shock), the following Lax entropy condition holds:
\begin{align}\label{laxcondi}
\begin{cases}
\rho_->\rho_+,& \\u_->u_+.    
\end{cases}   
\end{align}

However, because of the viscous term in equation \eqref{main}$_2$, we are required to incorporate this effect into the aforementioned Euler shock. Its viscous counterpart, called the `viscous shock', is defined in terms of the traveling wave variable $\xi:=x_1-\sigma t$ and satisfies the following ODEs:
\begin{equation} \label{introviscous}
\begin{cases}
-\sigma\tilde{\rho}'+(\tilde{\rho}\tilde{u})'=0, \qquad ':=\frac{d}{d\xi},\\
-\sigma(\tilde{\rho}\tilde{u})'+(\tilde{\rho}\tilde{u}^2)'+p(\tilde{\rho})'=(2\mu+\lambda)\tilde{u}'', \\
(\tilde{\rho}, \tilde{u})(-\infty)=(\rho_-, u_-), \quad (\tilde{\rho}, \tilde{u})(+\infty)=(\rho_+, u_+),
\end{cases}    
\end{equation}
where the conditions \eqref{rhoucoorRHcon} and \eqref{laxcondi} are also satisfied. Pictorially, its profile draws a smooth trajectory connecting smoothly two states $(\rho_-, u_-)$ and $(\rho_+, u_+)$.

In short, our focus is to explore in what manner the solution to \eqref{main}-\eqref{farfield condition} approaches the solution to \eqref{introviscous} as time goes to infinity.

\subsection{Literature Review}
\ 

Viscous shock waves have long been considered important objects of mathematical study. In 1960, Oleinik \cite{ilin1960asymptotic} first established the time-asymptotic stability of shocks in scalar conservation laws. Subsequently, Matsumura and Nishihara \cite{M1985} proved the stability of a single viscous shock for the 1D barotropic Navier-Stokes system under a zero-mass condition. Since then, several studies have extended these results by relaxing the initial restrictions \cite{G86,Liu85,SX93,LZ15}. Recently, Vasseur, Kang, and Wang \cite{KVW23} successfully proved the time-asymptotic stability of a composite wave combining a rarefaction wave and a shock using the method of $a$-contraction with shifts.  Furthermore, in \cite{HKK23}, the authors utilized the same method to establish the stability of a composite wave consisting of two viscous shocks and to remove several prior restrictions.
The $a$-contraction method serves a robust method for the uniform estimates for $L^2$-type perturbations around viscous shock waves. 
This method has been applied to the stability for viscous shocks on the half-space in \cite{han2026convergence,huang2025asymptotic,kang2025asymptotic} (refer also  \cite{eo2025traveling,HKK23,HKKL25,HL25,KV21,KV-Inven,KV-2shock,KVW23,KVW-NSF} for the whole-space setting).  
 Regarding the multi-dimensional case, there are relatively few results compared to 1D. Nevertheless, Zumbrun and his co-authors \cite{zumbrun2001multidimensional,humpherys2017multidimensional} demonstrated the spectral stability of shocks. In more recent works, \cite{wang2022nonlinear,lee4923620long,shi2025stability} established the stability of a single shock or planar composite waves involving viscous shocks. 

Beyond the whole space domain discussed so far, we may also consider spatial domains with boundaries, leading to the study of Initial-Boundary Value Problems (IBVPs). As a starting point, we set the domain to be the 1D half-space $\mathbb{R}_+$. At the boundary $x=0$, a constant state $u(t,0)=u_-$ is given. Depending on the sign of $u_-$, the problem is classified into three cases: 
\begin{align*}
\begin{aligned} 
\begin{cases}
u_->0, \quad inflow \ problem, \\
u_-=0, \quad impermeable \ wall \ problem,   \\
u_-<0, \quad outflow \ problem.
\end{cases}   
\end{aligned}    
\end{align*} 
Here, for the inflow problem, a boundary condition on $\rho$ is additionally given. For each case, the expected time-asymptotic profiles of the barotropic Navier-Stokes system were completely classified by Matsumura \cite{MA99}.

There have been a number of previous works on shock stability for these IBVPs. First, in \cite{MM99} and \cite{HMS03}, the stability of shocks was obtained for the impermeable wall and inflow problems under a zero-mass condition and specific $\gamma$ constraints. Later studies \cite{huang2025asymptotic, han2026convergence} overcame these initial limitations. Specifically, for the inflow problem, the stability of the most generic composite wave combining a boundary layer, a rarefaction wave, and a viscous shock was successfully established in \cite{han2026convergence}. For the outflow problem, however, there were no shock stability results until very recently, when \cite{kang2025asymptotic} obtained the first such result. Even more recently, the authors in \cite{huang2026stability,huang2026long} extended these stability results to the Navier-Stokes-Fourier system.

Second, generalizing the 1D half-space problem, we can also think about multi-dimensional domains such as $\mathbb{R}_+\times \mathbb{T}^2$ or $\mathbb{R}_+\times \mathbb{R}^2$. In \cite{wang2025planar,wang2021large,wang2022long,huang2026asymptotic,kagei2006local,suzuki2021stationary,KK06}, the authors showed a number of stability results for boundary layer and rarefaction waves, but not for shock waves. However, very recently, \cite{chang2025nonlinear} provided the first single shock stability result for the impermeable wall problem under Navier boundary conditions. 

Therefore, our goal in this paper is to obtain the orbital stability of a single shock for the outflow problem. This constitutes the first shock stability result for the outflow problem in the multi-dimensional barotropic Navier-Stokes system. 

\subsection{Main result}

\begin{theorem} \label{thmrho,ucoor}
For a given constant state $(\rho_+, u_+)\in \Psi:=\{(\rho, u)\in\mathbb{R}_+\times \mathbb{R} \ | \ u+\sqrt{p'(\rho)}\ge 0, \quad u<0. \}$, there exist constants $\delta^0, \varepsilon^0>0$ such that the following holds true.

\noindent For any constant state $(\rho_-,u_-)$ satisfying \eqref{rhoucoorRHcon} and \eqref{laxcondi} with
\begin{equation}
|\rho_+-\rho_-|<\delta^0,    
\end{equation}
let $(\tilde{\rho}, \tilde{u})(x_1-\sigma t)$ be the viscous 2-shock wave satisfying \eqref{introviscous} and $\tilde{\mathbf{u}}=(\tilde{u}, 0, 0)$. In addition, let $(\rho_0, \bu_0)$ be any initial data satisfying
\begin{align}
\begin{aligned}\label{smallperturbation}
&\lVert \rho_0-\rho_{+} \rVert_{L^2((\beta, \infty)\times \mathbb{T}^2)}+\lVert \bu_0-\bu_{+} \rVert_{L^2((\beta, \infty)\times\mathbb{T}^2)}\\ 
&+\lVert \rho_0-\rho_{-} \rVert_{L^2((0, \beta)\times \mathbb{T}^2)}+\lVert \bu_0-\bu_{-} \rVert_{L^2((0, \beta)\times\mathbb{T}^2)}+\lVert \nabla_x\rho_0 \rVert_{H^1(\Omega)}+\lVert \nabla_x\bu_0 \rVert_{H^1(\Omega)}<\varepsilon^0.
\end{aligned}
\end{align}
Then the compressible Navier-Stokes system \eqref{main}-\eqref{farfield condition} possesses a unique global-in-time solution $(\rho, \bu)$ in the following sense: there exists Lipschitz continuous function $X(t)$ satisfying

\begin{align}
\begin{aligned} \label{rhoucoordmainthm}  
\rho(t, x)-\tilde{\rho}(x_1-\sigma t-X(t)-\beta)&\in C\left(0, +\infty; H^2(\Omega)\right),\\
\mathbf{u}(t, x)-\mathbf{\tilde{u}}(x_1-\sigma t-X(t)-\beta)&\in C\left(0, +\infty; H^2(\Omega)\right),\\
\nabla_x\left(\rho(t, x)-\tilde{\rho}(x_1-\sigma t-X(t)-\beta)\right)&\in L^2\left(0, +\infty; H^1(\Omega)\right), \\
\nabla_x\left(\mathbf{u}(t, x)-\tilde{\mathbf{u}}(x_1-\sigma t-X(t)-\beta)\right)&\in L^2\left(0, +\infty; H^2(\Omega)\right),\\
\partial_{t}\left(\rho(t, x)-\tilde{\rho}(x_1-\sigma t-X(t)-\beta)\right)&\in L^2(0, +\infty; H^1(\Omega)),\\
\partial_{t}\left(\mathbf{u}(t, x)-\tilde{\mathbf{u}}(x_1-\sigma t-X(t)-\beta)\right)&\in L^2(0, +\infty; H^1(\Omega)).
\end{aligned}
\end{align}

\noindent Furthermore, we obtain the long-time behavior:
\begin{align}
\begin{aligned} \label{rhoucoordmainthm2}  
&\lim\limits_{t\rightarrow +\infty}\sup\limits_{x\in \Omega}\left|\rho(t, x)-\tilde{\rho}(x_1-\sigma t-X(t)-\beta)\right|=0,\\
&\lim\limits_{t\rightarrow +\infty}\sup\limits_{x\in \Omega}\left|\mathbf{u}(t, x)-\tilde{\mathbf{u}}(x_1-\sigma t-X(t)-\beta)\right|=0,
\end{aligned}
\end{align}
where
\begin{equation} \label{derivgrowi}
\lim\limits_{t\to+\infty} |\dot{X}(t)|=0,
\end{equation}
\end{theorem}

\begin{remark}

The constant shift $\beta>0$ and the dynamical shift $X(t)$ play
distinct roles. We choose
$\beta$ so that $\delta\beta$ is sufficiently large. This places the
initial shock front sufficiently far from the boundary and, by the
exponential decay of the viscous shock profile, makes the discrepancy
between the prescribed boundary data and the boundary trace of the
shifted profile exponentially small, of order $e^{-c\delta\beta}$.
Such boundary errors can therefore be absorbed into the energy
estimates.\\
The dynamical shift $X(t)$ contributes the time-modulation of shock, which will be defined by the a-contraction method.  
The asymptotic behavior \eqref{derivgrowi}
implies that $X(t)=o(t)$, so the modulated shock has the same
asymptotic propagation speed $\sigma$ as the unshifted profile.
However, we do not prove that $X(t)$ converges as $t\to\infty$.
Thus, our result establishes orbital convergence to the family of
spatial translates of the viscous shock, rather than convergence to
a single fixed translate.

\end{remark}


\subsection{Idea of proof}

The proof of our main theorem is reduced to establishing the a priori
estimate in Proposition~\ref{prop:main}. 
For the zeroth-order estimate, we employ the method of
$a$-contraction with shifts. To this end, we introduce the effective
velocity
\[
h:=u-(2\mu+\lambda)\nabla_x v,
\]
motivated by the BD entropy structure
\cite{bresch2003existence,bresch2003some,BD06}.
Rewriting the system in the $(v,h)$-coordinates yields a diffusion
term in the equation for $v$, which is unavailable in the original
$(v,u)$-coordinates. This additional dissipation is crucial for
controlling the nonlinear pressure perturbation and the terms
generated by the viscous shock profile. 

A new difficulty in the present multidimensional half-space problem
is that the zeroth-order estimate leaves a boundary trace of the form
\[
\int_{\mathbb T^2}
\left|
\nabla_{x'}(v-\widetilde v)
\right|^2
\bigg|_{x_1=0}
\,dx',
\qquad x'=(x_2,x_3).
\]
 We therefore estimate
this term independently by differentiating the perturbation system
in the tangential directions. Since the prescribed velocity satisfies
the outflow condition $u_-<0$, the transport flux through the boundary
has a favorable sign and produces a boundary dissipation term for the
tangential derivatives. This provides precisely the trace estimate
required to close the zeroth-order energy estimates.

The higher-order estimates require a separate treatment of the normal
and tangential directions (see \cite{KK06}). In particular, a direct top-order energy
estimate in the normal direction would generate boundary traces that
are not directly controllable. Instead, we first estimate the temporal
and tangential derivatives and then recover the pure normal derivatives
from the momentum equation.  Schematically, writing
\[
\phi:=v-\widetilde v,
\qquad
\psi:=u-\widetilde u,
\]
we obtain from the perturbation system~\eqref{eq:peq}
\[
\partial_{x_1}^2\psi
\sim \psi_t\,+\,\text{tangential and l.o.t.},
\qquad
\partial_{x_1}^3\psi
\sim \partial_{x_1}\psi_t
+\text{tangential and l.o.t.}.
\]
Thus, the principal remaining estimates are
\[
\partial_t(\phi,\psi)
\in L^\infty(0,T;L^2(\Omega)),
\qquad
\nabla_x\psi_t
\in L^2(0,T;L^2(\Omega)).
\]

Differentiating the perturbation system with respect to time introduces
the second derivative $\ddot X$ of the dynamical shift. We control this
quantity by differentiating the ODE defining $X(t)$ and using the
lower-order dissipation estimates, which yields
\[
\ddot X\in L^2(0,T).
\]
Another delicate term arises from the boundary contribution in the
time-differentiated energy identity, which contains a term of the form
\[
\int_{\mathbb T^2}
\psi_{1t}\,\partial_{x_1}\psi_{1t}
\bigg|_{x_1=0}
\,dx'.
\]
Since the prescribed boundary velocity is independent of time,
the boundary values of $\psi_t$ and $\psi_{tt}$ are determined by
the derivatives of the shifted shock profile. We therefore integrate
by parts in time on the boundary, transferring one time derivative
and exploiting the exponential decay of the shock tail near
$x_1=0$. The resulting terms are controlled by the estimate of
$\ddot X$ and by the exponentially small boundary mismatch generated
by the large constant shift $\beta$.

The estimates for
\[
\partial_t(\phi,\psi),\qquad
\nabla_x\psi_t,\qquad
\nabla_x^2\phi,\qquad
\nabla_x^3\psi
\]
are mutually coupled and are closed simultaneously through
Lemmas~\ref{L:22}--\ref{L:24}. Combining these estimates with the
zeroth- and lower-order estimates yields Proposition~\ref{prop:main}.

\section{Reformulation of the main result}\label{sec:2}
\setcounter{equation}{0}

To address the transverse perturbation of the viscous planar shock, we use the Helmholtz Decomposition \cite{wang2022nonlinear, lee4923620long}, given by 
\[\Delta_{x}\mathbf{u}=\nabla_x\operatorname{div}_x\mathbf{u}- \nabla_{x}\times \left(\nabla_{x}\times \mathbf{u}\right),\]
where $\nabla_x\operatorname{div}_x\mathbf{u}$ denotes the irrotational part of $\Delta \mathbf{u}$ and $\nabla\times (\nabla\times \mathbf{u})$ denotes the rotational part.

We use the above decomposition and the specific volume $v=\frac{1}{\rho}$ to rewrite the system \eqref{main} into 
(as in \cite{MA99} for 1D case)
\begin{equation} \label{changed main}
\begin{cases}
\rho\left( \partial_tv+\mathbf{u}\cdot \nabla_{x}v \right)=\operatorname{div}_{x}\mathbf{u}, \\
\rho\left( \partial_t\mathbf{u}+\mathbf{u}\nabla_{x}\mathbf{u} \right)+\nabla_{x}p(v)=(2\mu+\lambda)\nabla_{x}\operatorname{div}_{x}\mathbf{u}-\mu \nabla_{x}\times \left(\nabla_{x}\times \mathbf{u}\right), \\
(v, \mathbf{u})|_{t=0}=(v_0(x), \mathbf{u}_0(x)), \\
\mathbf{u}|_{x_1=0}=(u_-,0, 0),
\end{cases}
\end{equation}
together with the far-field conditions
\beq \label{nini}
(v(t, x), \mathbf{u}(t,x)) \to (v_{+}, \mathbf{u}_{+}), \quad as \quad x_1 \to +\infty , 
\eeq
where $p(v)=v^{-\gamma}$. \\

Recall the viscous 2-shock equations in Eulerian Coordinates (\cite{wang2022nonlinear, lee4923620long}):
\begin{equation} \label{original viscous shock equation}
\begin{cases}
-\sigma\tilde{\rho}'+(\tilde{\rho}\tilde{u})'=0, \qquad ':=\frac{d}{d\xi},\\
-\sigma(\tilde{\rho}\tilde{u})'+(\tilde{\rho}\tilde{u}^2)'+p(\tilde{\rho})'=(2\mu+\lambda)\tilde{u}'', \\
(\tilde{\rho}, \tilde{u})(-\infty)=(\rho_-, u_-), \quad (\tilde{\rho}, \tilde{u})(+\infty)=(\rho_+, u_+).
\end{cases}    
\end{equation}

Introducing $\tilde{v}:=1/\tilde{\rho}$, the first two equations in
\eqref{original viscous shock equation} can be rewritten as follows:
\begin{equation} \label{changed viscous shock equation}
\begin{cases}
\tilde{\rho}\left(-\sigma\tilde{v}'+\tilde{u}\tilde{v}'\right)=\tilde{u}', \\
\tilde{\rho}\left(-\sigma\tilde{u}'+\tilde{u}\tilde{u}'\right)
+p(\tilde{v})'=(2\mu+\lambda)\tilde{u}'',
\end{cases}
\end{equation}
where $p(\tilde{v})=\tilde{v}^{-\gamma}$. 

Integrating \eqref{original viscous shock equation}$_1$ on $(-\infty, x_1-\sigma t)$, we have
\begin{equation} \label{changed RH condition}
-\sigma\tilde{\rho}+\tilde{\rho}\tilde{u}=-\sigma\rho_-+\rho_-u_{-}.
\end{equation}

Let 
\begin{equation} \label{changed shock speed}
\sigma^* := \sigma\rho_- - \rho_-u_{-}.
\end{equation}

From \eqref{changed RH condition} and \eqref{changed shock speed}, the equations \eqref{changed viscous shock equation} and the far-field conditions can be rewritten as
\begin{equation} \label{change shock equation 1}
\begin{cases}
-\sigma^*\tilde{v}'=\tilde{u}',\\ 
-\sigma^*\tilde{u}'+p(\tilde{v})'=(2\mu+\lambda)\tilde{u}'',
\end{cases}    
\end{equation}
and
\begin{equation} \label{change shock equation 2}
(\tilde{v},\tilde{u})(-\infty)=(v_-,u_-), \quad (\tilde{v},\tilde{u})(+\infty)=(v_+,u_+), \quad v_-=\frac{1}{\rho_-}, \  v_+=\frac{1}{\rho_+}.
\end{equation}

Integrating \eqref{change shock equation 1} with respect to $\xi$, we have
\begin{equation} \label{changed RH condition 2}    
\begin{cases} 
-\sigma^*(v_+-v_-)=u_{+}-u_{-}, \\
-\sigma^*(u_{+}-u_{-})+p(v_+)-p(v_-)=0.
\end{cases}
\end{equation}
Since our viscous shock is 2-shock, \eqref{changed RH condition 2} yields
 \beq\label{sspeed}
 \sigma^*=\sqrt{-\frac{p(v_+)-p(v_-)}{v_+-v_-}}.
\eeq

In the remaining part of the paper, we will prove the following theorem that is stated in terms of the volume variable. Of course, Theorem \ref{thmrho,ucoor} obviously implies Theorem \ref{thmv,ucoor}, since we deal with small perturbations in $L^\infty ((0,\infty)\times\Omega)$ and so the values of $v$ and $\rho$ stay near the reference point.

\begin{theorem} \label{thmv,ucoor}
For a given constant state $(\rho_+, u_+)\in \Psi$ with $v_+=\frac{1}{\rho_+}$, there exist constants $\delta_0, \varepsilon_0>0$ such that the following holds true.

\noindent For any constant state $(v_-,u_-)$ satisfying \eqref{changed RH condition 2} and \eqref{sspeed} with
\begin{equation}
|v_+-v_-|<\delta_0,    
\end{equation}
let $(\tilde{v}, \tilde{u})(x_1-\sigma t)$ be the viscous 2-shock wave satisfying \eqref{changed viscous shock equation} and $\tilde{\mathbf{u}}=(\tilde{u}, 0, 0)$. In addition, let $(v_0, \bu_0)$ be any initial data satisfying
\begin{align*}
&\lVert v_0-v_{+} \rVert_{L^2((\beta, \infty)\times \mathbb{T}^2)}+\lVert \bu_0-\bu_{+} \rVert_{L^2((\beta, \infty)\times\mathbb{T}^2)}\\ 
&+\lVert v_0-v_{-} \rVert_{L^2((0, \beta)\times \mathbb{T}^2)}+\lVert \bu_0-\bu_{-} \rVert_{L^2((0, \beta)\times\mathbb{T}^2)}+\lVert \nabla_xv_0 \rVert_{H^1(\Omega)}+\lVert \nabla_x\bu_0 \rVert_{H^1(\Omega)}<\varepsilon_0.
\end{align*}
Then the compressible Navier-Stokes system \eqref{changed main}-\eqref{nini} possesses a unique global-in-time solution $(v, \bu)$ in the following sense: there exists Lipschitz continuous function $X(t)$ satisfying

\begin{align}
\begin{aligned} \label{thmstate1}  
v(t, x)-\tilde{v}(x_1-\sigma t-X(t)-\beta)&\in C\left(0, +\infty; H^2(\Omega)\right),\\
\mathbf{u}(t, x)-\mathbf{\tilde{u}}(x_1-\sigma t-X(t)-\beta)&\in C\left(0, +\infty; H^2(\Omega)\right),\\
\nabla_x\left(v(t, x)-\tilde{v}(x_1-\sigma t-X(t)-\beta)\right)&\in L^2\left(0, +\infty; H^1(\Omega)\right), \\
\nabla_x\left(\mathbf{u}(t, x)-\tilde{\mathbf{u}}(x_1-\sigma t-X(t)-\beta)\right)&\in L^2\left(0, +\infty; H^2(\Omega)\right),\\
\partial_{t}\left(v(t, x)-\tilde{v}(x_1-\sigma t-X(t)-\beta)\right)&\in L^2(0, +\infty; H^1(\Omega)),\\
\partial_{t}\left(\mathbf{u}(t, x)-\tilde{\mathbf{u}}(x_1-\sigma t-X(t)-\beta)\right)&\in L^2(0, +\infty; H^1(\Omega)).
\end{aligned}
\end{align}

\noindent Furthermore, we obtain the long-time behavior:
\begin{align}
\begin{aligned} \label{thmstate2}
&\lim\limits_{t\rightarrow +\infty}\sup\limits_{x\in \Omega}\left|v(t, x)-\tilde{v}(x_1-\sigma t-X(t)-\beta)\right|=0,\\
&\lim\limits_{t\rightarrow +\infty}\sup\limits_{x\in \Omega}\left|\mathbf{u}(t, x)-\tilde{\mathbf{u}}(x_1-\sigma t-X(t)-\beta)\right|=0,
\end{aligned}
\end{align}
where
\begin{equation} \label{as}
\lim\limits_{t\rightarrow +\infty}\dot{X}(t)=0.
\end{equation}
\end{theorem}

\section{Preliminaries}
\setcounter{equation}{0}

In this section, we present several analytical tools which play useful roles in the proof.

\subsection{Estimates on viscous shock}

Note that the ODE system \eqref{change shock equation 1} with \eqref{change shock equation 2}-\eqref{sspeed} shares the same structure as that of viscous shocks in Lagrangian mass coordinates. Consequently, the following lemma regarding the properties of viscous shocks follows directly from previous results \cite{MN85,KV21}. 
\begin{lemma}\label{lem:shock-est}
	Fix any given state $U_{+}:=(v_+, u_+)\in \mathbb{R}_+\times \mathbb{R}$. Then there exist a constant $C>0$ such that the following holds: 
  For every end state $U_-:=(v_-, u_-)$ connected with $U_+$ via 2-shock curve, there exists a unique solution $(\tilde{v}, \tilde{u})(\xi)$ joining from $U_-$ to $U_+$ satisfying $\tilde{v}(0)=\frac{v_-+v_+}{2}$ without loss of generality, where $\xi:=x_1-\sigma t$. Let $\delta:=|p(v_-)-p(v_+)|$.
Then the following estimates holds:
\[|\tilde{v}(\xi)-v_{\pm}|, |\tilde{u}(\xi)-u_{\pm}|\le C\delta e^{-C\delta|\xi|}, \quad \pm\xi>0,\]
\[
\tilde{v}'\sim \tilde{u}' \quad i.e., \quad \frac{|\tilde{v}'(\xi)|}{C}\le |\tilde{u}'(\xi)|\le C|\tilde{v}'(\xi)|, \quad for \ \ \xi\in \mathbb{R},
\]
\[
\frac{\delta^2e^{-C\delta|\xi|}}{C}\le \tilde{v}'(\xi)\le C\delta^2e^{-C\delta|\xi|},
\] 
in addition, 
\[
|(\tilde{v}''(\xi),\tilde{u}''(\xi))|\le C\delta|(\tilde{v}'(\xi),\tilde{u}'(\xi))|,
\]
\[
|(\tilde{v}'''(\xi),\tilde{u}'''(\xi))|\le C\delta^2|(\tilde{v}'(\xi), \tilde{u}'(\xi))|.
\]
\end{lemma}

\subsection{Useful inequalities}
Here, we present several inequalities that will be useful in the subsequent proofs.

\subsubsection{Sobolev-type inequalities} 
First, we introduce three Sobolev-type inequalities. Lemma \ref{lem-poin} is an extension of the 1D Poincaré-type inequalities in \cite[Lemma 2.9]{KV21} and \cite[Lemma 3.1]{wang2022nonlinear} to the multi-dimensional domain $[\alpha,1]\times\mathbb{T}^2$. Lemma \ref{infty interpolation inequality lemma} gives an interpolation inequality over $\mathbb{R}_+\times \mathbb{T}^2$. Finally, Lemma \ref{GNIsevera} is a straightforward application of the well-known Gagliardo-Nirenberg interpolation inequality.
\begin{lemma} \label{lem-poin}
	For any $f:\left[ \alpha, 1 \right]\times \mathbb{T}^2 \to \mathbb{R}$ and $\alpha\in (0, 1)$ satisfying
\[\int_{\mathbb{T}^2} \int_{\alpha}^{1} \left[ (y_1-\alpha)(1-y_1)\left| \partial_{y_1}f \right|^2+\frac{\left| \nabla_{y'}f \right|^2}{(y_1-\alpha)(1-y_1)} \right]  \, dy_1 \, dy'<\infty,\]
we have
\begin{equation} \label{poincare inequality}
\begin{aligned}
&\int_{\mathbb{T}^2} \int_{\alpha}^{1} \left| f-\bar{f} \right|^2 \, dy_1 \, dy'\le \frac{1}{2} \int_{\mathbb{T}^2} \int_{\alpha}^{1}  (y_1-\alpha)(1-y_1)\left| \partial_{y_1}f \right|^2 \, dy_1 \, dy' \\
& \quad \quad \quad \quad \quad \quad \quad \quad \quad \quad \quad \ +\frac{1}{16\pi^2}\int_{\mathbb{T}^2} \int_{\alpha}^{1} \frac{\left| \nabla_{y'}f \right|^2}{y_1(1-y_1)} \, dy_1 \, dy',
\end{aligned}
\end{equation}
where $\bar{f}:=\frac{1}{(1-\alpha)}\int_{\mathbb{T}^2}\int_{\alpha}^{1}fdy_1dy'$.
\end{lemma}
Since we can easily prove this Lemma by applying a change of variables to \cite[Lemma 3.1]{wang2022nonlinear}, we omit the detailed proof.

\begin{lemma} \label{infty interpolation inequality lemma}
For any $g \in H^2(\Omega)$ where $\Omega=\mathbb{R}_+\times\mathbb{T}^2$, we have
\[ 
\lVert g \rVert_{L^{\infty}(\Omega)} \le C\lVert g \rVert^{\frac{1}{2}}_{L^2(\Omega)}\lVert \partial_{x_1}g \rVert^{\frac{1}{2}}_{L^2(\Omega)} + C\lVert \nabla_xg \rVert^{\frac{1}{2}}_{L^2(\Omega)}\lVert \nabla_x^2g \rVert^{\frac{1}{2}}_{L^2(\Omega)}.
\]
\end{lemma}
We refer to \cite{adams2003jjf} for its proof.  

\begin{lemma} \label{GNIsevera}
For any $f : \Omega \to \mathbb{R}$ belonging to $H^1$, it holds from Gagliardo-Nirenberg interpolation inequality that
\begin{align} \label{GNIAPP1}
\lVert f \rVert_{L^3}\le C\sqrt{ \lVert f \rVert_{L^6}}\sqrt{ \lVert f \rVert_{L^2}}.   
\end{align}
On the other hand, using Sobolev embedding, we have
\begin{align} \label{GNIAPP2}
\lVert f \rVert_{L^6}\le C\lVert f \rVert_{H^1}.    
\end{align}
Combining \eqref{GNIAPP1} and \eqref{GNIAPP2}, we get
\begin{align} \label{GNIAPP3}
\lVert f \rVert_{L^3}\le C \lVert f \rVert_{H^1}.    
\end{align}
Finally, by H\"older's inequality and \eqref{GNIAPP2}-\eqref{GNIAPP3}, we get
\begin{align} \label{GNIAPP4}
\lVert fg \rVert_{L^2}\le \lVert f \rVert_{L^3}\lVert g \rVert_{L^6}\le C\lVert f \rVert_{H^1}\lVert g \rVert_{H^1}.   
\end{align}
\end{lemma}

\subsubsection{Estimates on the relative quantities}
Dafermos \cite{D79} and DiPerna \cite{diperna1979uniqueness} exploit the relative entropy method for the $L^2$ stability and uniqueness of Lipschitz solutions to the hyperbolic conservation laws. We adopt this method. Let $p(v)=v^{-\gamma}$ and $Q(v)=\frac{v^{-\gamma+1}}{\gamma-1}$. To estimate the wave perturbations, we consider the corresponding relative quantities defined as
\[p(v|w) := p(v)-p(w)-p'(w)(v-w),\quad Q(v|w) := Q(v)-Q(w)-Q'(w)(v-w).\]
These quantities exhibit a locally quadratic structure. In particular, $(3)$ gives sharp estimates in terms of local dynamics. This lemma can be proved by using a Taylor expansion. Since the proof can be found in \cite{KV21}, we omit it here.

\begin{lemma}\label{lem-rel-quant}
	Let $\gamma>1$ and $v_+$ be given constants. Then there exist constants $C,\delta_*>0$ such that the following assertions hold:
	\begin{enumerate}
		\item For any $v,w$ satisfying $0<w<2v_+$ and $0<v<3v_+$,
		\begin{equation}\label{est-rel-1}
		|v-w|^2\le CQ(v|w),\quad |v-w|^2\le Cp(v|w).
		\end{equation}
		\item For any $v,w$ satisfying $v,w>v_+/2$,
		\begin{equation}\label{est-rel-2}
		\frac{1}{C}|v-w|\le|p(v)-p(w)|\le C|v-w|.
		\end{equation}
		\item For any $0<\delta<\delta_*$ and any $(v,w)\in\bbr_+^2$ satisfying $|p(v)-p(w)|<\delta$ and $|p(w)-p(v_+)|<\delta$,
		\begin{align}
		\begin{aligned}\label{est-rel-3}
		&p(v|w) \le \left(\frac{\gamma+1}{2\gamma}\frac{1}{p(w)}+C\delta\right)|p(v)-p(w)|^2,\\
		&Q(v|w) \ge \frac{p(w)^{-\frac{1}{\gamma}-1}}{2\gamma}|p(v)-p(w)|^2-\frac{1+\gamma}{3\gamma^2}p(w)^{-\frac{1}{\gamma}-2}|p(v)-p(w)|^3,\\ 
		&Q(v|w)\le \left(\frac{p(w)^{-\frac{1}{\gamma}-1}}{2\gamma}+C\delta\right)|p(v)-p(w)|^2.
		\end{aligned}
		\end{align}
	\end{enumerate}
\end{lemma}

\subsection{Vector Identities}
In multi-dimensional settings, several vector identities are highly useful for facilitating the estimation process. We present these identities in advance in this subsection.
\begin{lemma} \label{vectori}
For scalar functions $f$ and $g$, and vector functions $\mathbf{F}$ and $\mathbf{G}$, the following identities hold:
\begin{enumerate}
		\item
		\qquad \qquad \qquad \qquad \qquad $\operatorname{div}(\mathbf{F}\times \mathbf{G})=(\nabla\times \mathbf{F})\cdot\mathbf{G}-(\nabla\times \mathbf{G})\cdot \mathbf{F}$
        \vspace{1em}
		\item 
		\qquad \qquad \qquad \qquad \qquad \qquad \qquad \qquad
		$\nabla\times\nabla f\equiv 0$
        \vspace{1em}
        \item 
         \qquad \qquad \qquad \qquad \qquad \qquad \qquad  $\nabla(fg)=g\nabla f+f\nabla g $ 
         \vspace{1em}
        \item 
        \qquad \qquad \qquad \qquad \qquad \qquad \qquad 
         $\operatorname{div}(f\mathbf{F})=f\operatorname{div}\mathbf{F}+\nabla f \cdot \mathbf{F}$
         \vspace{1em}
        \item
        \qquad \qquad \qquad \qquad \qquad \qquad \qquad 
         $\operatorname{div}(\mathbf{F}\nabla\mathbf{G})=\mathbf{F}\cdot\Delta \mathbf{G}+ \nabla \mathbf{F} : \nabla \mathbf{G}$
         \vspace{1em}
        \item 
        \qquad \qquad \qquad \qquad \qquad \qquad \qquad 
    $(\mathbf{F}\nabla\mathbf{G})\cdot \mathbf{G}=\frac{1}{2}\mathbf{F} \cdot \nabla|\mathbf{G}|^2$
     \vspace{1em}
            \item
        \qquad \qquad \qquad \qquad \quad
         $\nabla(\mathbf{F}\cdot\mathbf{G})=\mathbf{F}\nabla\mathbf{G}+\mathbf{G}\nabla\mathbf{F}+\mathbf{F}\times(\nabla\times \mathbf{G})+\mathbf{G}\times(\nabla\times \mathbf{F}).$
        
	\end{enumerate}   
\end{lemma}

\section{Main proposition for proof of Theorem \ref{thmv,ucoor}}\label{sec:3}
\setcounter{equation}{0}
In this section, we present a main proposition on the \textit{a priori} estimates for the proof of Theorem \ref{thmv,ucoor}. 

\subsection{Local existence}
For the local-in-time existence of solutions to \eqref{thmstate1}, we employ the classical result of \cite{kagei2006local} without further elaboration. Hence, provided that local existence is guaranteed, we hereafter dedicate our attention to obtaining the \textit{a priori} estimates.

\subsection{Construction of weight function}
In order to control the shock wave, we first construct a weight function $a : \mathbb{R} \to \mathbb{R}$. This function will yield several ``good terms" that are highly useful in the subsequent estimates. Recall $\delta$ is the shock strengh. We define
\begin{equation} \label{weightdef}
a(\xi):= 1+\frac{p(v_-)-p(\tilde{v}(\xi))}{\sqrt{\delta}}, \qquad \xi=x_1-\sigma t.
\end{equation}
Note that 
\begin{align}
a_{x_1} = -\frac{1}{\sqrt{\delta}}p(\tilde{v})_{x_1}>0,    
\end{align}\label{avd}
which produces
\beq\label{inta}
|a_{x_1}| \sim \frac{1}{\sqrt{\delta}} |\tilde{v}_{x_1}|,\quad \mbox{and so,}\quad  \|a_{x_1}\|_{L^\infty(\bbr)}\le \delta\sqrt{\delta},\quad  \|a_{x_1}\|_{L^1(\bbr)} =\sqrt{\delta}.
\eeq
In addition, from the increasing property and the smallness of $\delta$, we have
\[1 \le  a \le 1+\sqrt{\delta} \le \frac{5}{4}.\]

\subsection{Construction of shift functions}
Our aim is to prove the orbital stability of the viscous shock. Thus, we consider the shifted viscous shock instead of the original one as follows:
\begin{align}
	\begin{aligned}\label{composite_wave}
	(\tv^{X, \beta},\tilde{u}^{X, \beta})(t,x)&:= \left(\tv(x_1-\sigma t-X(t)-\beta),\tu(x_1-\sigma t-X(t)-\beta)\right).
	\end{aligned}
\end{align}
First, a sufficiently large constant $\beta$ serves as a parameter to shift the viscous shock far away from the boundary effect, and it depends only on the shock amplitude $\delta$. The function $X(t)$, on the other hand, is a dynamical shift for obtaining $H^2$-stability, defined as a solution to the following ODEs. 
We define a shift $X$ as a solution to the ODEs for the time-variable $t$:
\begin{equation}\label{X(t)}
\left\{
\begin{array}{ll}
\di \dot{X}(t)=-\frac{M}{\delta}\bigg[
 \int_{\Omega}\frac{a^{X, \beta}}{\sigma^*}\rho\tilde{h}_{x_1}^{X, \beta}(p(v)-p(\tilde{v}^{X, \beta}))\,dx  -\int_{\Omega}a^{X,\beta} \rho p(\tilde{v}^{X, \beta})_{x_1}(v-\tilde{v}^{X, \beta})\,dx\bigg],\\ 
\di  X(0)=0.
\end{array}
\right.
\end{equation}

\noindent Here, $a^{X,\beta}$ is the shifted weight function as defined in \eqref{weightdef}, $\tilde{h}:=\tilde{u}-(2\mu+\lambda)\partial_{x_1}\tilde{v}$, and $M$ is a particular constant suitably chosen as $M:=\frac{3}{2}\sigma_+^4v_+^2\alpha_+$, where $\sigma_+=\sqrt{-p'(v_+)}$ and $\alpha_+=\frac{\gamma+1}{2\gamma \sigma_+ p(v_+)}$. 
The Cauchy-Lipschitz theorem \cite[Lemma 3.1]{HKK23} guarantees the well-posedness of the above ODEs, and in particular, that the solution is Lipschitz continuous. We omit the detailed proof since it can be found in the reference.

\subsection{Main proposition for \textit{a priori} estimates} We here present the main proposition for \textit{a priori} estimates.

\begin{proposition} \label{prop:main}
Let $(\tilde{v}^{X,\beta},\tilde{u}^{X, \beta})$ denote the shifted viscous 2-shock wave as in \eqref{changed viscous shock equation} and \eqref{composite_wave}. For a given constant $U_+:=(v_+,\mathbf{u}_+)\in \mathbb{R_+}\times (\mathbb{R}\times\left\{ 0 \right\}\times\left\{ 0 \right\})$, there exist positive constants $\delta_0$, $C_0$, and $\varepsilon$ such that the following holds:   
For any constant state $U_-:=(v_-,u_-)$ satisfying \eqref{changed RH condition 2} 
satisfying $|p(v_-)-p(v_+)|=:\delta<\delta_0$, if $(v,\mathbf{u})$ is the solution to \eqref{changed main} on $\left[ 0,T \right]$ for some $T>0$ and satisfy 
\begin{equation*}
\begin{aligned}
v-\tilde{v}^{X,\beta}&\in C(\left[ 0,T \right]; H^2(\Omega)), \qquad \quad \ \mathbf{u}-\tilde{\mathbf{u}}^{X,\beta}\in C(\left[ 0,T \right]; H^2(\Omega)), \\   
\nabla_x(v-\tilde{v}^{X,\beta})&\in L^2(\left[ 0,T \right]; H^1(\Omega)), \quad \nabla_x(\mathbf{u}-\tilde{\mathbf{u}}^{X,\beta})\in L^2(\left[ 0,T \right]; H^2(\Omega)), \\
\partial_{t}(v-\tilde{v}^{X,\beta})&\in L^2(\left[0, T\right]; H^1(\Omega)), \quad \
\partial_{t}(\mathbf{u}-\tilde{\mathbf{u}}^{X,\beta})\in L^2(\left[0, T\right]; H^1(\Omega)),
\end{aligned}    
\end{equation*}
with
\begin{equation} \label{perturbation_small}
\lVert v-\tilde{v}^{X,\beta} \rVert_{L^\infty(0,T;H^2(\Omega))}+\lVert \mathbf{u}-\tilde{\mathbf{u}}^{X,\beta} \rVert_{L^\infty(0,T;H^2(\Omega))}\le \varepsilon,
\end{equation}
\end{proposition}
then we have
\begin{equation}
\begin{aligned} \label{C-3}
&\sup\limits_{t\in \left[0, T\right]}\lVert (v-\tilde{v}^{X,\beta})(t, \cdot) \rVert_{H^2(\Omega)}+\sup\limits_{t\in \left[0, T\right]}\lVert (\mathbf{u}-\tilde{\mathbf{u}}^{X,\beta})(t, \cdot) \rVert_{H^2(\Omega)}+\sqrt{\delta\int_{0}^{T}|\dot{X}(t)|^2\, dt} \\
&\quad \ \ \ +\sqrt{\int_{0}^{T}(\mathcal{G}^S(U)+\mathcal{D}(U)+\mathcal{D}_1(U)+\mathcal{D}_2(U))+\mathcal{D}_3(U) + \mathcal D_4(U))\, dt}\\
&\quad \le C_0(\lVert v_0-\tilde{v}(0,\cdot) \rVert_{H^2(\Omega)}+\lVert \mathbf{u}_0-\tilde{\mathbf{u}}(0,\cdot) \rVert_{H^2(\Omega)})+C_0e^{-C\delta\beta},
\end{aligned}
\end{equation}
where the constant $C_0$ is independent of $T$.

In particular, for all $t\in \left[0, T\right]$,
\begin{equation} \label{bddx12}
|\dot{X}(t)|\le C_0\lVert (v-\tilde{v}^{X,\beta})(t,\cdot) \rVert_{L^\infty(\Omega)}. 
\end{equation}

In addition,
\begin{equation}
\begin{aligned}
\mathcal{G}^S(U)&:=\int_{\mathbb{T}^2} \int_{\mathbb{R}_+} |\tilde{v}_{x_1}^{X,\beta}| |(v-\tilde{v}^{X,\beta})|^2 \, dx_1 \, dx', \\
\mathcal{D}(U)&:=\int_{\mathbb{T}^2} \int_{\mathbb{R}_+} |\nabla_x(p(v)-p(\tilde{v}^{X,\beta}))|^2 \, dx_1 \, dx' , \\
\mathcal{D}_1(U)&:=\int_{\mathbb{T}^2} \int_{\mathbb{R}_+} |\nabla_x^2(v-\tilde v^{X,\beta})|^2 \, dx_1 \, dx',  \\\mathcal{D}_2(U)&:=\int_{\mathbb{T}^2} \int_{\mathbb{R}_+} \big(|\nabla_x(\mathbf{u}-\tilde{\mathbf{u}}^{X,\beta})|^2 + |\nabla^2_x(\mathbf{u}-\tilde{\mathbf{u}}^{X,\beta})|^2 + |\nabla^3_x(\mathbf{u}-\tilde{\mathbf{u}}^{X,\beta})|^2 \big) \, dx_1 \, dx',  \\
\mathcal{D}_3(U)&:=\int_{\mathbb{T}^2} \int_{\mathbb{R}_+} \big(|\pa_t(\mathbf{u}-\tilde{\mathbf{u}}^{X,\beta})|^2 +|\pa_t\nabla_{x}(\mathbf{u}-\tilde{\mathbf{u}}^{X,\beta})|^2 \big) \, dx_1 \, dx', \\
\mathcal{D}_4(U)&::=\int_{\mathbb{T}^2} \int_{\mathbb{R}_+} \big(|\pa_t(v-\tilde{v}^{X,\beta})|^2 +|\pa_t\nabla_{x}(v-\tilde{v}^{X,\beta})|^2 \big) \, dx_1 \, dx'.
\end{aligned}
\end{equation}

\subsection{Global-in-time existence and long-time behavior}
Using Propositions 3.1 and 3.2, we can apply a classical continuation argument to prove the global-in-time existence of perturbations in $\eqref{thmstate1}$. We can also utilize Proposition 3.2 to prove the long-time behavior in $\eqref{thmstate2}$. Since these arguments are standard, we omit the detailed proof of Theorem $\ref{thmv,ucoor}$ and refer the reader to \cite{KVW23, HKK23}.

Hence, the remaining part of the paper is dedicated to the proof of Proposition 3.2.

\subsection{Estimate on shifts} 
Before closing this section, we prove the estimate \eqref{bddx12} in advance. \\
Using the assumption \eqref{perturbation_small}, Sobolev inequality, and \eqref{est-rel-2}, we have
\beq\label{smp1}
\|p(v)-p(\tilde{v}^{X,\beta})\|_{L^\infty((0,T)\times\Omega)}\le C\|v-\tv^{X,\beta}\|_{L^\infty((0,T)\times\Omega)} \le C\varepsilon.
\eeq
Thus, applying \eqref{smp1} and \eqref{viscous-shock-h} to \eqref{X(t)}, we have
\begin{align}
\begin{aligned}\label{dxbound}
|\dot{X}(t)| & \le \frac{C}{\delta}  \left(\|p(v)-p(\tv^{X,\beta}) \|_{L^\infty(\Omega)}+\|v-\tv^{X,\beta} \|_{L^\infty(\Omega)}\right) \int_{\mathbb{R}_+} |\tv_{x_1}^{X,\beta} | \,  dx_1  \\ 
& \le C\|v-\tv^{X,\beta}\|_{L^\infty(\Omega)},\quad t\le T.    
\end{aligned}   
\end{align}

Therefore, \eqref{bddx12} holds. 

\subsubsection{Remark}
Along with \eqref{bddx12}, we will further derive an estimate for $\ddot{X}(t)$ in Section 6 to control the other higher-order terms. For all $t\in [0, T]$, the following holds:
\[|\ddot{X}(t)|^2\le C(|\dot{X}(t)|^2+\delta\|\phi_t\|^2 +\delta^3\mathcal{G}^s).\]

\subsection{Notations} Hereafter, we use the following notations for simplicity. \\
1. $C$ denotes a generic positive $O(1)$-constant that may change from line to line, but is independent of the small parameters $\delta_0, \eps_1, \delta,$ and the time $T$.\\
2. Without danger of confusion, we drop the explicit dependence in the shifts $X$ and $\beta$ of the viscous 2-shock wave \eqref{composite_wave} and the weight function $a$ as follows:
\[
(\tv, \tu) (t,x) := (\tv^{X, \beta}, \tu^{X, \beta}) (t,x) \qquad and \qquad a(t,x):=a^{X,\beta}(t, x).
\]
3. For any scalar function $f$, 
\begin{equation*}
\nabla_xf:=(\partial_{x_1}f, \partial_{x_2}f, \partial_{x_3}f) \qquad and \qquad \nabla_{x'}f:=( \partial_{x_2}f, \partial_{x_3}f).  
\end{equation*}
4. For any three-dimensional vector function $\mathbf{F}=(F_1, F_2, F_3)$, 
\begin{equation*}
\operatorname{div}_x\mathbf{F}:=\partial_{x_1}F_1+\partial_{x_2}F_2+\partial_{x_3}F_3 \qquad and \qquad \operatorname{div}_{x'}\mathbf{F}:=\partial_{x_2}F_2+\partial_{x_3}F_3  
\end{equation*}
5. By Fubini's theorem, the order of integration does not matter. Nevertheless, it is often convenient to integrate successively with respect to $x_1$ and then $x'$. Thus, abusing the notation, we write:
\begin{equation*}
\int_{\Omega} f \, dx:=\int_{\mathbb{T}^2}\int_{\mathbb{R}_+} f \, dx_1dx'.
\end{equation*}
6. Unless otherwise specified, the integration domain for the norms is assumed to be $\Omega$. For example, $\| \cdot \|_{L^3}$ stands for the $L^3(\Omega)$ norm. In addition, the norm without a subscript, $\| \cdot \|$, denotes the $L^2(\Omega)$ norm, that is,
\[\|\cdot\|:=\|\cdot\|_{L^2(\Omega)}.\]

\section{Zeroth order estimates}\label{sec:4}
\setcounter{equation}{0}

To deal with the viscous shock efficiently, we bring in the effective velocity $\mathbf{h}:=\mathbf{u}-(2\mu+\lambda)\nabla_xv$. Thus, we rewrite the system \eqref{changed main} in terms of the $(v, h)$-coordinate as in \cite{lee4923620long}:

\begin{equation} \label{4.1}
\begin{cases}
\rho\left(\partial_tv+\mathbf{u}\cdot\nabla_xv\right)-\operatorname{div}_x\mathbf{h}=(2\mu+\lambda)\Delta_xv, \\
\rho\left(\partial_t\mathbf{h}+\mathbf{u}\nabla_x\mathbf{h}\right)+\nabla_xp(v)=R, \\
\end{cases}
\end{equation}
with the initial datum
\beq \label{ninini}
(v, \mathbf{h})|_{t=0}=(v_0, \mathbf{h}_0) \to (v_{+}, \mathbf{u}_{+}), \quad as \quad x_1 \to +\infty , 
\eeq
where
\begin{equation} \label{presentR}
R=\frac{2\mu+\lambda}{v}(\nabla_x\mathbf{u} \nabla_xv - \operatorname{div}_x\mathbf{u}\nabla_xv)-\mu \nabla_x \times \left(\nabla_x \times \mathbf{u}\right).
\end{equation}

We can also rewrite the viscous shock equations \eqref{changed viscous shock equation}-\eqref{change shock equation 2} as  
\begin{equation}\label{viscous-shock-h}
\begin{cases}
-\sigma^*\tilde{v}'-\tilde{h}'=(2\mu+\lambda)\tilde{v}'', \qquad ':=\frac{d}{d\xi},\\
-\sigma^* \tilde{h}' +p(\tilde{v})'=0,\\
(\tv,\tilde{h})(-\infty) = (v_-,u_{-}),\quad (\tv,\tilde{h})(+\infty) = (v_+,u_{+}),
\end{cases}
\end{equation} 
where $\tilde{h}:=\tilde{u}-(2\mu+\lambda)\tilde{v}'$.

In terms of notation, we denote
\[\tilde{\mathbf{h}}:=(\tilde{h}, 0, 0).\]

Finally, from \eqref{4.1} and \eqref{viscous-shock-h} we have the perturbation equations:
\begin{equation} \label{mixed equations}
\begin{cases}
\rho(v-\tilde{v})_t+\rho \mathbf{u}\cdot \nabla_x(v-\tilde{v})-\operatorname{div}_x(\mathbf{h}-\tilde{\mathbf{h}})-\rho \dot{X}(t)\tilde{v}'+F\tilde{v}'=(2\mu+\lambda)\Delta_x(v-\tilde{v}), \\
\rho(\mathbf{h}-\tilde{\mathbf{h}})_t+\rho \mathbf{u} \nabla_x(\mathbf{h}-\tilde{\mathbf{h}})+\nabla_x(p(v)-p(\tilde{v}))-\rho \dot{X}(t)\tilde{\mathbf{h}}_{x_1}+F\tilde{\mathbf{h}}_{x_1}=R,
\end{cases}
\end{equation}
where 
\begin{align}
\begin{aligned} \label{presentF}
F&=-\sigma(\rho-\tilde{\rho})+\rho u_1 - \tilde{\rho} \tilde{u}\\
&=-\frac{\sigma^*}{\tilde{\rho}}(\rho-\tilde{\rho})+\rho(u_1-\tilde{u}).\\
\end{aligned}
\end{align}

Define the relative entropy $\eta$ as
\begin{equation*}
\eta(U|\tilde{U}):=Q(v|\tilde{v})+\frac{|\mathbf{h}-\tilde{\mathbf{h}}|^2}{2}.    
\end{equation*}

In this section, we prove the following lemma, which plays a crucial role in this paper. 

\begin{lemma} \label{lemma5.1}
There exist positive constants $C, \nu$, and $C_\nu$ such that for any $t\in\left[0, T\right]$,
\begin{align}
\begin{aligned} \label{lemma4.1eq}
&\int_{\Omega}\eta(U|\tilde{U})(t,x) \, dx+\int_{0}^{t} \left( \delta|\dot{X}(\tau)|^2+\mathcal{G}_1(\tau)+\mathcal{G}_3(\tau)+\mathcal{G}^S(\tau)+\mathcal{D}(\tau) \right) \, d\tau \\
\le & C\int_{\Omega} \eta(U|\tilde{U})(0,x) \,dx\\
&+(\nu+C\delta+C\varepsilon)\int_{0}^{t}||\nabla_x(\mathbf{u}-\tilde{\mathbf{u}})||^2_{H^1}\, d\tau+Ce^{-C\delta \beta}+C_\nu\int_{0}^{t}\int_{\mathbb{T}^2} |\nabla_{x'}(v-\tilde{v})|^2 \bigg|_{x_1=0} \, dx'\,d\tau,
\end{aligned}
\end{align}
where $\nu$ is chosen sufficiently small, $C_\nu$ depends only on $\nu$, and 
\begin{equation*}
\begin{aligned}
\mathcal{G}_1&:=\int_{\Omega} |a_{x_1}| \left| h_1-\tilde{h}-\frac{p(v)-p(\tilde{v})}{\sigma^*} \right|^2\, dx, \\ 
\mathcal{G}_3&:=\int_{\Omega} |a_{x_1}| (h_2^2+h_3^2)\, dx,\\
\mathcal{G}^S&:= \int_{\Omega} |\tilde{v}_{x_1}| |p(v)-p(\tilde{v}) |^2 \,dx, \\
\mathcal{D}&:=\int_{\Omega} |\nabla_x(p(v)-p(\tilde{v}))|^2\, dx.
\end{aligned}    
\end{equation*}
\end{lemma}

\subsection{Computation of the weighted relative entropy} 

\

To obtain the desired result, we first compute the weighted relative entropy and initially categorize the terms. The proof of the following lemma is essentially the same as those in \cite{lee4923620long, wang2022nonlinear}, except for the additional boundary terms that appear here. Thus, we omit the details.

\begin{lemma}\label{lem:rel-ent}
It holds
\begin{align}
	\begin{aligned}\label{est-rel-ent}
	&\frac{d}{dt}\int_{\Omega} a \rho \eta(U|\tilde{U}) \,dx=\dot{X}(t)Y(U)+\mathcal{J}^{\rm bad}(U)-\mathcal{J}^{\rm good}(U)+\mathcal{P}(U),
	\end{aligned}
\end{align}
where
\begin{align}
\begin{aligned} \label{before diffusionterm}
\dot{X}(t)Y(U)&=-\int_{\Omega}\rho Q(v|\tilde{v})\dot{X}(t)a_{x_1}dx-\int_{\Omega}a\rho p'(\tilde{v})(v-\tilde{v})\dot{X}(t)\tilde{v}_{x_1}dx, \\ 
\mathcal{J}^{\rm bad}(U)&=\int_{\Omega}Q(v|\tilde{v}) F a_{x_1}dx-\int_{\Omega}ap(v|\tilde{v}) F \tilde{v}_{x_1}dx\\
& \quad +\int_{\Omega}ap(v|\tilde{v}) \sigma^* \tilde{v}_{x_1}dx+\int_{\Omega}a(p(v)-p(\tilde{v})) F\tilde{v}_{x_1}dx \\ 
& \quad -(2\mu+\lambda)\int_{\Omega}a\partial_{x_1}(p(v)-p(\tilde{v})) \partial_{x_1}p(\tilde{v})\left( \frac{1}{\gamma p(v)^{1+\frac{1}{\gamma}}}- \frac{1}{\gamma p(\tilde{v})^{{1+\frac{1}{\gamma}}}}\right)dx\\
& \quad -(2\mu+\lambda)\int_{\Omega}a_{x_1}(p(v)-p(\tilde{v}))\frac{\partial_{x_1}(p(v)-p(\tilde{v}))}{\gamma p(v)^{1+\frac{1}{\gamma}}} dx\\ 
& \quad -(2\mu+\lambda)\int_{\Omega}a_{x_1}(p(v)-p(\tilde{v}))\partial_{x_1}p(\tilde{v})\left( \frac{1}{\gamma p(v)^{1+\frac{1}{\gamma}}}- \frac{1}{\gamma p(\tilde{v})^{{1+\frac{1}{\gamma}}}}\right) dx, \\ 
&\quad +\int_{\Omega}a(\mathbf{h}-\tilde{\mathbf{h}})\cdot R \, dx \\
\mathcal{J}^{\rm good}(U)&=\int_{\Omega} \sigma^* a_{x_1} Q(v|\tilde{v})  dx+(2\mu+\lambda)\int_{\Omega}a\frac{|\nabla_x(p(v)-p(\tilde{v}))|^2}{\gamma p(v)^{1+\frac{1}{\gamma}}}dx,\\
\mathcal{P}(U)&=\int_{\mathbb{T}^2} a\left(\rho u_1 \left[Q(v|\tilde{v}) + \frac{|\mathbf{h}-\tilde{\mathbf{h}}|^2}{2}\right]+(p(v)-p(\tilde{v})(u_1-\tilde{u})\right)\Bigg|_{x_1=0} dx'. 
\end{aligned}
\end{align}
\end{lemma}

\subsection{Maximization and Decomposition}

\

In this subsection, we intrinsically decompose certain terms in Lemma \ref{lem:rel-ent}. In particular, we employ a maximization argument to obtain a quadratic structure.

In $\mathcal{J}^{\rm bad}(U)$, the most important term is
\[\int_{\Omega}a_{x_1}(p(v)-p(\tilde{v}))(h_1-\tilde{h})\,dx.\] 
In this term, $p(v)-p(\tilde{v})$ and $(h_1-\tilde{h})$ are mixed. It is indispensible that $(h_1-\tilde{h})$ is separated from $(p(v)-p(\tilde{v}))$ as
\begin{align*}
	a_{x_1}&(p(v)-p(\tilde{v}))(h_1-\tilde{h})-\frac{\sigma^*}{2}a_{x_1}|\mathbf{h}-\tilde{\mathbf{h}}|^2\\
	&=-\frac{\sigma^*a_{x_1}}{2}\left|h_1-\tilde{h}-\frac{p(v)-p(\tilde{v})}{\sigma^*}\right|^2+\frac{a_{x_1}}{2\sigma^*}|p(v)-p(\tilde{v})|^2-\sigma^* a_{x_1} \left( \frac{h_2^2+h_3^2}{2} \right),
\end{align*}
we rewrite the terms $\mathcal{J}^{\rm bad}(U)-\mathcal{J}^{\rm good}(U)$ in the above lemma \eqref{est-rel-ent} as $\mathcal{J}^{\text{bad}}(U) - \mathcal{J}^{\text{good}}(U) = \mathcal{B}(U)-\mathcal{G}(U)$ and name each term of $\mathcal{B}(U)$, $\mathcal{G}(U)$, and $\mathcal{P}(U)$ as follows:

\begin{align*}
	&\mathcal{B}(U) =\sum_{i=1}^8 \mathcal{B}_i(U),\\
	&\mathcal{G}(U) = \mathcal{G}_1(U) + \mathcal{G}_2(U) + \mathcal{G}_3(U)  + \mathcal{D}(U), \\
    &\mathcal{P}(U) = \mathcal{P}_1(U)+\mathcal{P}_2(U)+\mathcal{P}_3(U),
\end{align*}
where
\begin{align*}
 & \mathcal{B}_1(U):=\frac{1}{2\sigma^*} \int_{\Omega}a_{x_1} |p(v)-p(\tilde{v})|^2 \, dx\\ 
 & \mathcal{B}_2(U):=\sigma^* \int_{\Omega}  ap(v|\tilde{v}) \tilde{v}_{x_1} \, dx\\
 & \mathcal{B}_3(U):=\int_{\Omega} aF\left(p'(\tilde{v})(v-\tilde{v}) \tilde{v}_{x_1}-a(h_1-\tilde{h})\tilde{h}_{x_1}\right) \, dx \\
 & \mathcal{B}_4(U):= \int_{\Omega}    a_{x_1}F\left( Q(v|\tilde{v})+\frac{|\mathbf{h}-\tilde{\mathbf{h}}|^2}{2} \right)\, dx \\
 & \mathcal{B}_5(U):=-(2\mu+\lambda)  \int_{\Omega} a\partial_{x_1}(p(v)-p(\tilde{v})) \partial_{x_1}p(\tilde{v})\left( \frac{1}{\gamma p(v)^{1+\frac{1}{\gamma}}}- \frac{1}{\gamma p(\tilde{v})^{{1+\frac{1}{\gamma}}}}\right) \, dx \\ 
 & \mathcal{B}_6(U):=-(2\mu+\lambda)   \int_{\Omega} a_{x_1} (p(v)-p(\tilde{v}))\frac{\partial_{x_1}(p(v)-p(\tilde{v}))}{\gamma p(v)^{1+\frac{1}{\gamma}}} \, dx \\ 
 & \mathcal{B}_7(U):=-(2\mu+\lambda) \int_{\Omega} a_{x_1}(p(v)-p(\tilde{v}))\partial_{x_1}p(\tilde{v})\left( \frac{1}{\gamma p(v)^{1+\frac{1}{\gamma}}}- \frac{1}{\gamma p(\tilde{v})^{{1+\frac{1}{\gamma}}}}\right)  \, dx \\ 
 & \mathcal{B}_8(U):= \int_{\Omega}a(\mathbf{h}-\tilde{\mathbf{h}})\cdot R\, dx, \quad \quad (R \ was \ defined \ in \ \eqref{presentR})
\end{align*}
\begin{align*}
& \mathcal{G}_1(U):=\frac{\sigma^*}{2} \int_{\Omega} a_{x_1} \left| h_1-\tilde{h}-\frac{p(v)-p(\tilde{v})}{\sigma^*} \right|^2\, dx, \\
&\mathcal{G}_2(U):=\sigma^* \int_{\Omega} a_{x_1}Q(v|\tilde{v}) \, dx \\
& \mathcal{G}_3(U):= \frac{\sigma^*}{2} \int_{\Omega} a_{x_1} (h_2^2+h_3^2)\, dx, \\
&\mathcal{D}(U):=(2\mu+\lambda) \int_{\Omega} a\frac{|\nabla_x(p(v)-p(\tilde{v}))|^2}{\gamma p(v)^{1+\frac{1}{\gamma}}}\, dx,
\end{align*}
and
\begin{align*}
& \mathcal{P}_1(U):=\int_{\mathbb{T}^2} a\rho u_1 Q(v|\tilde{v})  \Bigg|_{x_1=0} \, dx', \\
&\mathcal{P}_2(U):=\int_{\mathbb{T}^2} a\rho u_1 \frac{|\mathbf{h}-\tilde{\mathbf{h}}|^2}{2} \Bigg|_{x_1=0} \, dx',  \\
& \mathcal{P}_3(U):= \int_{\mathbb{T}^2} a(p(v)-p(\tilde{v}))(u_1-\tilde{u}) \Bigg|_{x_1=0} \, dx'. \\
\end{align*}

In short, we estimate the right-hand side of the below equation: 
\begin{equation}\label{est-1}
	\frac{d}{dt}\int_{\Omega}a\rho\eta(U|\tilde{U})\,dx = \dot{X}(t)Y(U) + \mathcal{B}(U)-\mathcal{G}(U)+\mathcal{P}(U).
\end{equation}
\begin{remark}
Since $\sigma^{*},\  a_{x_1}>0$, and $a\ge1$, $\mathcal{G}(U)$ consists of four terms with a positive sign. On the other hand, $\mathcal{B}(U)$ consists of bad terms which we have to control.
\end{remark}

For $Y(U)$ in \eqref{before diffusionterm}, we initially write it more explicitly as follows:
\begin{align*}
	Y(U)
	& = -\int_{\Omega} \rho a_{x_1} \left( Q(v|\tilde{v})+\frac{|\mathbf{h}-\tilde{\mathbf{h}}|^2}{2} \right)  \, dx  +\int_{\Omega} a\rho \tilde{h}_{x_1}(h_1-\tilde{h}) \, dx\\
	&\quad -\int_{\Omega} a\rho p'(\tilde{v})\tilde{v}_{x_1}(v-\tilde{v}) \, dx.
\end{align*}
We decompose $Y_i$ as below:
\[Y = \sum_{i=1}^5 Y_{i},\]
where
\begin{align*}
	&Y_{1} :=\int_{\Omega}  \frac{a}{\sigma^*}\rho\tilde{h}_{x_1} (p(v)-p(\tilde{v}))  \, dx,\\
	&Y_{2} := - \int_{\Omega} a\rho p(\tilde{v})_{x_1}(v-\tilde{v})  \, dx,\\
	&Y_{3} := \int_{\Omega} a\rho\tilde{h}_{x_1} \left( h_1-\tilde{h}-\frac{p(v)-p(\tilde{v})}{\sigma^*} \right) \, dx\\
	&Y_{4} :=-\frac{1}{2} \int_{\Omega} \rho a_{x_1} \left( h_1-\tilde{h} - \frac{p(v)-p(\tilde{v})}{\sigma^*} \right)\left( h_1-\tilde{h} + \frac{p(v)-p(\tilde{v})}{\sigma^*} \right) \, dx\\
	&Y_{5} :=-\int_{\Omega} \rho a_{x_1} \left(Q(v|\tilde{v})+\frac{|p(v)-p(\tilde{v})|^2}{2(\sigma^*)^2}+\frac{h_2^2+h_3^2}{2} \right) \, dx.
\end{align*}
From construction of the shift $X$ in \eqref{X(t)}, we have
\[\dot{X} = -\frac{M}{\delta}(Y_{1}+Y_{2}),\]
which implies 
\begin{equation}\label{XiYi}
	\dot{X}Y(U) = -\frac{\delta}{M}|\dot{X}|^2+\dot{X}\sum_{i=3}^5Y_{i}.
\end{equation}
Combining \eqref{est-1} and \eqref{XiYi}, we get
\begin{align}
	\begin{aligned}\label{est}
	&\frac{d}{dt} \int_{\Omega} a\rho \eta (U| \tU) \,dx =\mathcal{R}, \quad\mbox{where}\\
	&\mathcal{R} :=-\frac{\delta}{M}|\dot{X}|^2 + \left(\dot{X}\sum_{i=3}^5Y_{i}\right) +\sum_{i=1}^8\mathcal{B}_i -\mathcal{G}_1-\mathcal{G}_2-\mathcal{G}_3-\mathcal{D}+\mathcal{P}. \\
	&\quad= \underbrace{-\frac{\delta}{2M}|\dot{X}|^2 + \mathcal{B}_1 +\mathcal{B}_2 -\mathcal{G}_2-\frac{3}{4}\mathcal{D}}_{=:\mathcal{R}_1}+\mathcal{P}\\
	&\quad\quad \underbrace{-\frac{\delta}{2M}|\dot{X}|^2 + \left(\dot{X}\sum_{i=3}^5Y_{i}\right) +\sum_{i=3}^8\mathcal{B}_i -\mathcal{G}_1-\mathcal{G}_3-\frac{1}{4}\mathcal{D}}_{=:\mathcal{R}_2}.
	\end{aligned}
\end{align}

We first concentrate on the estimate of $\mathcal{R}_1$.

\subsection{Estimate of the main part $\mathcal{R}_{1}$}\label{sec:4.5} 

\

Since $X(t)$ is bounded on $\left[0, T\right]$ by \eqref{dxbound}, for any fixed $t\ge0$, we consider the following substitution in space: 
\[
y:=1-\frac{p(\tilde{v}(x_1-\sigma t-X(t)-\beta))-p(v_+)}{\delta}.
\]
Obviously, $y :\bbr\to (0,1)$ is a strictly increasing function such that
\[\frac{d y}{dx_1} = -\frac{1}{\delta}p'(\tilde{v})\tilde{v}' >0,\]
and
\[\lim_{x_1\to-\infty}y = 0,\quad \lim_{x_1\to+\infty} y = 1,  \qquad and \qquad y_0:=y(0),\]
where we choose $\beta\gg1$ and so $0<y_0<\frac{1}{10}$.

In addition, to use Poincar\'e-type inequality, we adopt new variable $w$ as: for fixed any $t>0$,
 \[
 \begin{aligned}
&w(x):=p(v(t,x)) - p\left(\tv(x_1-\sigma t-X(t)-\beta)\right).
\end{aligned}
\]

For simplicity, we use the following notations to denote constants of $O(1)$-scale:
\begin{equation}\label{sigmaalpha}
\sigma_+:=\sqrt{-p'(v_+)},\quad \alpha_+:=\frac{\gamma+1}{2\gamma\sigma_+p(v_+)}=\frac{p''(v_+)}{2\sigma_+|p'(v_+)|^2}.    
\end{equation}

In addition, the following estimates on the $O(1)$-constants hold:
\begin{equation}\label{shock_speed_est}
	|\sigma^*- \sigma_+| \leq C \delta,
\end{equation}
and
\begin{equation}\label{shock_speed_est-2}
	\|\sigma_+^2+p'(\tv)\|_{L^\infty} \le C\delta, \quad \left\|\frac{1}{\sigma_+^2}-\frac{p(\tv)^{-\frac{1}{\gamma}-1}}{\gamma}\right\|_{L^\infty}\le C\delta.
\end{equation}

In this subsection, our goal is to get the below Lemma:
\begin{lemma} \label{mainlemma}
There exist the positive generic constants $C_1, C>0$ such that
\begin{align*} 
	&-\frac{\delta}{2M}|\dot{X}|^2+\mathcal{B}_1+\mathcal{B}_2-\mathcal{G}_{2}-\frac{3}{4}\mathcal{D}\\
	\le & -C_1\int_{\Omega} |\tilde{v}_{x_1}||p(v)-p(\tilde{v})|^2 dx+C\int_{\Omega}a_{x_1}|p(v)-p(\tilde{v})|^3\,dx
\end{align*}    
for all $t\in \left[0, T\right].$
\end{lemma}

\noindent $\bullet$  {\bf(Estimate of shift part $\frac{\delta}{2M}|\dot{X}|^2$):}

Using \eqref{viscous-shock-h} and change of variable, we have
	\begin{align*}
		Y_{1} 
		&=\int_{\Omega} \frac{a}{({\sigma^*})^2v}p(\tilde{v})_{x_1}(p(v)-p(\tilde{v}))\, dx =-\frac{\delta}{(\sigma^*)
    ^2}\int_{\mathbb{T}^2}  \int_{y_0}^{1} \frac{a}{v} w  \, dy \, dx'.
	\end{align*}
Using \eqref{shock_speed_est} and $\|a-1\|_{L^\infty}\le \sqrt{\delta}$, we have
	\[\left| Y_{1}+\frac{\delta}{\sigma_+^2 v_+} \int_{\mathbb{T}^2}  \int_{y_0}^{1} w \, dy \, dx' \right|\le C\delta(\delta+\sqrt{\delta})\int_{\mathbb{T}^2}  \int_{y_0}^{1} |w| \, dy \, dx' .\]
 When we estimate $Y_{2}$, we first use Taylor expansion in terms of $v=p(v)^{-1/\gamma}$ to get
\[\left|v-\tv -\left(-\frac{p(\tilde{v})^{-\frac{1}{\gamma}-1}}{\gamma}(p(v)-p(\tilde{v}))\right)\right|\le C|p(v)-p(\tilde{v})|^2,\]
	which together with the estimates \eqref{shock_speed_est-2} and \eqref{smp1} implies
	\[\left|v-\tv -\left(-\frac{1}{\sigma_+^2}(p(v)-p(\tilde{v}))\right)\right|\le C(\delta+\varepsilon)|p(v)-p(\tilde{v})|.\]	
As above,
\begin{align*}
    \left| Y_{2}-\frac{\delta}{\sigma_+^2v_+}\int_{\mathbb{T}^2}  \int_{y_0}^{1} w  \, dy \, dx' \right| \le&  C\delta(\delta+\sqrt{\delta}+\varepsilon)\int_{\mathbb{T}^2}  \int_{y_0}^{1} |w|  \, dy \, dx'.\\
\end{align*}
	Combining the above estimates, we have
	\begin{align*}
		\left| \dot{X}+\frac{2M}{\sigma_+^2v_+} \int_{\mathbb{T}^2}\int_{y_0}^{1} w \, dy \, dx' \right|&\le \frac{M}{\delta}	\left( \left| Y_{1}+\frac{\delta}{\sigma_+^2 v_+} \int_{\mathbb{T}^2}  \int_{y_0}^{1} w \, dy \, dx' \right| + 	\left| Y_{2}-\frac{\delta}{\sigma_+^2v_+}\int_{\mathbb{T}^2}  \int_{y_0}^{1} w  \, dy \, dx' \right| \right)\\
		&\le C(\sqrt{\delta}+\varepsilon)\int_{\mathbb{T}^2}  \int_{y_0}^{1} |w| \, dy \, dx', \\
	\end{align*}
	which implies by squaring both sides and using Young's inequality, 
 \begin{align}
 \begin{aligned} \label{est-X1}
\frac{2M^2}{\sigma_+^4v_+^2} \left( \int_{\mathbb{T}^2}\int_{y_0}^{1} w \, dy \, dx' \right)^2 -|\dot{X}|^2 \le & \ C(\delta+\varepsilon^2) \int_{\mathbb{T}^2}  \int_{y_0}^{1} |w|^2 \, dy \, dx'.\\ 
 \end{aligned}
 \end{align}
In short, combining the above estimate, we have following estimate on $\dot{X}$:
	\begin{equation} \label{est|x'|^2}
		-\frac{\delta}{2M}|\dot{X}|^2 \le -\frac{M\delta}{\sigma_+^4v_+^2} \left( \int_{\mathbb{T}^2}\int_{y_0}^{1} w \, dy \, dx' \right)^2 +  C\delta(\delta+\varepsilon^2) \int_{\mathbb{T}^2}  \int_{y_0}^{1} |w|^2 \, dy \, dx'.  
	\end{equation}
	
\noindent $\bullet$ {\bf Estimate of the bad term $\mathcal{B}_1$ and good term $\mathcal{G}_2$:}
Remind that
\begin{align*}
	&\mathcal{B}_1(U) := \frac{1}{2\sigma^*}\int_{\Omega}a_{x_1} |p(v)-p(\tilde{v})|^2\, dx,\\
	&\mathcal{G}_2(U):=  \sigma^*\int_{\Omega}a_{x_1}Q(v|\tilde{v})\,dx.
\end{align*}

We first apply Lemma \ref{lem-rel-quant} to obtain
	\begin{align*}
		\mathcal{G}_{2}&\ge \sigma^*\int_{\Omega}a_{x_1} \frac{p(\tilde{v})^{-\frac{1}{\gamma}-1}}{2\gamma}|p(v)-p(\tilde{v})|^2\, dx\\
		&\quad -\sigma^* \int_{\Omega}a_{x_1}\frac{(\gamma+1)}{3\gamma^2}p(\tilde{v})^{-\frac{1}{\gamma}-2}|p(v)-p(\tilde{v})|^3\,dx. \\
	\end{align*}
Let $\widehat{\mathcal{G}}_{2}$ be the good term defined as
	\[\widehat{\mathcal{G}}_{2}:=\sigma_+ \int_{\Omega}a_{x_1} \frac{p(\tilde{v})^{-\frac{1}{\gamma}-1}}{2\gamma}|p(v)-p(\tilde{v})|^2 \, dx.
	\]
Using \eqref{shock_speed_est} and \eqref{shock_speed_est-2}, we have
\begin{align*}
\mathcal{B}_{1} \leq \frac{1}{2\sigma_+} \int_{\Omega}a_{x_1} |p(v)-p(\tilde{v})|^2 \, dx+ \frac{C \delta}{2\sigma_+} \int_{\Omega}  a_{x_1} |p(v)-p(\tilde{v})|^2 \, dx,
	\end{align*}
	and
	\begin{align*}
		\widehat{\mathcal{G}}_{2} \geq \frac{1}{2\sigma_+}(1-C \delta) \int_{\Omega}a_{x_1} |\pv-\tpv|^2 \, dx.
	\end{align*}
Then we have
	\begin{align*}
		\mathcal{B}_{1} - \widehat{\mathcal{G}}_{2}
		&\leq C\delta \int_{\Omega}a_{x_1} |p(v)-p(\tilde{v})|^2 \, dx.\\
	\end{align*}
Thus, we have the desired estimate in terms of the new variables $w, y$ using the change of variables:
\begin{equation} \label{B21G22}
 \mathcal{B}_1-\widehat{\mathcal{G}}_{2}\le C \delta^{\frac{3}{2}}\int_{\mathbb{T}^2}\int_{y_0}^1|w|^2\,dy\,dx'.   
\end{equation}
\noindent $\bullet$ {\bf Estimate of the bad term $\mathcal{B}_2$:}
Remind that
\begin{align*}
	&\mathcal{B}_2(U) := \sigma^*\int_{\Omega} a\tilde{v}_{x_1}p(v|\tilde{v})\,dx.
\end{align*}

Using Lemma \ref{lem-rel-quant}, \eqref{shock_speed_est}, \eqref{shock_speed_est-2}, and the change of variables, we have
\begin{align*}	
\mathcal{B}_2(U) &=\sigma^*\int_{\Omega}a \frac{p(\tv)_{x_1}}{p'(\tv)}p(v|\tv) \,dx\\
	&\le \frac{(\gamma+1)}{2\gamma\sigma_+p(v_+)}(1+C(\sqrt{\delta}+\varepsilon)) \int_{\Omega} |p(\tv)_{x_1}|  |(p(v)-p(\tilde{v}))|^2 \,dx \\		&\le\delta\alpha_+(1+C(\sqrt{\delta}+\varepsilon))\int_{\mathbb{T}^2}\int_{y_0}^1|w|^2\,dy \,dx'.
\end{align*}

In short, combining the above estimates with \eqref{shock_speed_est} and the smallness of $\delta$ and $\varepsilon$, we have
\begin{align}
\begin{aligned} \label{L2.3}
\mathcal{B}_1-\mathcal{G}_{2} + \mathcal{B}_2	 &\le
\frac{11\delta\alpha_+}{10}\int_{\mathbb{T}^2}\int_{y_0}^1|w|^2\,dy \, dx'\\
&\qquad +     \sigma^* \int_{\Omega}a_{x_1}\frac{(\gamma+1)}{3\gamma^2}p(\tilde{v})^{-\frac{1}{\gamma}-2}|p(v)-p(\tilde{v})|^3\,dx.\\
		\end{aligned}
	\end{align}
\noindent $\bullet$ {\bf Estimate of the diffusion term $\mathcal{D}(U)$:}
At first, we separate $\mathcal{D}(U)$ into
\begin{align*}    
\mathcal{D}(U) =& (2\mu+\lambda) \int_{\Omega} a\frac{|\partial_{x_1}(p(v)-p(\tilde{v}))|^2}{\gamma p(v)^{1+\frac{1}{\gamma}}}\, dx\\
&+(2\mu+\lambda) \int_{\Omega} a\frac{|\nabla_{x'}(p(v)-p(\tilde{v}))|^2}{\gamma p(v)^{1+\frac{1}{\gamma}}}\, dx \\
=&:\mathcal{D}_1(U)+\mathcal{D}_2(U).\\
\end{align*}

Using $a\ge1$ and the change of variables, we have
\begin{align*}
\mathcal{D}_1(U)\ge & \frac{(2\mu+\lambda)}{\gamma}\int_{\Omega}\frac{|\partial_{x_1}(p(v)-p(\tilde{v}))|^2}{p(v)^{1+\frac{1}{\gamma}}}dx\\  
 = & \frac{(2\mu+\lambda)}{\gamma} \int_{\mathbb{T}^2} \int_{y_0}^{1} \frac{1}{p(v)^{1+\frac{1}{\gamma}}} |\partial_yw|^2 \frac{dy}{dx_1} dy dx'. \\
\end{align*}

Because we want to write $\mathcal{D}_1$ in terms of the variables $y$ and $w$, we apply the following inequality in \cite{KVW23} to the right hand side:
\begin{align} \label{diffusion lemma}
\left| \frac{1}{y(1-y)} \frac{(2\mu+\lambda)}{\gamma p(\tilde{v})^{\frac{1}{\gamma}+1}}  \frac{dy}{dx_1}-\delta\alpha_+\right|\le C\delta^2,
\end{align}
where $\alpha_+$ was defined in \eqref{sigmaalpha}.

So, \eqref{shock_speed_est-2} and \eqref{diffusion lemma} yield
\begin{align}
\begin{aligned} \label{estD_1}
\mathcal{D}_1\ge & \delta \left(\alpha_+-C\delta\right)\int_{\mathbb{T}^2}\int_{y_0}^{1} y(1-y) |\partial_yw|^2  dydx' \\ 
\ge & \delta \left(\alpha_+-C\delta\right)\int_{\mathbb{T}^2}\int_{y_0}^{1} (y-y_0)(1-y) |\partial_yw|^2 dydx'. \\
\end{aligned}
\end{align}

\noindent From \eqref{diffusion lemma}, we know
\[y(1-y)\frac{dx_1}{dy}\ge \frac{2\mu+\lambda}{\gamma p(\tilde{v})^{1+\frac{1}{\gamma}}\delta(\alpha_++C\delta)}\ge \frac{2\mu+\lambda}{2\alpha_+\delta|p'(v_+)|}.\]
This implies 
\begin{align} 
\begin{aligned}\label{estD_2}
\mathcal{D}_2\ge & \frac{(2\mu+\lambda)}{\gamma}\int_{\mathbb{T}^2}\int_{y_0}^1\frac{|\nabla_{x'}w|^2}{p(v)^{1+\frac{1}{\gamma}}}\left(\frac{dx_1}{dy}\right)dydx' \\
\ge & \frac{(2\mu+\lambda)^2}{\gamma} \int_{\mathbb{T}^2}\int_{y_0}^{1} \frac{|\nabla_{x'}w|^2}{p(v)^{1+\frac{1}{\gamma}}y(1-y)}\cdot\frac{1}{2\alpha_+\delta|p'(v_+)|
} dydx'\\
\ge & (1-C(\delta+\varepsilon))\frac{\sigma_+(2\mu+\lambda)^2}{|p''(v_+)|\delta}\int_{\mathbb{T}^2}\int_{y_0}^{1}\frac{|\nabla_{x'}w|^2}{y(1-y)}dydx', 
\end{aligned}
\end{align}
where we used \eqref{sigmaalpha}.

Thus, combining \eqref{estD_1} and \eqref{estD_2} with the smallness of $\delta$ and $\varepsilon$, we have
\begin{align}
\begin{aligned}\label{L2.4'}
-\mathcal{D}\le& -\frac{9\delta \alpha_+}{10}\int_{\mathbb{T}^2}\int_{y_0}^{1} (y-y_0)(1-y) |\partial_yw|^2 dydx'\\
& -\frac{9}{10}\frac{\sigma_+(2\mu+\lambda)^2}{|p''(v_+)|\delta}\int_{\mathbb{T}^2}\int_{y_0}^{1}\frac{|\nabla_{x'}w|^2}{y(1-y)}dydx'.
\end{aligned}    
\end{align}
\noindent $\bullet$ {\bf Conclusion:}
Combining the estimates \eqref{est|x'|^2}, \eqref{L2.3}, and \eqref{L2.4'} we have
\begin{align*}
	&-\frac{\delta}{2M}|\dot{X}|^2+\mathcal{B}_1+\mathcal{B}_2-\mathcal{G}_{2}-\frac{3}{4}\mathcal{D}\\
	\le & -\frac{M\delta}{\sigma_+^4v_+^2} \left( \int_{\mathbb{T}^2}\int_{y_0}^{1} w \, dy \, dx' \right)^2 +\delta\left(\frac{11\alpha_+}{10}+C(\delta+\varepsilon^2)\right)\int_{\mathbb{T}^2}\int_{y_0}^1|w|^2\,dy \, dx'\\
    &  +\sigma^* \int_{\Omega}a_{x_1}\frac{(\gamma+1)}{3\gamma^2}p(\tilde{v})^{-\frac{1}{\gamma}-2}|p(v)-p(\tilde{v})|^3\,dx\\
    & -\frac{27\delta\alpha_+}{40}\int_{\mathbb{T}^2}\int_{y_0}^{1} (y-y_0)(1-y) |\partial_yw|^2 dydx'-\frac{27\sigma_+(2\mu+\lambda)^2}{40|p''(v_+)|\delta}\int_{\mathbb{T}^2}\int_{y_0}^{1}\frac{|\nabla_{x'}w|^2}{y(1-y)}dydx'.
\end{align*}

Denoting $\mathcal{K}:=\int_{\Omega}a_{x_1}|p(v)-p(\tilde{v})|^3\,dx$ and using the smallness of $\delta$ and $\varepsilon$, we can rewrite the above inequality as
\begin{align}
\begin{aligned} \label{4.36}
	&-\frac{\delta}{2M}|\dot{X}|^2+\mathcal{B}_1+\mathcal{B}_2-\mathcal{G}_{2}-\frac{3}{4}\mathcal{D}\\
	\le & \delta\alpha_+\left[\frac{6}{5}\int_{\mathbb{T}^2}\int_{y_0}^1|w|^2\,dy \, dx'-\frac{27}{40}\int_{\mathbb{T}^2}\int_{y_0}^{1} (y-y_0)(1-y) |\partial_yw|^2 dydx'\right]\\
    &-\frac{M\delta}{\sigma_+^4v_+^2} \left( \int_{\mathbb{T}^2}\int_{y_0}^{1} w \, dy \, dx' \right)^2-\frac{27\sigma_+(2\mu+\lambda)^2}{40|p''(v_+)|\delta}\int_{\mathbb{T}^2}\int_{y_0}^{1}\frac{|\nabla_{x'}w|^2}{y(1-y)}dydx'+ C\mathcal{K}.
\end{aligned}
\end{align}

Note that the identity:
\begin{equation} \label{widentity}
\int_{\mathbb{T}^2}\int_{y_0}^1 |w-\bar {w}|^2 dydx' = \int_{\mathbb{T}^2}\int_{y_0}^1w^2 dydx' -\frac{1}{(1-y_0)}\left(\int_{\mathbb{T}^2}\int_{y_0}^{1}wdydx'\right)^2,
\end{equation}
where $\bar{w}=\frac{1}{(1-y_0)}\int_{\mathbb{T}^2}\int_{y_0}^{1}wdydx'$.

Applying Lemma \ref{lem-poin} to \eqref{4.36} with \eqref{widentity} and $y_0<\frac{1}{10}$, we have
\begin{align*}
	&-\frac{\delta}{2M}|\dot{X}|^2+\mathcal{B}_1+\mathcal{B}_2-\mathcal{G}_{2}-\frac{3}{4}\mathcal{D} \\
	\le &-\frac{3\delta\alpha_+}{20}\int_{\mathbb{T}^2}\int_{y_0}^1|w|^2\,dy\,dx' \\
	&+\frac{3\delta\alpha_+}{2}\left(\int_{\mathbb{T}^2}\int_{y_0}^1w\,dy\,dx'\right)^2-\frac{M\delta}{\sigma_+^4v_+^2} \left( \int_{\mathbb{T}^2}\int_{y_0}^{1} w \, dy \, dx' \right)^2 \\
        &\quad -\left(\frac{27\sigma_+(2\mu+\lambda)^2}{40\delta |p''(v_+)|}-\frac{27\delta\alpha_+}{320\pi^2}\right)\int_{\mathbb{T}^2}\int_{y_0}^{1}\frac{|\nabla_{x'}w|^2}{y(1-y)}dydx'+C\mathcal{K}.
\end{align*}
Finally, choosing $M=\frac{3}{2}\sigma_+^4v_+^2\alpha_+$, $C_1=\frac{3\alpha_+}{20}$ and using the smallness of $\delta$, we conclude
\begin{align}
	\begin{aligned} \label{est main lemma}
	\mathcal{R}_1=&-\frac{\delta}{2M}|\dot{X}|^2+\mathcal{B}_1+\mathcal{B}_2-\mathcal{G}_{2}-\frac{3}{4}\mathcal{D}\\
	\le & -C_1\int_{\Omega} |\tilde{v}_{x_1}||p(v)-p(\tilde{v})|^2 dx+C\int_{\Omega}a_{x_1}|p(v)-p(\tilde{v})|^3\,dx.
	\end{aligned}
\end{align}

\subsection{Estimate of the boundary term $\mathcal{P}$}

\ 

In this subsection, we estimate the boundary terms. Such terms arise naturally in half-line problems, unlike in the whole space case. Consequently, their handling poses an additional difficulty. In addition, since our setting is outflow condition, certain terms give us some good terms.

\ 

Since $a\ge 1$, $\rho > 0$, $Q(v|\tilde{v})\ge0$, and $u_1\big |_{x_1=0}=u_-<0$, we have
\begin{equation*}
\mathcal{P}_1(U)= u_-\int_{\mathbb{T}^2} a\rho Q(v|\tilde{v}) \bigg|_{x_1=0}  dx'\le 0.    
\end{equation*}

As above, we also have
\begin{equation*}
\mathcal{P}_2(U)=u_-\int_{\mathbb{T}^2} a\rho \frac{|\mathbf{h}-\tilde{\mathbf{h}}|^2}{2} \bigg|_{x_1=0}  dx' \le 0.
\end{equation*}

To estimate $\mathcal{P}_3(U)$, first note that from \eqref{lem:shock-est},
\begin{equation}\label{ushockproperty}
|u_1-\tilde{u}| \bigg|_{x_1=0}  \le C\delta e^{-C\delta t-C\delta X(t)-C\delta \beta}.    
\end{equation}

From \eqref{bddx12}, we have 
\begin{equation}\label{shiftproperrttty}
|\dot{X}(t)|\le \varepsilon  \Longrightarrow |X(t)|=\left|\int_{0}^{t}\dot{X}(\tau)d\tau\right|\le \int_{0}^{t}|\dot{X}(\tau)|d\tau \le \varepsilon t\le \frac{t}{2}.   
\end{equation}

Thus, combining \eqref{ushockproperty} and \eqref{shiftproperrttty}, we have
\begin{equation*}
|u_1-\tilde{u}| \bigg|_{x_1=0}  \le C\delta e^{-C\delta t-C\delta \beta}.    
\end{equation*}

Using this with \eqref{perturbation_small}, we have
\begin{align*}
|\mathcal{P}_3(U)|=&\left|\int_{\mathbb{T}^2} a(p(v)-p(\tilde{v}))(u_1-\tilde{u}) \bigg|_{x_1=0} dx'\right|\\
\le&C\|(p(v)-p(\tilde{v}))\|_{L^\infty}  \int_{\mathbb{T}^2} |u_1-\tilde{u}| \bigg|_{x_1=0} dx'\\
\le& C \varepsilon \delta e^{-C\delta t-C\delta \beta}.
\end{align*}

Therefore, we summarize
\begin{equation}
\mathcal{P}(U) \le C\varepsilon \delta e^{-C\delta t-C\delta \beta}.    
\end{equation}

\subsection{Estimate of the remaining part $\mathcal{R}_2$}\label{sec:4.6}

\ 

In this subsection, we finish the proof of Lemma \ref{lemma5.1} with previous results.

We first use Young's inequality to get 
\beq\label{2young}
\dot{X}\sum_{i=3}^5Y_{i}\le \frac{\delta}{4M}|\dot{X}|^2 +\frac{C}{\delta}\sum_{i=3}^5|Y_{i}|^2.
\eeq

Substituting \eqref{est main lemma} and \eqref{2young} into \eqref{est}, we have 
\begin{align}
	\begin{aligned}\label{est-2}
	&\frac{d}{dt}\int_{\Omega}a\rho\eta(U|\tilde{U})\,dx \\
	\leq&-C_1\mathcal{G}^S + C\mathcal{K}\\
	&-\frac{\delta}{4M}|\dot{X}|^2 + \frac{C}{\delta}\sum_{i=3}^5|Y_{i}|^2 +\sum_{i=3}^8\mathcal{B}_i -\mathcal{G}_1-\mathcal{G}_3-\frac{1}{4}\mathcal{D}+\mathcal{P},
	\end{aligned}
\end{align}
where
\begin{align*}
&\mathcal{G}^S= \int_{\Omega} |\tilde{v}_{x_1}| |(p(v)-p(\tilde{v})) |^2 \,dx, \\
&\mathcal{K}=\int_{\Omega} |a_{x_1}| |p(v)-p(\tilde{v})|^3 \,dx. \\ 
\end{align*}

\ 

To estimate the remaining terms, we utilize the good terms $\mathcal{G}_1$, $\mathcal{G}^S$, $\mathcal{G}_3$, and $\mathcal{D}$.

The estimates of the remaining terms are conventional and less significant, except for $\mathcal{B}_8$. The detailed proof can be found in \cite{lee4923620long}, so we omit the proof here and only give the resulting estimates.

\begin{align*}
\mathcal{K}&\le C\varepsilon(\mathcal{D}+\mathcal{G}^S), \\
\frac{C}{\delta}\sum\limits_{i=3}^5|Y_i|^2 &\le \frac{C_1\mathcal{G}^S+\mathcal{G}_1+\mathcal{G}_3}{10},\\
\mathcal{B}_3+\mathcal{B}_4+\mathcal{B}_5+\mathcal{B}_6+\mathcal{B}_7&\le\frac{C_1\mathcal{G}^S+\mathcal{G}_1+\mathcal{G}_3+\mathcal{D}}{15}.
\end{align*}

For now, we concentrate only on $\mathcal{B}_8$. To estimate it, first note that
\[R=\frac{2\mu+\lambda}{v}(\nabla_x \mathbf{u} \nabla_x v- \operatorname{div}_x\mathbf{u} \nabla_x v)-\mu \nabla_x \times \left(\nabla_x \times \mathbf{u}\right).\]

Substituting this into $\mathcal{B}_8$, we have
\begin{align*}
\mathcal{B}_8=&(2\mu+\lambda)\int_{\Omega}a(\mathbf{h}-\tilde{\mathbf{h}})\cdot \frac{(\nabla_x \mathbf{u}  \nabla_x v - \operatorname{div}_x\mathbf{u} \nabla_x v)}{v} \, dx\\
&-\mu \int_{\Omega}a(\mathbf{h}-\tilde{\mathbf{h}})\cdot \nabla_x \times  \left(\nabla_x \times \mathbf{u}\right) \, dx\\
=:&\mathcal{B}_{8,1}+\mathcal{B}_{8,2}.
\end{align*}

We can estimate $\mathcal{B}_{8,1}$ as 
\begin{equation*}
\mathcal{B}_{8,1}\le C(\delta+\varepsilon)\left(\mathcal{G}^S+\mathcal{G}_1+\mathcal{G}_3+\mathcal{D}+||\nabla_x(\mathbf{u}-\tilde{\mathbf{u}})||_{H^1}^2\right).
\end{equation*}

Since the proof is the same as \cite{lee4923620long}, we omit it.

\ 

When we estimate $\mathcal{B}_{8,2}$, we should consider the boundary term different $\mathcal{B}_{8,1}$. Thus, we need to estimate this term more delicately.

From the definition of $\mathbf{h}$, 
\begin{align*}
\mathcal{B}_{8,2}=&-\mu \int_{\Omega}a(\mathbf{u}-\tilde{\mathbf{u}})\cdot \nabla_x \times \left(\nabla_x \times \mathbf{u}\right)\, dx\\
&+\mu(2\mu+\lambda)\int_{\Omega}a\nabla_x(v-\tilde{v})\cdot \nabla_x \times \left(\nabla_x \times \mathbf{u}\right) \, dx\\
=&:\mathcal{B}_{8,2,1}+\mathcal{B}_{8,2,2}.
\end{align*}

Thus, using the integration by parts with Lemma \ref{vectori}$-(1)$ and the fact that $\nabla_x \times \tilde{u}\equiv 0$,
\begin{align*} 
\mathcal{B}_{8,2,1}=&-\mu \int_{\Omega}a(\mathbf{u}-\tilde{\mathbf{u}})\cdot \nabla_x \times \left(\nabla_x \times (\mathbf{u}-\mathbf{\tilde{u}})\right) \, dx \\
=&-\mu \int_{\Omega}\nabla_x\times \left(a(\mathbf{u}-\tilde{\mathbf{u}})\right)\cdot \nabla_x \times \left(\mathbf{u}-\tilde{\mathbf{u}}\right) \, dx\\
&-\mu\int_{\mathbb{T}^2} \left[a(\mathbf{u}-\mathbf{\tilde{u}})\times(\nabla_x\times(\mathbf{u}-\mathbf{\tilde{u}}))\right] \cdot (1, 0, 0) \bigg|_{x_1=0}dx'  \\
=&:\mathcal{B}_{8,2,1,1}+\mathcal{B}_{8,2,1,2}. 
\end{align*}

\begin{align}
\begin{aligned}\label{B8.2.1}
\mathcal{B}_{8,2,1,1}=&-\mu \int_{\Omega}\left[\nabla_xa\times\left(\mathbf{u}-\tilde{\mathbf{u}}\right)+a\nabla_x\times \left(\mathbf{u}-\tilde{\mathbf{u}}\right) \right]\cdot \nabla_x \times \left(\mathbf{u}-\tilde{\mathbf{u}}\right) \, dx \\
=&-\mu\int_{\Omega} a|\nabla_x \times (\mathbf{u}-\tilde{\mathbf{u}})|^2 \, dx\\
& \ \ \ -\mu \int_{\Omega}a_{x_1} \left[u_2(\partial_{x_1}u_2-\partial_{x_2}u_1)-u_3(\partial_{x_3}u_1-\partial_{x_1}u_3)\right]\, dx \\
\le& -\mu\int_{\Omega} a|\nabla_x \times (\mathbf{u}-\tilde{\mathbf{u}})|^2 \, dx \\
& \ \ \ +\mu \delta^{\frac{3}{4}}\int_{\Omega}\left|a_{x_1}\right| (u_2^2+u_3^2)\, dx \\
& \ \ \ +\mu \delta^{\frac{3}{4}}\int_{\Omega}( |\partial_{x_1}u_2-\partial_{x_2}u_1|^2+|\partial_{x_3}u_1-\partial_{x_1}u_3|^2)\, dx.\\
\end{aligned}
\end{align}

To use the $\mathcal{G}_3$, we apply $\mathbf{u}=\mathbf{h}+(2\mu+\lambda)\nabla_xv$ to the second term in \eqref{B8.2.1}. So, we have
\begin{align*}
\mathcal{B}_{8,2,1,1}\le&-\mu\int_{\Omega} a|\nabla_x \times (\mathbf{u}-\tilde{\mathbf{u}})|^2 \, dx \\
&+C\delta^{\frac{3}{4}}\int_{\Omega}\left|a_{x_1}\right| (h_2^2+h_3^2)\, dx+C\delta^{\frac{3}{4}}\int_{\Omega} |\nabla_{x'} v|^2\, dx\\
&+C\delta^{\frac{3}{4}} \int_{\Omega}|\nabla_x \times (\mathbf{u}-\tilde{\mathbf{u}})|^2\, dx \\
\le&-\frac{3\mu}{4}\int_{\Omega} a|\nabla_x \times (\mathbf{u}-\tilde{\mathbf{u}})|^2 \, dx+C\sqrt{\delta}(\mathcal{G}_3+\mathcal{D}).       
\end{align*}

As in $\mathcal{P}_3$, we can easily estimate $\mathcal{B}_{8,2,1,2}$ as
\begin{align*}
|\mathcal{B}_{8,2,1,2}|\le &C\int_{\mathbb{T}^2} |\mathbf{u}-\mathbf{\tilde{u}}||\nabla_x\times(\mathbf{u}-\mathbf{\tilde{u}})| \bigg|_{x_1=0}dx' \\
\le &C \varepsilon \delta e^{-C\delta t-C\delta \beta}.  
\end{align*}

Similarly, we apply integration by parts to $\mathcal{B}_{8,2,2}$ with Lemma \ref{vectori}$-(2)$ as
\begin{align*}
\mathcal{B}_{8,2,2}=&\mu(2\mu+\lambda)\int_{\Omega}\nabla_x\times(a\nabla_x(v-\tilde{v}))\cdot (\nabla_x\times \mathbf{u}) dx\\
&+\mu(2\mu+\lambda)\int_{\mathbb{T}^2}  \left[a\nabla_x(v-\tilde{v})\times(\nabla_x\times\mathbf{u})\right]\cdot (1, 0, 0)\bigg|_{x_1=0} dx' \\
=&\mu(2\mu+\lambda)\int_{\Omega}(\nabla_xa\times\nabla_x(v-\tilde{v}))\cdot (\nabla_x\times (\mathbf{u}-\mathbf{\tilde{u}})) dx \\
&+\mu(2\mu+\lambda)\int_{\mathbb{T}^2}  \left[a\nabla_x(v-\tilde{v})\times(\nabla_x\times(\mathbf{u}-\mathbf{\tilde{u}}))\right]\cdot(1, 0, 0)\bigg|_{x_1=0} dx'\\
=&:\mathcal{B}_{8,2,2,1}+\mathcal{B}_{8,2,2,2}.
\end{align*}

Using Cauchy-Schwartz inequality and Young's inequality, 
\begin{align*}
\mathcal{B}_{8,2,2,1}\le& \int_{\Omega} a_{x_1}|\nabla_{x'}(v-\tilde{v})| |\nabla_x\times(\mathbf{u}-\mathbf{\tilde{u}})| dx\\
\le& \frac{\mu}{4}\int_{\Omega} a|\nabla_x \times (\mathbf{u}-\tilde{\mathbf{u}})|^2 \, dx+C\delta^3 \mathcal{D}.
\end{align*}

To estimate \(\mathcal{B}_{8,2,2,2}\), we first note that the following estimate holds by virtue of the facts that \(\partial_{x_2}(\mathbf{u}-\tilde{\mathbf{u}})\bigg\vert{}_{x_1=0}=\partial_{x_3}(\mathbf{u}-\tilde{\mathbf{u}})\bigg\vert{}_{x_1=0}=0\) and the interpolation inequality.
\begin{align*}
&\left[\nabla_x(v-\tilde{v})\times(\nabla_x\times(\mathbf{u}-\tilde{\mathbf{u}}))\right]\cdot(1, 0, 0) \bigg|_{x_1=0}\\
\le &C|\nabla_{x'}(v-\tilde{v})|\bigg|_{x_1=0}|\partial_{x_1}(\mathbf{u}-\tilde{\mathbf{u}})|\bigg|_{x_1=0} \\
\le &C|\nabla_{x'}(v-\tilde{v})|\bigg|_{x_1=0} \sqrt{\|\partial_{x_1}(\mathbf{u}-\tilde{\mathbf{u}})\|_{L^2(\mathbb{R}_+)}\|\partial_{x_1x_1}(\mathbf{u}-\tilde{\mathbf{u}})\|_{L^2(\mathbb{R}_+)}}\\
\le &C|\nabla_{x'}(v-\tilde{v})|\bigg|_{x_1=0}\|\nabla_x(\mathbf{u}-\tilde{\mathbf{u}})\|_{H^1(\mathbb{R}_+)}
\end{align*}

From the above inequality, we have
\begin{align*}
\mathcal{B}_{8,2,2,2}=&-\mu(2\mu+\lambda)\int_{\mathbb{T}^2}  \left[a\nabla_x(v-\tilde{v})\times(\nabla_x\times(\mathbf{u}-\mathbf{\tilde{u}}))\right]\cdot(1, 0, 0)\bigg|_{x_1=0} dx'\\ 
\le& C_{\nu}\int_{\mathbb{T}^2} |\nabla_{x'}(v-\tilde{v})|^2 \bigg|_{x_1=0} \, dx'+\nu  \|\nabla_x(\mathbf{u}-\tilde{\mathbf{u}})\|_{H^1}^2, 
\end{align*}
for some sufficiently small constant $\nu>0$ and $C_\nu>0$ depending on $\nu$.

Combining the above estimates with the smallness of the parameters, we get
\begin{align*}
\mathcal{B}_8\le& \frac{1}{40}(C_1\mathcal{G}^S+\mathcal{G}_1+\mathcal{G}_3+\mathcal{D})+(\nu+C\delta+C\varepsilon)||\nabla_x(\mathbf{u}-\tilde{\mathbf{u}})||^2_{H^1}\\
&+C\delta e^{-C\delta t-C\delta \beta}+ C_\nu\int_{\mathbb{T}^2} |\nabla_{x'}(v-\tilde{v})|^2 \bigg|_{x_1=0}dx'.
\end{align*}

In conclusion, we summarize that \eqref{est-2} implies
\begin{align}
\begin{aligned} \label{4.36con}
	&\frac{d}{dt}\int_{\Omega}a\rho\eta(U|\tU)\,dx+\frac{\delta}{4M}|\dot{X}|^2+\frac{1}{2}\mathcal{G}_1+\frac{1}{2}\mathcal{G}_3 +\frac{C_1}{2}\mathcal{G}^S+\frac{1}{10}\mathcal{D}+\mathcal{P}\\
	 \le& (\nu+C\delta+C\varepsilon)||\nabla_x(\mathbf{u}-\tilde{\mathbf{u}})||^2_{H^1}+ C\delta e^{-C\delta t-C\delta \beta}+C_\nu\int_{\mathbb{T}^2} |\nabla_{x'}(v-\tilde{v})|^2 \bigg|_{x_1=0}dx'.
\end{aligned}
\end{align}

Therefore, integrating \eqref{4.36con} over $\left[0, t\right]$, we complete the proof of Lemma \ref{lemma5.1}. \qed

\section{Higher Order Estimate}
In this section, we obtain the energy estimates up to the $H^2$ level, whereas the half-space problem requires a more refined analysis because of the presence of the boundary in our setting.
\subsection{Notation}
Throughout Section~6, we use the following notation.

\begin{enumerate}
\item \emph{Perturbation variables: }
We denote the perturbations of $v$ and $\mathbf{u}$ by
$$
\phi := v-\widetilde v,
\qquad
{\psi} := \mathbf{u}-\tilde{\mathbf{u}},
$$
respectively. Also, we use the notation
$$\psi' \in \{\psi_2, \psi_3\} $$
to denote either transverse velocity perturbation.\medskip
\item \emph{Spatial variables: }
Since the estimates in the \(x_2\)- and \(x_3\)-directions are essentially identical, we use
\[
    z\in\{x_2,x_3\}
\]
to denote either transverse variable whenever no distinction is needed.
Moreover, since the transverse domain is \(\mathbb{T}^2\), the Plancherel identity and the elementary inequality
\(2ab\leq a^2+b^2\) yield
\[
    \|\boldsymbol{\psi}_{x_2x_3}\|^2
    \leq
    \|\boldsymbol{\psi}_{x_2x_2}\|^2
    +
    \|\boldsymbol{\psi}_{x_3x_3}\|^2.
\]
\medskip
\item \emph{Constants in Young's inequality: }
Unless otherwise specified, we use \(\chi>0\) to denote a sufficiently small constant appearing in Young's inequality. In particular, for any \(f,g\),
\[
    |fg|
    \leq
    \chi |f|^2+\frac{2}{\chi}|g|^2.
\]
Constants depending on \(\chi\), such as \(2/\chi\), are denoted by \(C_\chi\), or simply by \(C\) whenever their dependence on \(\chi\) is irrelevant to the analysis.
\end{enumerate}
Using these notations, we rewrite the perturbed systems as follows:
\beq
\left\{
\begin{aligned} \label{eq:peq}
&\partial_t \phi
 + \mathbf{u} \cdot \nabla_x \phi
 - \dot{X}\tilde v_{x_1}
 + (vF)\,\tilde v_{x_1}
 = v\operatorname{div}_x \psi,
\\[1mm]
&\partial_t \psi
 + \mathbf{u} \,\nabla_x \psi
 + vp'(v) \nabla_x \phi
 + (vS)\nabla_x \tilde v
 + (vF - \dot X) \,\mathbf{\tilde u}_{x_1} 
 = v\,\big(\mu \Delta_x \psi
    + (\mu + \lambda) \nabla_x \operatorname{div}_x \psi\big).
\end{aligned}
\right.
\eeq
Here, $F$ and $S$ are defined as
\begin{align} \label{eq:FS}
vF = \sigma^*\phi + \psi_1, \quad S = p'(v) - p'(\tilde v).
\end{align}
\subsection{Useful Estimates} We introduce the following two lemmas for computational convenience in the subsequent analysis.
\begin{lemma} \label{L: estiF}
For $F$ and $S$ defined in \eqref{eq:FS}, the following estimates hold:
\begin{itemize}
\setlength{\itemsep}{8pt} 
\item[(1)] $\|\nabla_x^\alpha(vF)\|^2 \le C\big(\|\nabla_x^\alpha \phi\|^2 + \|\nabla_x^\alpha \psi_1\|^2\big)$ for any $\alpha \in \{0,1,2\},$
\item[(2)] $\|(vF)_t\|^2 \le C\big(\|\phi_t\|^2 + \|\psi_{1t}\|^2\big),$
\item[(3)] $\|\nabla_x^\alpha S\|^2 \le C\big(\sum_{i=1}^\alpha \|\nabla_x^i \phi\|^2 + \delta^2 \mathcal{G}^s \big)$ for any $\alpha \in \{1,2\},$
\item[(4)] $\|S_t\|^2 \le C\big( \|\phi_t\|^2 + \delta^2 \mathcal{G}^2 \big).$
\end{itemize}
\end{lemma}
\begin{proof}
Using the definition of $F$ in \eqref{eq:FS}, estimates (1) and (2) follow trivially.

For the case $|\alpha| = 2$ in (3), for any $\beta \in \{x_1, x_2, x_3\}$, we obtain the following:
$$\pa_\beta^2 S = p'''(v) \pa_\beta\phi + (p'''(v) - p'''(\tilde{v}))\pa_\beta\tilde{v} + p''(v) \pa_\beta^2 \phi + (p''(v) - p''(\tilde{v})) \pa_\beta^2 \tilde{v}.$$
Squaring both sides and integrating the spatial variables yields the desired inequality. The remaining cases for (3) and (4) can be easily obtained in a similar manner.
\end{proof}
\begin{lemma}\label{L:psig}
There exists a constant $C > 0$ independent of \(\varepsilon\), \(\delta\), \(\chi\), and \(T\), such that
for all \(t \in [0,T]\),
$$ \int_{\Omega} |\tilde v_{x_1}||\psi|^2 \, dx \le C\big(\mathcal{G}^s + \delta \mathcal{G}_1 + \delta \mathcal{G}_3 + \delta^2 \mathcal{D}\big).$$
\end{lemma}
\begin{proof}
See \cite{lee4923620long}, Lemma 5.1.
\end{proof}
\subsection{Vanishing the $O(1)$-scale bad term}
As seen in \eqref{4.36con}, the zeroth-order estimate leaves the boundary term
$$
    |\phi_z\big|_{x_1=0}|^2
$$
uncontrolled. Since this term cannot be made arbitrarily small, we first estimate it independently of the zeroth-order estimate to avoid a circular argument.

\begin{lemma}
Under the \textit{a priori} assumption, there exists a constant \label{L:bdesti}
\(C>0\) independent of \(\varepsilon\), \(\delta\), \(\chi\) and \(T\), such that
for all \(t \in [0,T]\), it holds
\begin{equation}
\begin{aligned}\nonumber
&\sup_{0 \le t \le T} \|\partial_{z}(\phi, \psi)(t) \|^2 
 + \int_0^t \big( \|\phi_z \big|_{x_1=0}\|_{L^2(\mathbb{T}^2)}^2 + \|\nabla_x \psi_{z}\|^2 \big) \, ds \\[2mm]
&\le C\|\partial_{z} ( \phi_0, \psi_0 ) \|^2 + C(\varepsilon + \delta^2) \int_0^t \big( \|\nabla_x \psi\|_{H^2}^2 + \|\nabla_x \phi\|_{H^1}^2 + \mathcal{G}^s(s) \big) \, ds
\end{aligned}
\end{equation}
\end{lemma}
\begin{proof}
Taking $\partial_{z}$ to \eqref{eq:peq}, it can be written as follows:
\begin{equation}
\left\{
\begin{aligned}
&\partial_t \phi_{z}
 + \mathbf{u} \cdot \nabla_x \phi_{z}
 = \, v \operatorname{div}_x \psi_{z} + R,
\\[2mm]\label{eq:L3}
&\partial_t \psi_{z}
 + \mathbf{u} \,\nabla_x \psi_{z}
 + vp'(v) \nabla_x \phi_{z}
 = v\big(\mu \Delta_x \psi_{z}
    + (\mu + \lambda) \nabla_x \operatorname{div}_x \psi_{z} \big)
 + T,
\end{aligned}
\right.
\end{equation}
where $R$ and $T$ in this lemma are defined as
\begin{equation}
\begin{aligned}\nonumber
R &= v_{z}\operatorname{div}_x \psi - \mathbf{u}_{z} \cdot \nabla_x \phi - (vF)_{z}(\tilde v)_{x_1}, \\[2mm]
T &= - \mathbf{u}_{z} \nabla_x \psi - \big(vp'(v)\big)_{z} \nabla_x \phi - (vS)_{z} \nabla_x \tilde v - (vF)_{z} \tilde u_{x_1} + v_{z}\big(\mu \Delta_x \psi
    + (\mu + \lambda) \nabla_x \operatorname{div}_x \psi \big).
\end{aligned}
\end{equation}
Multiplying $\eqref{eq:L3}_1$ by $-p'(v)\phi_z$ and $\eqref{eq:L3}_2$ by $\psi_z$, and adding the resulting equations, we obtain the following.
\beq
\begin{aligned}\nonumber
&-\frac{p'(v)}{2}\partial_t |\phi_z|^2 + \frac12\partial_t |\psi_z|^2 + \mathbf{u} \big(-p'(v) \phi_z \nabla_x \phi_z + \psi_z \nabla_x \psi_z \big) \\
&\quad+vp'(v)\big(\phi_z \operatorname{div}_x \psi_z + \nabla_x \phi_z \psi_z \big)  - v\big(\mu \Delta_x \psi_{z}
    + (\mu + \lambda) \nabla_x \operatorname{div}_x \psi_{z} \big) \psi_z = -p'(v)R \phi_z + T\psi_z
\end{aligned}
\eeq
Note that the Lemma~\ref{vectori} yields the following relationships.
\beq
\begin{aligned}\nonumber
-\frac{p'(v)}{2}\partial_t |\phi_z|^2 + \frac12\partial_t |\psi_z|^2 &=\partial_t \big(-\frac{p'(v)}{2}|\phi_z|^2 + |\psi_z|^2 \big) + \frac{p''(v) v_t}{2} |\phi_z|^2 \\[2mm]
\mathbf{u}\cdot \big(-p'(v) \phi_z \nabla_x \phi_z + \psi_z \nabla_x \psi_z \big) & = \operatorname{div}_x\big(-\frac{p'(v)\mathbf{u}}{2}|\phi_z|^2 + \frac12 \mathbf{u}|\psi_z|^2 \big) \\
&\quad- \operatorname{div}_x \mathbf{u} \big(-\frac{p'(v)}{2}|\phi_z|^2 + \frac12|\psi_z|^2 \big) + \frac12 p''(v) \nabla_x v \cdot\mathbf{u} |\phi_z|^2 \\[4mm]
vp'(v) \big(\phi_z \operatorname{div}_x \psi_z + \nabla_x \phi_z \psi_z \big)  &= \operatorname{div}_x \big(vp'(v)\phi_z \psi_z \big) - \phi_z \nabla_x\big(vp'(v)\big)\psi_z\\[2mm]
v\big(\mu \Delta_x \psi_{z}
    + (\mu + \lambda) \nabla_x \operatorname{div}_x \psi_{z} \big)\cdot \psi_z &= \operatorname{div}_x \big(\mu v\, \psi_z \nabla_x \psi_z+(\mu + \lambda)v\, \operatorname{div}_x \psi_z \psi_z\big) \\[1mm]
&\quad-v\big(\mu |\nabla_x \psi_z|^2 + (\mu + \lambda)|\operatorname{div}_x \psi_z|^2 \big) \\[1mm]
&\quad -\nabla_x v \cdot\big( \mu \psi_z \nabla_x \psi_z+(\mu + \lambda) \operatorname{div}_x \psi_z \psi_z\big)\\
\end{aligned}
\eeq
Furthermore, using the following fact from mass conservation,
$$\frac12 p''(v)|\phi_z|^2 \underbrace{(v_t + \mathbf{u} \cdot\nabla_x v)}_{=\,v\operatorname{div}_x \mathbf u} = \frac12 vp''(v)\operatorname{div}_x \mathbf{u}|\phi_z|^2 $$
Consequently, we obtain the following equation.
\beq
\begin{aligned}\label{eq:ref}
&\frac{d}{dt} \int_{\Omega} \frac12\big(-p'(v) |\phi_z|^2 + |\psi_z|^2 \big)\,dx \\
&\qquad -u_- \int_{\mathbb{T}^2} \big(-\frac{p'(v)}{2}|\phi_z|^2 + \frac12 |\psi_z|^2 \big)\Big|_{x_1 = 0} \, dx'
+ \int_{\Omega} v\big(\mu |\nabla_x \psi_z|^2 + (\mu+\lambda)|\operatorname{div}_x \psi_z|^2 \big) \, dx \\[2mm]
&=\int_{\Omega} \operatorname{div}_x \mathbf{u} \big(-\frac12vp''(v)|\phi_z|^2 + \frac12 p'(v) |\phi_z|^2 - \frac12 |\psi_z|^2 \big) \, dx + \int_{\Omega} -p'(v) R \phi_z \, dx + \int_{\Omega} T \psi_z \, dx \\[2mm]
&\qquad + \int_{\Omega} \phi_z \nabla_x \big(vp'(v)\big) \cdot \psi_z + \nabla_x v \cdot \big((\mu + \lambda) \operatorname{div}_x \psi_z \psi_z + \mu \psi_z \nabla_x \psi_z \big) \, dx \\[2mm]
&\qquad - \int_{\mathbb{T}^2} v\big((\mu + \lambda) \operatorname{div}_x \psi_z \psi_z + \mu\psi_z \nabla_x \psi_z - p'(v) \phi_z \psi_z\big) \Big|_{x_1 = 0} \, dx' \quad := \sum_{i=1}^5 J_i.
\end{aligned}
\eeq
Since the \textit{a priori} assumption does not control $\operatorname{div}_x \mathbf{u}$, we decompose this term as follows.
\[
\operatorname{div}_x \mathbf{u} = \operatorname{div}_x \psi + \operatorname{div}_x \tilde {\mathbf{u}} = \operatorname{div}_x \psi + \tilde u_{x_1},
\]
In light of the above identity, Lemmas~\ref{lem:shock-est} and \ref{GNIsevera}, and the following \textit{a priori} assumptions
$$\|\phi_z\|,\,\|\psi_z\|,\, \|\operatorname{div}_x\psi\|_{H^1} \le C\e $$
the following estimates hold:
\beq
\begin{aligned} \label{eq:div}
|J_1| &\le C\int_{\Omega} |\operatorname{div}_x \mathbf{u}| \big(|\phi_z|^2 + |\psi_z|^2 \big) \, dx \le C\big(\|\operatorname{div}_x \psi\|_{H^2} + \|\tilde u_{x_1}\|_{L^\infty}\big)\big(\|\phi_z\|^2 + \|\psi_z\|^2 \big)\\[2mm]
&\le C\big(\|\nabla_x^2 \operatorname{div}_x \psi\| + \|\operatorname{div}_x \psi\|_{H^1} + \delta^2\big)\big(\|\phi_z\|^2 + \|\psi_z\|^2 \big) \\[2mm]
&\le C\eps(\|\phi_z\| + \|\psi_z\| \big)\big\|\nabla_x^2 \operatorname{div}_x \psi\big\| + C(\varepsilon + \delta^2)\big(\|\phi_z\|^2 + \|\psi_z\|^2 \big) \\[2mm]
&\le C(\varepsilon + \delta^2)\big(\|\phi_z\|^2 + \|\psi_z\|^2 + \|\nabla_x^2 \operatorname{div}_x \psi\|^2\big).
\end{aligned}
\eeq
Since the planar shock profile depends only on \(x_1\), we have $v_z=\phi_z$ and
$    \mathbf{u}_z=\psi_z
$. Consequently, $R$ and $T$ can be expressed as
\begin{equation}
\begin{aligned}\nonumber
R &= \phi_{z}\operatorname{div}_x \psi - \psi_{z} \cdot \nabla_x \phi - (vF)_{z}(\tilde v)_{x_1}, \\[2mm]
T &= - \psi_{z} \nabla_x \psi - \big(vp'(v)\big)_{z} \nabla_x \phi - (vS)_{z} \nabla_x \tilde v - (vF)_{z} \tilde u_{x_1} + \phi_{z}\big(\mu \Delta_x \psi
    + (\mu + \lambda) \nabla_x \operatorname{div}_x \psi \big). \\
\end{aligned}
\end{equation}
Moreover, using the following identity derived by the barotropic condition,
\begin{align}\label{eq:baro1}
    \big(vp'(v)\big)_z = \big(p'(v) + vp''(v)\big)\phi_z = \gamma^2 v^{-\gamma-1} \phi_z
\end{align}
$T$ is further simplified as follows:
\begin{align*} \nonumber
T &= - \psi_{z} \nabla_x \psi - -\gamma^2v^{-\gamma-1}\phi_{z} \nabla_x \phi - S \phi_z \nabla_x \tilde v \\
&\quad - vS_{z} \nabla_x \tilde v - (vF)_{z} \tilde u_{x_1} + \phi_{z}\big(\mu \Delta_x \psi
    + (\mu + \lambda) \nabla_x \operatorname{div}_x \psi \big).
\end{align*}
Therefore, applying Lemmas \ref{lem:shock-est}, \ref{GNIsevera}, and \ref{L: estiF}; and using the \textit{a priori} assumptions
$$\|\phi_z\|_{H^1},\,\|\nabla_x \phi\|_{H^1},\,\|\nabla_x \psi\|_{H^1} \le C\e, $$
we obtain the following estimates for $R$ and $T$:
\beq
\begin{aligned} \label{eq:estiRT}
\|R\| & \le C\big(\|\phi_z\|_{H^1}\|\operatorname{div}_x \psi\|_{H^1} + \|\psi_z\|_{H^1}\|\nabla_x \phi\|_{H^1} + \|\tilde v_{x_1}\|_{L^\infty}\|(vF)_x\|\big)\\
&\le C(\varepsilon + \delta^2) \big(\|\operatorname{div}_x \psi\|_{H^1} + \|\psi_z\|_{H^1} + \|(vF)_z\|^2 \big), \\[3mm]
\|T\| &\le C\Big(\|\psi_z\|_{H^1} \|\nabla_x \psi\|_{H^1} + \|\phi_z\|_{H^1} \|\nabla_x \phi\|_{H^1} + \|\nabla_x \tilde v\|_{L^\infty} \|S\|_{H^2} \|\phi_z\| \\
&\qquad + \|\nabla_x \tilde v\|_{L^\infty} \|S_z\| + \|\mathbf{\tilde u}_{x_1}\|_{L^\infty} \|(vF)_z\| + \|\phi_z\|_{H^1} \|\Delta_x \psi\|_{H^1} + \|\phi_z\|_{H^1} \|\nabla_x \operatorname{div}_x \psi\|_{H^1} \Big) \\[1mm]
&\le C(\delta^2 + \varepsilon) \big(\|\psi_z\|_{H^1} + \|\phi_z\|_{H^1} + \|\nabla_x  \phi \| + \delta^2 \mathcal{G}^s + \|\nabla_x^2 \psi\|_{H^1}\big).
\end{aligned}
\eeq
Using Young's inequality with \eqref{eq:estiRT} and Lemma \ref{L: estiF}, we obtain the following:
\beq
\begin{aligned} \nonumber
|J_2| &\le C \|R\| \|\phi_z\| \le C(\varepsilon + \delta^2) \big(\|\phi_z\|^2 + \|\psi_z\|_{H^1}^2 + \|\operatorname{div}_x \psi\|_{H^1}^2 \big), \\[2mm]
|J_3| &\le C\|T\|\|\psi_z\| \le C(\delta^2 + \varepsilon) \big(\|\psi_z\|_{H^1}^2 + \|\phi_z\|_{H^1}^2 + \|\nabla_x \phi \|^2 + \delta^2 \mathcal{G}^s + \|\nabla_x^2 \psi\|_{H^1}^2 \big).
\end{aligned}
\eeq
 Similarly as in \eqref{eq:baro1} yield the following identity.
$$\nabla_x \big(vp'(v)\big) = p'(v)\nabla_x\phi + v\nabla_x S + \gamma^2v^{-\gamma-1} \nabla_x \tilde v \quad$$
Using Lemma~\ref{lem:shock-est}, Young's inequality, the above fact, and the \textit{a priori} assumption $\|\nabla_x\phi\|_{H^1}\leq C\e$, we obtain the following estimate:
\beq
\begin{aligned} \label{eq:addi}
|J_4| &\le \|\phi_z\| \|\nabla_x\big(vp'(v)\big)\|_{H^1} \|\psi_z\|_{H^1} +\|\nabla_x \psi_z\| \|\nabla_x \phi\|_{H^1}\|\psi_z\|_{H^1} + \|\nabla_x \tilde v\|_{L^\infty} \|\psi_z\| \|\nabla_x \psi_z\| \\[2mm]
&\le C(\varepsilon + \delta^2) \big(\|\phi_z\|^2 + \|\psi_z\|_{H^1}^2 \big).
\end{aligned}
\eeq
Finally, since the outflow condition in \eqref{main} ensures that $\psi_z \big|_{x_1 = 0} = 0$, which implies $J_5 = 0$. Therefore, by adding the result of zeroth order estimate, we obtain the desired result.
\end{proof}
\subsection{$H^1$-estimate for $\mathbf{u} - \mathbf{\tilde u}$}
 In order to close the $H^1$ estimate, we need to obtain good terms about the $H^2$-level term for the perturbation $\mathbf{u}$. \medskip
 
\noindent $\bullet$ \textbf{ $L^2$-estimate of $\nabla_x \psi$ :} Since the $\mathbf{u}$-perturbation has a parabolic structure, the corresponding $H^1$-level dissipation is obtained by applying the same relative entropy argument as in the zeroth-order estimate.
\begin{lemma}
Under the \textit{a priori} assumption, there exists a constant \label{L:psiH1}
\(C>0\) independent of \(\varepsilon\), \(\delta\), \(\chi\) and \(T\), such that
for all \(t \in [0,T]\), it holds
\begin{equation}
\begin{aligned}
 &\sup_{0 \le t \le T} \bigl\|(\phi, \psi)(t) \bigr\|^{2}
  + \int_{0}^{t} \bigl(\|\nabla_x\psi\|^{2} + \|(\phi, \psi) \big|_{x_1 = 0} \|^2_{L^2(\mathbb{T}^2)} \bigr)\, ds
\\[2mm]
 & \quad \le C \bigl\|(\phi_0, \psi_0 ) \bigr\|^{2}
 + Ce^{-C\delta \beta} + C\delta \int_{0}^{t} \|\psi_{1x_1x_1} \|^{2} \, ds \\
&\qquad + C\delta \int_0^t \big(\, |\dot X(s)|^2 + \mathcal{G}_3(s) + \delta \, \mathcal{D}(s) \, \big) ds
 + C\int_0^t \mathcal{G}^s(s) \, ds. 
\end{aligned}
\end{equation}
\end{lemma}
\begin{proof}
We can rewrite our perturbed equations as follows:
\beq
\left\{
\begin{aligned} \label{eq:ori}
&\rho\big(\partial_t \phi + \mathbf{u} \cdot \nabla_x \phi \big) - \rho \dot X \tilde v_{x_1} + F \tilde v_{x_1} = \operatorname{div}_x \psi, \\[0.5ex]
&\rho\big(\partial_t \psi + \mathbf{u} \nabla_x \psi \big) + \nabla_x \big(p(v)-p(\tilde v)\big) - \rho \dot X\, \nabla_x \mathbf{\tilde u} + F\, \nabla_x \mathbf{\tilde u} = \mu \Delta_x \psi + (\mu + \lambda) \nabla_x \operatorname{div}_x \psi.
\end{aligned}
\right.
\eeq
Multiplying $\eqref{eq:ori}_1$ by $-\bigl(p(v) - p(\tilde v) \bigr)$ and $\eqref{eq:ori}_2$ by $\psi$, and adding them together, we get the following:
\begin{equation}
\begin{aligned}
&\partial_t \bigl(\rho \eta(V | \tilde V) \bigr) + \mu|\nabla_x \psi|^2 + (\mu + \lambda)|\operatorname{div}_x \psi|^2 + \operatorname{div}_x \bigl(\rho \mathbf{u} \eta(V|\tilde V)\bigr)\\[1mm] \nonumber
&\qquad = -\dot X(t)Y(t) + J_{bd} + J_{hyper} + J_{para},
\end{aligned}
\end{equation}
where
\begin{align}\nonumber
&Y(t) := -\rho \bigl(p'(\tilde v) \phi \tilde v_{x_1} - \psi_1 \tilde u_{x_1} \bigr)\nonumber\\[1mm]
&J_{bd} := - \operatorname{div}_x\Bigl(\bigl(p(v) - p(\tilde v)\bigr)\psi\Bigr) + \operatorname{div}_x\bigl(\mu(\nabla_x \psi) \psi + (\mu + \lambda) (\operatorname{div}_x \psi) \psi \bigr)\nonumber\\[1mm]
&J_{hyper} := Fp'(\tilde v)\phi\tilde v_{x_1} + \sigma_* p(v|\tilde v) \tilde v_{x_1}, \nonumber \quad J_{para}:= - F \psi \tilde u_{x_1}\nonumber.
\end{align}
The derivation  of the above equation is identical to the equation of the zeroth order estimate. Here, $V = (v, \mathbf{u})$ and $\tilde V = (\tilde v, \mathbf{\tilde u})$ are the solution vectors, and  $$\eta(V | \tilde V) = Q(v | \tilde v) + \frac{|\mathbf{u} - \mathbf{\tilde u}|^2}{2}$$ denotes the relative entropy.\\
As a similar way of proof of Lemma \ref{L:psig}, we can control the first term on the right hand side as
\begin{equation}
\begin{aligned} \nonumber
\Big|\dot X(t) \int_{\Omega} Y(t) dx \Big| &\le C |\dot X(t)| \int_{\Omega} \Big| \bigl(\phi (\tilde v)_{x_1} + \psi_1 (\tilde u_1)_{x_1}\bigr) \Big| \, dx \\[2mm]
&\le C|\dot X(t)|\Bigl(\|(\tilde v)_{x_1}^{1/2} \phi\| \|(\tilde v)_{x_1}^{1/2}\| + \|(\tilde u_1)_{x_1}^{1/2} \psi_1\| \| (\tilde u_1)_{x_1}^{1/2}\| \Bigr)\\[2mm]
&\le C |\dot X(t)|\Bigl(\delta^{1/2} \sqrt{\mathcal G^s(t)} + \delta^{3/2}\sqrt{\mathcal D(t)} + \delta \sqrt{\mathcal G_3(t)} \Bigr) \\[1mm]
&\le C {\delta} |\dot X(t)|^2 + C\Bigl(\mathcal G^s(t) + \delta \mathcal G_3(t) + \delta^2 \mathcal D(t) \Bigr)
\end{aligned}
\end{equation}
To treat the hyperbolic and parabolic bad term, we will use the Lemmas~\ref{lem:shock-est}, \ref{L:psig}, and Young's inequality.
\begin{equation}
\begin{aligned} \nonumber
\Big|\int_{\Omega}  J_{hyper}\, dx \Big| &\le C \int_{\Omega} \Bigl( |\phi|^2 (\tilde v)_{x_1} + \phi \psi_1 (\tilde v)_{x_1} + p(v|\tilde v) (\tilde v)_{x_1} \Bigr)\,dx \\[2mm]
&\le C \Bigl(\mathcal G^s(t) + \delta \mathcal G_3 + \delta^2 \mathcal D(t) \Bigr)
\end{aligned}
\end{equation}
The bad term that arises from the parabolic equation can be controlled in a similar way.
\begin{equation}
\begin{aligned} \nonumber
\Big|\int_{\Omega} J_{para}\, dx \Big| &\le C \int_{\Omega} \bigl( |\psi|^2 (\tilde u_1)_{x_1} + \phi \psi_1 (\tilde u_1)_{x_1} \bigr)\,dx 
\le C \bigl(\mathcal G^s(t) + \delta \mathcal G_1(t) + \delta \mathcal G_3(t)+ \delta^2 \mathcal D(t) \bigr)
\end{aligned}
\end{equation}
To control the term of boundary effects $J_{bd}$, observe that due to Lemma \ref{lem:shock-est}, we have the following for given $\beta > 0$:
\beq
\begin{aligned} \label{eq:psmall}
&\|\psi_1 \big|_{x_1 = 0}\|_{L^\infty} = \|\tilde u_1(-\sigma t - X - \beta) - u_-\|_{L^\infty} \le C\delta e^{-C\delta t - C\delta \beta}, \\[2mm]
& \|\big(p(v) - p(\tilde v)\big) \big|_{x_1 = 0}\|_{L^1(\mathbb T^2)} \le C \|\mathbf 1_{\mathbb T^2}\| \|(v - \tilde v) \big|_{x_1 = 0}\|_{L^2(\mathbb T^2)} \le C\big(1 + \|(v - \tilde v) \big|_{x_1 = 0}\|_{L^2(\mathbb T^2)}^2 \big).
\end{aligned}
\eeq

Hence, we can control the boundary term by Young's inequality and the interpolation inequality:
\begin{equation}
\begin{aligned} \label{eq:jbd}
\int_{\Omega} J_{bd}\, dx &= \int_{\mathbb{T}^2} \bigl( (p(v) - p(\tilde v)\bigr) \psi \big|_{x_1 = 0} \, dx' - \int_{\mathbb{T}^2} \bigl(\mu\nabla_x \psi \cdot \psi + (\mu + \lambda) \operatorname{div}_x \psi \psi \bigr) \big|_{x_1 = 0} \, dx' \\[2mm]
&\le \|\big(p(v) - p(\tilde v)\big)\big|_{x_1 = 0} \|_{L^1(\mathbb{T}^2)} \|\psi_1\big|_{x_1 = 0}\|_{L^\infty}\\[1mm]
&\quad + C\,\|\psi_1\big|_{x_1=0}\|_{L^\infty} \|\nabla_x \psi_1 \big|_{x_1 = 0}\|_{L^2(\mathbb{T}^2)} \|\|\mathbf{1}_{\mathbb{T}^2}\|_{L^2(\mathbb{T}^2)}\\[2mm]
&\le C\delta e^{-C\delta t - C \delta \beta} \big(1 + \|\phi\big|_{x_1 = 0} \|_{L^2(\mathbb T^2)}^2 + \|\psi_{1x_1} \big|_{x_1 = 0}\|_{L^2(\mathbb{T}^2)}^2 \big) \\[2mm]
&\le C\delta e^{-C\delta t - C \delta \beta} + C\delta \big(\, \|\phi\big|_{x_1 = 0} \|_{L^2(\mathbb T^2)}^2 + \|\psi_{1x_1}\|^2 + \|\psi_{1x_1x_1} \|^2\big),
\end{aligned}
\end{equation}
where the last line is derived from the following fact, which can be derived by integration by parts.
\beq
\begin{aligned} \label{eq:treatbd}
|\psi_{1x_1}\big|_{x_1 =0}|^2 = -2 \int_{0}^\infty \psi_{1x_1} \psi_{1x_1x_1} \, dx_1 \le \|\psi_{1x_1}\|_{L^2(\mathbb{R}^+)}^2 + \|\psi_{1x_1x_1}\|_{L^2(\mathbb{R}^+)}^2
\end{aligned}
\eeq
\end{proof}
\noindent $\bullet$ \textbf{$L^2$-estimate of $\nabla_x^2 \psi$ :}
 Since we have already obtained the good term for $\nabla_x \psi_{z}$ in Lemma \ref{L:bdesti}, it remains only to estimate a $L^2$-norm of $\psi_{x_1 x_1}$.
\begin{lemma}Under the \textit{a priori} assumption, there exists a constant\label{L:psiH2}
\(C>0\) independent of \(\varepsilon\), \(\delta\), \(\chi\) and \(T\), such that
for all \(t \in [0,T]\), it holds
\begin{equation}
\begin{aligned}\nonumber
&\sup_{0 \le t \le T} \|(\phi,\psi)(t)\|_{H^{1}}^{2}
+ \delta \int_{0}^{t} \big| \dot X(s) \big|^2 ds
+ \int_{0}^{t} \Bigl(
\mathcal{G}^{s}(s)+\mathcal{D}(s)+\mathcal{G}_{3}(s) +\|\psi_t\|^{2}\\
&\hspace{73mm} +\|\nabla_x\psi\|_{H^1}^{2}
+\|\nabla_{x'} \phi\big|_{x_1=0}\|_{L^2(\mathbb{T}^2)}^{2}
\Bigr)\,ds \nonumber\\[1mm]
&\quad\le C \|(\phi_0,\psi_0)\|_{H^{1}}^{2}
+Ce^{-C\delta\beta}
+C(\varepsilon + \delta^2 + \chi)\int_{0}^{t}
\|\nabla^{2}\operatorname{div}_x\psi\|^{2}
\,ds
\end{aligned}
\end{equation}
\end{lemma}
\begin{proof}
By virtue of the lemma \ref{L:bdesti}, it suffices to obtain a good term involving $\psi_{x_1x_1}$. Furthermore, we also need to obtain a good term for $\psi_t$. To do this, we performed the following computations.
\begin{itemize}
\item[(1)] \textbf{$L^2$-estimate of $\psi_t$:}
\end{itemize}
Multiplying $(\ref{eq:ori})_2$ by $\psi_t$, we get the following equation.
\begin{equation}
\begin{aligned}\nonumber
\rho\, |\psi_t|^2
  + \left(\rho\, \mathbf{u} \nabla_x \psi \right) \cdot \psi_t
  + \nabla_x\!\bigl(p(v)-p((\tilde v))\bigr)\cdot \psi_t
  - \rho\, \dot X(t)\,\tilde{\mathbf u}_{x_1}\cdot \psi_t
  + F\,\tilde{\mathbf u}_{x_1}\cdot \psi_t &\\[1mm]
= \mu\, \Delta_x \psi \cdot \psi_t
  + (\mu+\lambda)\, \nabla_x \operatorname{div}_x \psi \cdot \psi_t&
\end{aligned}
\end{equation}
Hence, by using integration by parts, we get
\begin{align}\nonumber
&\int_{\Omega} \rho |\psi_t|^2 dx+ \frac{(\mu+\lambda)}{2}\frac{d}{dt}\int_{\Omega} |\operatorname{div}_x \psi|^2 dx + \frac{\mu}{2}\frac{d}{dt} \int_{\Omega} |\nabla_x \psi|^2 dx \\[2mm]
\nonumber&\quad=\underbrace{\int_{\Omega} \Big( \mu \operatorname{div}_x(\nabla_x \psi \cdot \psi_t) + (\mu+\lambda)\operatorname{div}_x(\psi_t \cdot \operatorname{div}_x \psi)\Big) dx}_{J_1}
 - \underbrace{\int_{\Omega}(\rho \mathbf{u}\nabla_x \psi) \cdot \psi_t dx}_{J_2}\\
\nonumber&\qquad
 - \underbrace{\int_{\Omega}\nabla_x\bigl(p(v)-p(\tilde v)\bigr) \cdot \psi_t dx}_{J_3} 
+ \underbrace{\dot X(t) \int_{\Omega} \, \tilde{\mathbf u}_{x_1} \cdot \psi_t dx}_{J_4} -\underbrace{\int_{\Omega} F \tilde{\mathbf u}_{x_1} \cdot \psi_t dx.}_{J_5}
\end{align}
First, from the fact that $\mathbf{u}$ is a constant vector in time at $\{x_1 = 0\}$, and $|\sigma + \dot X| \le C$ for some constant $C > 0$, given $\beta > 0$, the following holds for $\psi_t$ on the boundary.
\beq
\begin{aligned} \label{eq:convs}
\| \psi_t \big|_{x_1 = 0}\|_{L^\infty} &= \big\|\big(\mathbf{u}_t - \mathbf{\tilde u}(x_1 - \sigma t - X -\beta)_t\big)\big|_{x_1=0} \big\|_{L^\infty} \\[3mm]
&= |\sigma + \dot X|\, \|\mathbf{\tilde u}_{x_1} \big|_{x_1 =0}\|_{L^\infty} \le C\delta^2 e^{-C\delta t - C\delta \beta}.
\end{aligned}
\eeq
Similarly as in \eqref{eq:jbd} with \eqref{eq:convs}, we obtain the following.
\begin{equation}
\begin{aligned}\nonumber
|J_1| \le (2\mu + \lambda) \int_{\mathbb{T}^2} \big|\psi_{1x_1}\psi_{1t}\big| \Big|_{x_1 = 0} dx' 
&\le C \|\psi_{1t}\|_{L^\infty} \|\mathbf{1}_{\mathbb{T}^2}\|_{L^2(\mathbb{T}^2)}\|\psi_{1x_1}\big|_{x_1 = 0}\|_{L^2(\mathbb{T}^2)} \\[2mm]
&\le C\delta^2 e^{-C\delta t - C\delta \beta} + C\delta^2 \|\psi_{1x_1} \big|_{x_1 = 0}\|_{L^2(\mathbb{T}^2)}^2
\end{aligned}
\end{equation}
For $J_2$, using Young's inequality, we can get 
$$ |J_2| \le C \int_{\Omega} \big| (\nabla_x \psi)\, \psi_t \big| \, dx
\le C\chi\|\psi_t\|^2 + C_{\chi} \|\nabla_x \psi \|^2. $$
Also, as in above, we get
$$|J_3| \le C_{\chi} \mathcal{D} + C\chi \|\psi_t\|^2. $$
Using Lemma~\ref{lem:shock-est}, we obtain the following estimate.
\begin{equation}
\begin{aligned}\nonumber
|J_4| &\le C |\dot X(t)| \int_{\Omega} \,|(\tilde u)_{x_1}||\psi_t| dx 
\le C\,\|\tu_{x_1}\|\big( |\dot X(t)|^2+\|\psi_t\|^2\big)
\le C\delta^{3/2} \big(|\dot X(t)|^2 + \|\psi_t\|^2\big)
\end{aligned}
\end{equation}
Finally, Lemma \ref{L:psig} yields the following inequality for $J_5$.
\begin{align*}
|J_5|
\le C \|vF(\tilde u)_{x_1}\| \|\psi_t\|
&\le C\|({\tilde u}_{x_1})^{1/2}\|_{L^\infty} \|vF({\tilde u})_{x_1}^{1/2} \| \|\psi_t\| \\[2mm]
&\le C\delta\big(\mathcal{G}^s(t) + \delta \, \mathcal{G}_3(t) + \delta^2 \, \mathcal{D}(t) + \|\psi_t\|^2 \big)
\end{align*}
Collecting all these estimates, we obtain the following bound.
\beq
\begin{aligned}
\sum_{i=1}^5 |J_i| \le C(\delta, \chi) \Big(\|\psi_{1x_1} \big|_{x_1 = 0}\|_{L^2(\mathbb{T}^2)}^2 + \mathcal{G}_3 + \mathcal{G}^s + \delta|\dot X|^2 + \|\psi_t\|^2\Big) + C\,\Big(\|\nabla_x \psi\|^2 + \mathcal{D} \Big)\label{eq:18}
\end{aligned}
\eeq

\begin{itemize}
\item[(2)] \textbf{$L^2$-estiamte of $\psi'_{x_1x_1}$:}
\end{itemize}

Rewriting the second and third components of $\eqref{eq:peq}_2$, we have the following:
\begin{equation}
\begin{aligned}\nonumber
\psi_t'
+ \mathbf{u} \nabla_x\psi'
+\bigl(p(v)-p(\tilde v)\bigr)_{z}
- (\mu+\lambda)v\, \operatorname{div}_x\psi_{z}
- \mu v\,\Delta_{x'}\psi'
= \mu v\,\psi'_{x_1x_1}.
\end{aligned}
\end{equation}
Applying $L^2$-Minkowski's inequality to the above equation, we obtain the following estimate.
\begin{align}
\int_{0}^{t}\|\psi'_{x_{1}x_{1}}\|^{2}\,ds
\le
C\int_{0}^{t}\Bigl(\|\psi'_{t}\|^{2}+\|\nabla_{x}\psi'\|^{2}+\mathcal{D}(s)
+\|\nabla_{x}\psi_{z}\|^{2}\Bigr)\,ds \label{eq:19}
\end{align}

\begin{itemize}
\item[(3)] \textbf{$L^2$-estimate of $\psi_{1x_1x_1}$:}
\end{itemize}

Sobolev norms of the shock profiles $\tilde{v}$ and $\tilde{u}_1$ do not exhibit temporal decay. Therefore, to estimate $\psi_{1 x_1 x_1}$, we exploit Young's inequality instead of the $L^2$-Minkowski inequality.\\
The first component of the momentum equation $\eqref{eq:peq}_1$ can be rewritten as follows:
\begin{equation}
\begin{aligned}\nonumber
&\psi_{1t}
+ \mathbf{u} \cdot \nabla_x\psi_1
+\bigl(p(v)-p(\tilde v)\bigr)_{x_1}
- v \big((\mu+\lambda)\operatorname{div}_{x'}\psi_{x_1}'
- \mu \,\Delta_{x'}\psi_1 \big)+ (vF) \tilde u_{1x_1} + \dot X \tilde u_{1x_1}\\
&\quad = (2\mu + \lambda) v\,\psi_{1x_1x_1}.
\end{aligned}
\end{equation}
Multiplying both sides by $\psi_{1x_1x_1}$, integrating over the spatial domain, and applying Lemma~\ref{L: estiF} and Young's inequality, we obtain
\beq
\begin{aligned}\nonumber
(2\mu + \lambda)\int_{\Omega} v|\psi_{1x_1x_1}|^2 \, dx 
&\le C_{\chi}\big(\|\psi_{1t}\|^2 + \|\nabla_x \psi_1\|^2 + \mathcal{D} + \|\nabla_x \psi_{z}\|^2\big)\\
&\quad + \|(vF)(\tilde u_{x_1})\|\|\psi_{1x_1x_1}\| + \dot X \|\tilde u_{x_1}\|\|\psi_{1x_1x_1}\| + \chi \|\psi_{1x_1x_1}\|^2 \\[3mm]
&\le C_{\chi} \big(\|\psi_{1t}\|^2 + \|\nabla_x \psi_1\|^2 + \mathcal{D} + \|\nabla_x \psi_{z}\|^2 \big) \\[2mm]
&\quad + C\delta \big(\mathcal{G}^s + \mathcal{G}_3 + \mathcal{D} + \delta |\dot X|^2 \big) + (\chi + \delta) \|\psi_{1x_1x_1}\|^2.
\end{aligned}
\eeq
Consequently, choosing $\chi$ and $\delta$ to be sufficiently small yields the following estimate.
\begin{align}\label{eq:new}
\|\psi_{1x_1x_1}\|^2 \le C\big(\|\psi_{1t}\|^2 + \|\nabla_x \psi_1\|^2 + \|\nabla_x \psi_{z}\|^2 + \mathcal{D} \big) + C\delta \big(\,\mathcal{G}_3 + \mathcal{G}^s\,\big) + C\delta^2 |\dot X|^2.
\end{align}
Combining the results of Lemma \ref{L:bdesti} and \ref{L:psiH1} with 
\eqref{eq:18}, \eqref{eq:19} and \eqref{eq:new}, we obtain the conclusion of Lemma \ref{L:psiH2}.
\end{proof}

\subsection{Time-derivative estimates for recovering the highest normal derivatives} 

As discussed earlier, to estimate the normal derivatives, we first derive estimates for the time derivatives. To this end, we begin by estimating the second order derivative of the shift $X$.
\begin{lemma} Under the \textit{a priori} assumption, for sufficiently small $\delta > 0$, the following estimate holds. \label{L:21}
$$|\ddot X(t)|^2 \le C \big(|\dot X(t)|^2 + \delta \|\phi_t\|^2 + \delta^{3} \mathcal{G}^s(t) \big)$$
Here, $\phi_t$ satisfies the following estimate for given $\delta > 0$.
\begin{align}
\|\phi_t\|^2 \le C \Big(\mathcal{D}(t) + \|\operatorname{div}_x \psi\|^2 + \delta^2 |\dot X(t)|^2 + \delta\, \mathcal{G}^s(t) + \delta\, \mathcal{G}_3(t) \Big) \label{eq:7}
\end{align}
Therefore, $\ddot{X}(t)$ is bounded and belongs to $L^2$, which means that $\dot{X}(t)$ is absolute continuous.
\end{lemma}
\begin{proof}
Let the right-hand side of \eqref{X(t)} be $G(X, t)$. Then $\ddot X(t) = \dot X G_X + G_t$. Here, note that this $G_t$ is different from the standard definition of time derivative; this term is a time derivative when $X$ is fixed.

By \cite{KV21}, $G_X$ is uniformly bounded since $G$ is Lipschitz with respect to $X$. $G_t$ can be calculated as
\begin{equation}
\begin{aligned}\nonumber
-\frac{\delta}{M}G_t = \int_{\Omega}&\dfrac{a^{-X}(x_1)}{\sigma_*}\,\rho_t(\tilde h_1)_{x_1}
\Big(p(v)-p(\tilde v)\Big)\,dx
+\int_{\Omega}\dfrac{a^{-X}(x_1)}{\sigma_*}\,\rho(\tilde h_1)_{x_1}
\Big(p(v)-p(\tilde v)\Big)_t\,dx
\\[1mm]
&-\displaystyle \int_{\Omega}
a^{-X}(x_1)\,\rho_t\,p'(\tilde v)\,
(v-\tilde v)\,\tilde v_{x_1}\,dx
-\displaystyle \int_{\Omega}
a^{-X}(x_1)\,\rho\,p'(\tilde v)\,
(v-\tilde v)_t\,\tilde v_{x_1}\,dx
\end{aligned}
\end{equation} 
By the Lipschitz continuity of $p$ and Lemma \ref{lem:shock-est}, the following holds for sufficiently small $\delta > 0$. 
$$\big(p(v) - p\bigl(\tilde v\bigr)\big)_t \le C\big((v - \tilde v)_t + (v - \tilde v)(\tilde v)_t \big), \quad \tilde h_{1x_1} \le C( \tilde v_{x_1} + \tu_{x_1x_1}) \le C\,\tv_{x_1}$$ 
Furthermore, by \textit{a priori} smallness and Lemma \ref{lem:shock-est}, the following uniform estimate holds for $\mathcal{G}_s$
\beq
\begin{aligned}\nonumber
\mathcal{G}^s(t) &= \int_{\Omega}|\tilde v_{x_1}|\big(p(v) - p(\tilde v)\big)^2 \, dx \le \|\tilde v_{x_1}\|_{L^\infty} \|\big(p(v) - p(\tilde v)\big)\|^2 \le C\delta^2 \varepsilon^2
\end{aligned}
\eeq
Using the above facts, we can obtain the following.
\begin{equation}
\begin{aligned}\nonumber
\frac{\delta}{M}|G_t| \le & C \int_{\Omega} a^{-X}\Big(|(\tilde h_1)_{x_1} -\tilde v_{x_1}|\big(|v_t||\phi| + |\phi_t| + \phi|\tilde v_t| \big) \Big)\, dx \\[2mm]
\le& C \int_{\Omega} a^{-X}|\tilde v_{x_1}|\bigl(|\phi||\phi_t| + |\phi_t| + |\phi||\tilde v_{x_1}| \bigr) \, dx \\[2mm]
\le & C \big( \|\tilde v_{x_1}^{1/2}\|_{L^\infty} \|\tilde v_{x_1}^{1/2}\phi\| \|\phi_t\| + \|\tilde v_{x_1}\| \|\phi_t\| + \|\tilde v_{x_1}^{3/2}\| \|\tilde v_{x_1}^{1/2} \phi\| \big)\\[2mm]
\le & C\big(\delta \sqrt{\mathcal{G}^s}\|\phi_t\| + \delta^{3/2} \|\phi_t\| + \delta^{5/2} \sqrt{\mathcal{G}^s} \big) 
\le  C\delta \big(\delta^{1/2} \|\phi_t\| + \delta^{3/2} \sqrt{\mathcal{G}^s} \big)
\end{aligned}
\end{equation}
Note that by \eqref{eq:convs}, the time derivative of the shock can be converted into its spatial derivative up to a constant factor.

On the other hand, inequality \eqref{eq:7} is derived by exactly the same manner as \eqref{eq:new}. Note that the mass equation can be written as follows.
$$ \partial_t \phi = v \operatorname{div}_x \psi - \mathbf{u} \cdot \nabla_x \phi - \dot X (\tilde v)_{x_1} - (vF)(\tilde v)_{x_1}$$
Multiplying $\phi_t$ on the both sides of the above equation and using the Young's inequality, we obtain the desired estimate.
\end{proof}

To control the third-order spatial derivatives of \(\psi\), we first estimate its mixed space-time derivatives as follows.
\begin{lemma}Under the \textit{a priori} assumption, there exists a constant \label{L:22}
\(C>0\) independent of \(\varepsilon\), \(\delta\), \(\chi\) and \(T\), such that
for all \(t \in [0,T]\), it holds
\begin{equation}
\begin{aligned}\nonumber
&\sup_{0 \le t \le T} \|\partial_{t} (\phi, \psi)(t) \|^2 + \int_0^t \bigl( \|\nabla_x \psi_{t} \|^2 + \|\phi_{t}\big|_{x_1=0}\|_{L^2(\mathbb{T}^2)}^2 \bigr) \,ds \\
&\qquad \le C \|\partial_{t} (\phi, \psi)(0) \|^2 + C \delta^2 e^{-C\delta \beta} + C\delta^2\int_0^t |\dot X(s) |^2 \,ds \\
&\qquad \quad+ C(\varepsilon + \delta^2) \int_0^t \Bigl( \|\phi_{x_1}\big|_{x_1=0}\|_{L^2(\mathbb{T}^2)}^2  + \mathcal{G}_s + \mathcal{G}_3 + \mathcal{D}(s) + \|\phi_t\|^2 
+\|\nabla_x\psi\|_{H^2}^2 + \|\nabla_x^2 \phi\|^2 \Bigr) \,ds
\end{aligned}
\end{equation}
\end{lemma}
\begin{proof}
Note that the following relation holds between the time and spatial derivatives of the shock.
$$\tilde v_t = \partial_t \big(\tilde v(x_1 - \sigma t - X - \beta)\big) = -(\sigma + \dot X) \tilde v_{x_1}$$
We get the following equation by differentiating \eqref{eq:peq} with respect to $t$. 
\begin{equation}
\left\{
\begin{aligned}
&\phi_{tt} +\mathbf u\cdot\nabla\phi_t \label{eq:8}
=
v\,\operatorname{div}\psi_t + \ddot X \tilde v_{x_1} -(\sigma + \dot X) \dot X \tilde v_{x_1x_1} + R
\\[1mm]
&\psi_{tt} +\mathbf{u}\,\nabla\psi_t + vp'(v)\nabla\phi_t 
= v\big(\mu\Delta\psi_t+(\mu+\lambda)\nabla\operatorname{div}\psi_t\big) + \ddot X \mathbf{\tilde u}_{x_1} - (\sigma + \dot X) \dot X \mathbf{\tilde u}_{x_1x_1} + T
\end{aligned}
\right.
\end{equation}
The terms $R$ and $T$ in this lemma are defined as follows. 
\begin{align}
R
&= v_t\,\operatorname{div}\psi - \mathbf{u}_t\cdot\nabla_x\phi
- (vF)_t\,(\tilde v)_{x_1}
+ (vF)(\sigma+\dot X)\,(\tilde v)_{x_1x_1} \nonumber \\[1mm]
T
&= \mathbf{u}_t\,\nabla_x\psi - \big(vp'(v)\big)_t\,\nabla_x\phi 
-(vS)_t\,\nabla_x \tilde v + v_t\big(\mu\Delta_x \psi + (\mu+\lambda)\nabla_x \operatorname{div}_x \psi\big)\\
&\qquad +(\sigma+\dot X)(vS)\,\nabla_x (\tilde v)_{x_1}
- (vF)_t\,(\mathbf{\tilde u})_{x_1}
+ (vF)(\sigma+\dot X)\,(\mathbf{\tilde u})_{x_1x_1} \nonumber
\end{align}
By Lemma \ref{L:21}, the dynamical shift is in $H^{2}$; hence, the above computation is well defined.
In the same way, multiplying $\eqref{eq:8}_1$ by $-p'(v)\phi_t$, $\eqref{eq:8}_2$ by $\psi_t$, and adding them, we get the below equation.
\begin{align}
&\frac{d}{dt}
\int_{\Omega}\frac12
\Bigl(|\psi_t|^2-p'(v)\phi_t^2\Bigr)\,dx 
-u_-\int_{\mathbb{T}^2}
\frac12\Bigl(|\psi_t|^2-p'(v_-)\phi_t^2\Bigr)\Big|_{x_1=0}\,dx'\quad
\nonumber\\
&\qquad +\int_{\Omega}v
\Bigl(\mu|\nabla_x\psi_t|^2+(\mu+\lambda)|\operatorname{div}_x\psi_t|^2\Bigr)\,dx
=\ddot X\,Y_1+\dot X\,Y_2+J_{\mathrm{bd}}+\sum_{k=1}^{4}J_k .\nonumber
\end{align}
where symbols are defined as; 
\beq
\begin{aligned}
Y_1 &:= -\int_{\Omega} \Bigl(p'(v)\,(\tilde v)_{x_1}\,\phi_t-(\tilde{\mathbf u})_{x_1}\cdot\psi_t\Bigr) \, dx \qquad 
Y_2 := \int_{\Omega} \Big( (\sigma+\dot X)\bigl(p'(v)\,\phi_t\,(\tilde v)_{x_1x_1}+\psi_t (\tilde{\mathbf u})_{x_1x_1}\bigr) \, dx \nonumber\\[0.5ex]
J_{\mathrm{bd}} &:= -\int_{\mathbb{T}^2} \big((\mu + \lambda)v (\operatorname{div}_x \psi_t) \psi_{1t} + \mu v\,\psi_t\cdot \nabla_x \psi_{1t} - p'(v) \phi_t \psi_{1t}\big) \Big|_{x_1 = 0}  \nonumber\\[0.5ex]
J_1 &:= \frac12\int_{\Omega} \operatorname{div}_x \mathbf{u}\Big(-vp''(v)|\phi_t|^2 + p'(v)|\phi_t|^2 - |\psi_t|^2 \Big) \, dx \\[0.5ex]
J_2 &:=  \int_{\Omega}\Big(\phi_t \nabla_x \big(vp'(v)\big) \cdot \psi_t + \nabla_x v \cdot \big((\mu + \lambda) \operatorname{div}_x \psi_t \psi_t + \mu \psi_t \nabla_x \psi_t \big) \Big) \, dx \\[0.5ex]
J_3 &:= \int_{\Omega} -p'(v)R \phi_t \, dx \qquad
J_4 := \int_{\Omega} T\cdot \psi_t \, dx
\end{aligned}
\eeq
We omit the derivation since it is identical to that of \eqref{eq:ref}, merely replacing the $z$-derivative with the $t$-derivative.

Using Young's inequality and the Lemma~\ref{lem:shock-est}, we obtain 
\begin{equation}
\begin{aligned}\nonumber
\big|\ddot X Y_1 \big| &\le C|\ddot X(t)|\int_{\Omega} \Bigl(\big|p'(v)\,(\tilde v)_{x_1}\,\phi_t\big| + \big| (\tilde{\mathbf u})_{x_1}\cdot\psi_t \big| \Bigr) \, dx 
\le C\delta^2\Big(|\ddot X(t) |^2 + \|\phi_t\|^2 + \|\psi_t\|^2 \Big)
\end{aligned}
\end{equation}
To control $\dot X Y_2$, we use same manner as in the above with the fact $|(\sigma + \dot X)| \le C$.
\beq
\begin{aligned}
\big|\dot X Y_2 \big| &\le C\,|\dot X(t)|\,\int_{\Omega} \Bigl(\big|p'(v)\,(\tilde v)_{x_1x_1}\,\phi_t\big| + \big| (\tilde{\mathbf u})_{x_1x_1}\cdot\psi_t \big|\Big)\,dx \nonumber\\[1mm]
& \le C\,|\dot X(t)|\,\int_{\Omega} \|\tilde v_{x_1x_1}\| \bigl(\|\phi_t\big\| + \| \psi_t \|\big) \, dx
 \nonumber\\[1mm]
& \le C(\delta^{5/2})\Big(|\dot X|^2 + \|\phi_t\|^2 + \|\psi_t\|^2 \Big) 
\end{aligned}
\eeq
To control $J_1$, we use the same argument as in \eqref{eq:div}, together with the fact $\|\phi_t\|,\,\|\psi_t\|\le C\e$ which can be derived from the \textit{a priori} assumption and Lemmas 6.6, 6.7. 
\begin{equation}
\begin{aligned}
|J_1| \le C\big(\|\operatorname{div}_x \psi\|_{L^{\infty}} + \delta^2 \big) \big(\|\phi_t\|^2 + \|\psi_t\|^2 \big)
 \le C(\varepsilon + \delta) \big( \|\phi_t \|^2 + \|\psi_t\|^2 +\|\nabla_x^2 \operatorname{div}_x\psi\|^2 \big) \nonumber
\end{aligned}
\end{equation}
We obtain the estimate for $J_2$ in the same manner as in \eqref{eq:addi}.
$$|J_2| \le C(\varepsilon + \delta^2)\big(\|\phi_t\|^2 + \|\psi_t\|_{H^1}^2 \big)$$
To estimate $J_3$ and $J_4$, we first estimate the $L^2$-norms of $R$ and $T$. To do this, we first note that the following fact arises from the barotropic condition.
$$\big(vp'(v)\big)_t = p'(v) v_t + vp''(v) v_t  = \gamma^2 v^{-\gamma -1} v_t.$$
Using this, we rewrite $R$ and $T$ as follows.
\beq
\begin{aligned}
R
&= \phi_t\,\operatorname{div}_x\psi - {\psi}_t\cdot\nabla_x\phi + (\sigma + \dot X)(\operatorname{div}_x\psi)\tilde v_{x_1} \\
&\hspace{10mm}- (\sigma + \dot X) \nabla_x\phi \cdot \mathbf{\tilde u}_{x_1} 
 - (vF)_t\,(\tilde v)_{x_1}
+ (vF)(\sigma+\dot X)\,(\tilde v)_{x_1x_1}  \nonumber \\[2mm]
T
&=\psi_t\,\nabla_x\psi - (\sigma+\dot X)\mathbf \tu_{x_1} \nabla_x \psi - (\gamma^2v^{-\gamma -1}) \phi_t \,\nabla_x\phi - (\sigma + \dot X)(\gamma^2v^{-\gamma -1}) \,\nabla_x\phi \, \tilde v_{x_1} \\[2mm]
&\hspace{10mm}
+\phi_t\big(\mu\Delta_x \psi+(\mu+\lambda) \nabla_x\operatorname{div}_x \psi \big) -(\sigma+\dot X)\tv_{x_1} \big(\mu \Delta_x \psi +(\mu+\lambda)\nabla_x \operatorname{div}_x\psi \big) \\[2mm]
&\hspace{10mm}
-(vS)_t\,\nabla_x \tilde v
+(\sigma+\dot X)(vS)\,\nabla_x (\tilde v)_{x_1} 
- (vF)_t\,(\mathbf{\tilde u})_{x_1}
+ (vF)(\sigma+\dot X)\,(\mathbf{\tilde u})_{x_1x_1}  \nonumber
\end{aligned}
\eeq
 Using Lemmas \ref{lem:shock-est}, \ref{GNIsevera}, \ref{L: estiF}, \ref{L:psig}, and the fact that $|\sigma + \dot X| \le C$, $(vS) \sim \big(p(v) - p(\tilde v)\big)$, we obtain the following $L^2$ bounds.
\beq
\begin{aligned} \nonumber
\|R\| &\le  C\big(\|\operatorname{div}_x \psi\|_{H^1} \|\phi_t\|_{H^1} + \|\psi_t \|_{H^1} \|\nabla_x \phi\|_{H^1} + \|\tilde v_{x_1}\|_{L^\infty} \|\operatorname{div}_x \psi \| \\[1mm]
&\qquad + \|\mathbf{\tilde u}_{x_1}\|_{L^\infty} \|\nabla_x \phi\| + \|\tilde v\|_{L^\infty} \|(vF)_t \| + \|(vF) \tilde v_{x_1x_1}\|  \big)\\[2mm]
&\le C(\varepsilon + \delta^2) \big(\|\phi_t\|_{H^1} + \|\psi_t\|_{H^1} + \|\operatorname{div}_x \psi\| + \|\nabla_x \phi\| + (\mathcal{G}^s + \delta \mathcal{G}_3 + \delta^2\mathcal{D})^{1/2} \big)
\end{aligned}
\eeq
and,
\beq
\begin{aligned} \nonumber
\|T\|
&\le C \big(\|\nabla_x \psi\|_{H^1} \|\psi_t\|_{H^1} + \|\mathbf \tu_{x_1} \|_{L^\infty} \|\nabla_x \psi\| + \|\phi_t\|_{H^1} \|\nabla_x \phi\|_{H^1} + \|\tilde v_{x_1}\|_{L^\infty} \|\nabla_x \phi\|\\[1mm]
&\qquad + \|\tv_{x_1}\|_{L^\infty} \|\nabla_x^2 \psi\| + \|\nabla_x \tilde v\|_{L^\infty} \|(vS)_t \| + \|(vS) \nabla_x \tilde v_{x_1}\| + \|\mathbf{\tilde u}_{x_1} \|_{L^\infty} \|(vF)_t\| + \|(vF) \mathbf{\tilde u}_{x_1x_1}\| \big) \\[1mm]
&\qquad +\phi_t\big(\mu\Delta_x \psi + (\mu+\lambda)\nabla_x \operatorname{div}_x \psi\big)\\[3mm]
&\le C(\varepsilon + \delta^2) \big(\|\psi_t\|_{H^1} + \|\phi_t\| +\|\nabla_x \phi\|_{H^1} + \|(vS)_t \|  + (\mathcal{G}^s + \delta \mathcal{G}_3 + \delta^2 \mathcal{D})^{1/2} \big)\\[1mm]
&\qquad +\phi_t\big(\mu\Delta_x \psi + (\mu+\lambda)\nabla_x \operatorname{div}_x \psi\big)
\end{aligned}
\eeq
Since $(vS)_t = \phi_t S + (\sigma + \dot X) \tilde v_{x_1}S + v S_t$, using the relation $S \sim \big(p(v) - p(\tilde v)\big)$ which arise from Lipschitz continuity of $p$ and $p'$, we obtain the following bound.
$$\|(vS)_t\| \le \|\phi_t\| + \mathcal{D} + \|\psi_t\|.$$
Therefore, combining all of the above, we can estimate $J_3$ and $J_4$ as follows.
\beq
\begin{aligned} \nonumber
|J_3| 
&\le C(\varepsilon + \delta^2) \big(\|\phi_t\|^2 + \|\nabla_x^2 \operatorname{div} \psi\|^2 + \|\psi_t\|_{H^1}^2 + \|\operatorname{div}_x \psi\|^2 + \mathcal{D} + (\mathcal{G}^s + \delta \mathcal{G}_3) \big) \\[3mm]
|J_4| &\le C(\varepsilon + \delta^2) \|\psi_t\|\big(\|\psi_t\|_{H^1} + \|\phi_t\| +\|\nabla_x \phi\|_{H^1} + \|(vS)_t \|  + (\mathcal{G}^s + \delta \mathcal{G}_3 + \delta^2 \mathcal{D})^{1/2} \big) \\[1mm]
&\qquad +\, C\,\|\phi_t\| \|\psi_t\|_{L^3}\|\nabla_x^2 \psi\|_{L^6} \\[2mm]
&\le C(\varepsilon + \delta^2) \big(\|\psi_t\|_{H^1}^2 + \|\phi_t\|^2 + \|\nabla_x^2 \psi\|_{H^1}^2 +\|\nabla_x \phi\|_{H^1}^2 + (\mathcal{G}^s + \delta \mathcal{G}_3) \big)
\end{aligned}
\eeq
Finally, we should control the bad terms arising from the presence of the boundary. Before calculating, note that
\begin{align}
J_{\mathrm{bd}} &= -(2\mu+\lambda) \int_{\mathbb{T}^2} \bigl(v\,\psi_{1t}\,\psi_{1x_1t}\bigr) \Big|_{x_1 = 0}\, dx'
+ \int_{\mathbb{T}^2} \bigl(p'(v)\,\psi_t\phi_t\bigr) \Big|_{x_1 = 0} \, dx' := J_{11} + J_{12} \nonumber
\end{align}
By Young's inequality, together with \eqref{eq:convs}, we can control $J_{12}$ as follows.
$$|J_{12}| \le C\|\psi_t \big|_{x_1 = 0}\|_{L^\infty}\big(\|\mathbf{1}_{\mathbb{T}^2}\|_{L^2(\mathbb{T}^2)}^2 + \|\phi_t \big|_{x_1 =0 }\|_{L^2(\mathbb{T}^2)}^2\big) \le C\delta^2 e^{-C\delta t - C\delta \beta} + C\delta^2 \|\phi_t\big|_{x_1 = 0}\|_{L^2(\mathbb{T}^2)}^2$$  
To estimate \(J_{11}\), we use the fact that \(z\)- and \(t\)-derivatives are tangential to the boundary \(\{x_1=0\}\). Hence, time derivatives can be transferred by integration by parts directly on the boundary.
\begin{align*}
\nonumber\int_0^t J_{11}\, dt & \le C \int_{\mathbb{T}^2} \Big(|(\psi_{1x_1}\, \psi_{1t})(t)\big|_{x_1 = 0}| + |(\psi_{1x_1}\, \psi_{1t})(0) \big|_{x_1 = 0}|\Big) \, dx' + C \int^t_0 \int_{\mathbb{T}^2} (\psi_{1x_1}\, \psi_{1tt}) \big|_{x_1 = 0} \, dx' dt \\[2mm]
& \quad + \,C \int_0^t\,\int_\mathbb {T^2} \big( \phi_t \psi_{1x_1} \psi_{1t} -(\sigma + \dot X)\, \tv_{x_1} \psi_{1x_1} \psi_{1t} \big)\, dx'\\[2mm]
\nonumber&\le C\int_{\mathbb{T}^2} \delta^2 e^{-C \delta \beta} \big(|\psi_{1x_1} \big|_{x_1 = 0}(t)| + |\psi_{1x_1}(0) \big|_{x_1 = 0}| \big)\, dx' \\[1mm]
\nonumber &\quad + C\int_0^t \|\psi_{1tt}\big|_{x_1 = 0}\|_{L^\infty} \big(\|\mathbf{1}_{\mathbb T^2}\|_{L^2(\mathbb{T}^2)}^2 + \|\psi_{1x_1} \big|_{x_1 = 0}\|_{L^2(\mathbb{T}^2)}^2 \big)\, dt \\[1mm]
\nonumber &\quad + C\,\int_0^t \|\phi_t\big|_{x_1 = 0}\|_{L^2(\mathbb T^2)}\|\psi_{1x_1}\big|_{x_1 = 0}\|_{L^2(\mathbb T^2)} \|\psi_{1t}\big|_{x_1 = 0}\|_{L^\infty} dt\\[1mm]
\nonumber &\quad +\,C\,\int_0^t \|\tv_{x_1} \psi_{1t} \big|_{x_1 = 0}\|_{L^\infty}\big(\|\mathbf{1}_{\mathbb T^2}\|_{L^2(\mathbb T^2)}^2 + \|\psi_{1x_1}\big|_{x_1 = 0}\|_{L^2(\mathbb{T}^2)}^2\big)\,dt \\[2mm]
\nonumber & \le C \delta^2 e^{-C\delta \beta}\big(1 + \|\psi_{1x_1} \big|_{x_1 =0}(t)\|_{L^2(\mathbb{T}^2)}^2 + \|\psi_{1x_1} \big|_{x_1 =0}(0)\|_{L^2(\mathbb{T}^2)}^2 \big) \\[1mm]
\nonumber&\quad + C\delta^3 e^{-C\delta \beta} +C\delta^2 \int_0^t \|\phi_t\big|_{x_1 = 0}\|_{L^2(\mathbb T^2)}^2 \, dt + C\delta^4 \int_0^t \|\psi_{1x_1} \big|_{x_1 = 0}\|_{L^2(\mathbb{T}^2)}^2 \, dt \\[2mm]
\nonumber & \le C\delta^2 e^{-C\delta \beta} + C\varepsilon^2 \delta^2 e^{-C\delta\beta} + C\delta^2 \int_0^t\big(\|\phi_t\big|_{x_1 = 0}\|_{L^2(\mathbb T^2)}^2 + \|\psi_{1x_1} \big|_{x_1 = 0}\|_{L^2(\mathbb{T}^2)}^2\big) \, dt
\end{align*}
In the last line, we used the relation $$\big(\psi_{1x_1}(t) \big|_{x_1 = 0}\big)^2 = -\int_{\mathbb{R}^+} \psi_{1x_1} \psi_{1x_1x_1}(t) \, dx_1 \le \int_{\mathbb{R}^+} \big( \psi_{1x_1}^2(t) + \psi_{1x_1x_1}^2(t) \big) \, dx_1$$ together with the \textit{a priori} assumption. Also, note that the boundary term $\psi_{1x_1} \big|_{x_1 = 0}$ appearing in the last line can be treated exactly as in \eqref{eq:treatbd}.
\end{proof}

\subsection{$H^2$ estimate of $v - \tilde v$}
We begin the \(H^2\)-level estimate with the \(v\)-perturbation. By the outflow condition, the boundary contribution from the transport term is dissipative. We therefore choose the test function so that the boundary term generated by diffusion is canceled.

\begin{lemma}Under the \textit{a priori} assumption, there exists a constant
\(C>0\) independent of \(\varepsilon\), \(\delta\), \(\chi\) and \(T\), such that
for all \(t \in [0,T]\), it holds \label{L:23}
\begin{equation}
\begin{aligned}
&\sup_{0 \le t \le T}\big(\|\phi\|_{H^2}^2 + \|\psi\|_{H^1}^2 + \|\partial_{z}^2 \psi\|^2\big) \\[1mm]
&\quad + \int_0^t \bigg[
\|\nabla_x^2 \phi\|^2
+ \|\nabla_x \psi\|_{H^1}^2
+ \|\nabla_x \psi_{zz}\|^2
+ \|\psi'_{x_1x_1z}\|^2 
+ \mathcal{G}_3(s) + \mathcal{G}^s(s)\\
&\qquad \qquad  + \mathcal{D}(s) + \delta^2 |\dot X(s)|^2 
+ \|\nabla_x^2 \phi \big|_{x_1 = 0}\|^2
+ \|\nabla_x \phi \big|_{x_1 = 0} \|^2
+ \|\psi_{1x_1} \big|_{x_1 = 0} \|^2
\bigg]\,ds \nonumber\\[2mm]
&\quad \le C\, \big(\|\phi_0\|_{H^2}^2 + \|\psi_0\|_{H^1}^2 + \|\partial_{z}^2 \psi_0\|^2\big) \\[1mm]
&\qquad + C(\varepsilon + \delta^2 + \chi)\int_0^t
\|\nabla_x^3 \psi\|^2 \, ds  
+ C\int_0^t 
 \|\nabla_x \psi_t\|^2 \,ds + Ce^{-C\delta \beta}
\end{aligned}
\end{equation}
\end{lemma}
\begin{proof}
\noindent $\bullet$ \textbf{Estimate of $\partial_{z}^2 \phi$ :}
After taking $\partial_{z}^2$ to \eqref{eq:peq}, we get the following system.
\begin{equation}
\left\{
\begin{aligned}
& \partial_t\phi_{zz}
+ \textbf{u} \cdot \nabla_x \phi_{zz} 
= v \operatorname{div}_x \psi_{zz} + R
\\[1mm]
&\partial_t \psi_{zz} \label{eq:212}
+ \textbf{u} \nabla_x \psi_{zz}
+ vp'(v)\nabla_x \phi_{zz} = v \big(\mu \Delta_x \psi_{zz}
+ (\mu + \lambda) \nabla_x \operatorname{div}_x \psi_{zz}\big) + T
\end{aligned}
\right.
\end{equation}
where $R$ and $T$ in this lemma are defined by
\begin{equation}
\left\{
\begin{aligned}\nonumber
R \;:=\;&
2\,\phi_{z}\,\operatorname{div}_x\psi_{z}
+\phi_{zz}\,\operatorname{div}_x\psi
-(vF)_{zz}\,\partial_{x_1}\tilde v
-2\,\psi_{z}\cdot\nabla_x\phi_{z}
-\psi_{zz}\cdot\nabla_x\phi
\\[0.4em]
T \;:=\;&
2\mu\,\phi_{z}\,\Delta_x\psi_{z}
+2(\mu+\lambda)\,\phi_{z}\,\nabla_x\operatorname{div}_x\psi_{z}
+\mu\,\phi_{zz}\,\Delta_x\psi
+(\mu+\lambda)\,\phi_{zz}\,\nabla_x\operatorname{div}_x\psi
\\
&\;
-\bigl(vp'(v)\bigr)_{zz}\,\nabla_x (\tilde v)
-\bigl(vF\bigr)_{zz}\,(\tilde{\mathbf u})_{x_1}
-2\bigl(vp'(v)\bigr)_{z}\,\nabla_x\phi_{z}
-\bigl(vp'(v)\bigr)_{zz}\,\nabla_x\phi
\\
&\;
-2\,\psi_{z}\,\nabla_x\psi_{z}
-\psi_{zz}\,\nabla_x\psi.
\end{aligned}
\right.
\end{equation}
Multiplying $\bigl(-p'(v)\phi_{zz}\bigr)$ to $(\refeq{eq:212})_1$, $(\psi_{zz})$ to $(\refeq{eq:212})_2$, adding the resultant equations and integrating with respect to spatial variables, we can get the following.
\begin{equation}
\begin{aligned} \label{eq:divxzzg}
&\frac12 \frac{d}{dt}
\int_{\Omega}
\Bigl( |\psi_{zz}|^2 - p'(v)\,|\phi_{zz}|^2 \Bigr)\,dx
\;-\;\frac{u_{-}}{2}
\int_{\mathbb{T}^2}
\Bigl(
\bigl|\psi_{zz}\bigr|_{x_1=0}\bigr|^2
- p'(v_-)\,\bigl|\phi_{zz}\bigr|_{x_1=0}\bigr|^2
\Bigr)\,dx'
\\
&\;+\;
\int_{\Omega}
v\Bigl(
\mu\,|\nabla_x\psi_{zz}|^2
+(\mu+\lambda)\,|\operatorname{div}_x\psi_{zz}|^2
\Bigr)\,dx
\;=\;\sum_{j=1}^5 J_j .
\end{aligned}
\end{equation}
where the bad terms are defined by
\begin{align}\nonumber
&J_1 := 
\frac12 \int_{\Omega} \operatorname{div}_x \mathbf{u} \big(-vp''(v)|\phi_{zz}|^2 +  p'(v) |\phi_{zz}|^2 - |\psi_{zz}|^2 \big) \, dx \\[1mm]
\nonumber&J_2 :=  \int_{\Omega} \phi_{zz} \nabla_x \big(vp'(v)\big) \psi_{zz} + \nabla_x v \cdot \big((\mu + \lambda) \operatorname{div}_x \psi_{zz} \psi_{zz} + \mu \psi_{zz} \nabla_x \psi_{zz} \big) \, dx
\\[1mm]
\nonumber&J_3 :=
\int_{\Omega}-p'(v) R \phi_{zz} \, dx, \qquad
J_4 := \int_{\Omega} T \cdot \psi_{zz}\, dx
\\[1mm]
\nonumber& J_5 := - \int_{\mathbb{T}^2} v\big((\mu + \lambda) \operatorname{div}_x \psi_{zz} \psi_{1zz} + \mu\psi_{zz} \cdot \nabla_x \psi_{1zz} - p'(v) \phi_{zz} \psi_{1zz}\big) \Big|_{x_1 = 0} \, dx'.
\end{align}
We omit the derivation since it is exactly the same as before. Additionally, note that $$J_5 = 0$$ holds due to outflow boundary condition.\\
For $J_1$, employing exactly the same method as in \eqref{eq:div}, we obtain the following estimate.
$$|J_{1}|
\le \|\operatorname{div}_{x}\mathbf{u}\|_{L^{\infty}}\bigl(\|\psi_{zz}\|^2 + \|\phi_{zz}\|^2\bigr) \le C(\varepsilon + \delta^2) \bigl(\|\nabla_x^2 \operatorname{div}_x \psi\|^2 + \|\psi_{zz}\|^2 + \|\phi_{zz}\|^2\bigr).$$
For $J_2$, using the same manner as in \eqref{eq:addi}, we can get
\begin{align} \nonumber
|J_2| &\le C(\varepsilon + \delta^2) \big(\|\phi_{zz}\|^2 + \|\psi_{zz}\|_{H^1}^2 \big)
\end{align}
To obtain estimates for $J_3$, we first estimate the $L^2$-norms of $R$. 
The $L^2$-norms of $R$ can be obtained in the same manner as in \eqref{eq:estiRT}
\beq
\begin{aligned}\nonumber
\|R\| 
&\le C \big(\|\operatorname{div}_x \psi_z \|_{H^1}\|\phi_{z}\|_{H^1} + \|\operatorname{div}_x \psi \|_{H^2}\|\phi_{zz}\| \\[1mm]
& \hspace{20mm}+ \|\tilde v_{x_1}\|_{L^\infty} \|(vF)_{zz}\|+ \|\psi_z\|_{H^2} \|\nabla_x \phi_z\|  + \|\psi_{zz}\|_{H^1} \|\nabla_x \phi\|_{H^1} \big) \\[2mm]
&\le C(\varepsilon + \delta^2)\big( \|\operatorname{div}_x \psi_z\|_{H^1} + \|\phi_{zz}\| + \|\nabla_x^2 \operatorname{div}_x \psi \| + \|(vF)_{zz}\|\\[1mm]
&\hspace{40mm}  + \|\nabla_x^2 \phi_z\| + \|\nabla_x \phi_z\| + \|\psi_{zz}\|_{H^1}^2 + \|\nabla_x \phi\|_{H^1}^2 \big)
\end{aligned}
\eeq
Hence, we can obtain the following.
\beq
\begin{aligned} \nonumber
|J_3| &\le C\|R\| \|\phi_{zz}\|
\le C(\varepsilon + \delta^2) \big(\|\operatorname{div}_x \psi_z\|_{H^1} + \|\nabla_x^2 \operatorname{div}_x \psi \| + \|\psi_{zz}\|_{H^1}^2 + \|\nabla_x \phi\|_{H^1}^2 \big)
\end{aligned}
\eeq
To estimate $J_4$, first rewrite $T$ as follows.
\beq
\begin{aligned} \nonumber
T &= 2\mu\,\phi_{z}\,\Delta_x\psi_{z}
+2(\mu+\lambda)\,\phi_{z}\,\nabla_x\operatorname{div}_x\psi_{z}
+\mu\,\phi_{zz}\,\Delta_x\psi
+(\mu+\lambda)\,\phi_{zz}\,\nabla_x\operatorname{div}_x\psi
\\[1mm]
&\;
 -\phi_{zz} S\, \nabla_x \tilde v - 2\phi_z S_z \nabla_x \tilde v - v S_{zz} \nabla_x \tilde v
-\bigl(vF\bigr)_{zz}\,\tilde{\mathbf u}_{x_1}
-2\phi_z S \,\nabla_x\phi_{z} -2v S_z \,\nabla_x\phi_{z} \\[1mm]
&\:
-\phi_{zz} S\,\nabla_x \phi - 2\phi_z S_z \nabla_x \phi - v S_{zz} \nabla_x \phi
-2\,\psi_{z}\,\nabla_x\psi_{z}
-\psi_{zz}\,\nabla_x\psi.
\end{aligned}
\eeq
Using the Lemmas~\ref{GNIsevera}, \ref{L: estiF} and \ref{lem:shock-est}, we obtain the following.
\beq
\begin{aligned}\nonumber
|J_4| &\le C\|\psi_{zz}\|_{H^1} \Big(\|\phi_z\|_{H^1}\|\nabla_x^2 \psi_z \| + \|\nabla_x^2 \psi\|_{H^1} \|\phi_{zz}\| + \|\phi_z\|_{H^1} \|\nabla_x \phi_z\| \\
&\hspace{20mm} + \|S_z\|_{H^1} \|\nabla_x \phi_z\| + \|\nabla_x \phi\|_{H^1}\|\phi_{zz}\| + \|\nabla_x \phi\|_{H^1} \|S_{zz}\| \Big) \\[1mm]
&\quad + C\|\psi_{zz}\| \Big(\|S \tilde v_{x_1}\|_{L^\infty} \|\phi_{zz}\| + \|\tilde v_{x_1}\|_{L^\infty} \|\phi_z\|_{H^1} \|S_z\|_{H^1} + \|\tilde v_{x_1} \|_{L^\infty} \|S_{zz}\| \\
&\hspace{20mm} + \|\tilde v_{x_1} \|_{L^\infty} \|(vF)_{zz}\| +\|\psi_z \|_{H^1} \|\nabla_x \psi_z\|_{H^1} +  \|\nabla_x \psi\|_{H^1} \|\psi_{zz}\|_{H^1} \Big) \\[1mm]
&\quad + C\|\psi_{zz}\|_{H^1} \|\phi_z\|_{H^1}\|S_z\|_{H^1}\|\nabla_x \phi\|_{H^1}\\[2mm]
&\le C(\varepsilon + \delta^2) \big(\|\nabla_x \psi\|_{H^2}^2  + \|\nabla_x \phi_z\|^2 + (\mathcal{G}^s + \delta \mathcal{G}^3) + \|\phi_z\|^2 \big)
\end{aligned}
\eeq
Consequently, we get the following
\begin{align}
\sum_{i=1}^5 |J_i| \le C(\varepsilon, \delta^2) \Big(\mathcal{G}^s + \delta \mathcal{G}_3 + \mathcal{D} + \|\nabla_x \psi\|_{H^2}^2 + \|\nabla_x^2 \phi\|^2 \Big) \label{eq:214}
\end{align}

\noindent $\bullet$ \textbf{Estimate of $\partial_{x_1} \partial_{z} \phi$ :} 
We can get below system by taking $\partial_{x_1} \partial_{z}$ to $\eqref{eq:peq}_1$ and $\partial_{z}$ to the first component of $\eqref{eq:peq}_2$.
\begin{equation}
\begin{aligned}
\partial_t \phi_{x_1z}
+ \mathbf{u}\cdot\nabla_x \phi_{x_1z}
&= v\,\psi_{1x_1x_1z} + R
\\[1mm] \label{eq:215}
(2\mu + \lambda)\, v\,\psi_{1x_1x_1z} - v p'(v)\phi_{x_1z}&= T
\end{aligned}
\end{equation}
where $R$ and $T$ in this lemma are defined by
\begin{equation}
\begin{aligned}\nonumber
R &:=
 \phi_{z}\,\operatorname{div}_x \psi_{x_1}
+ v_{x_1}\,\operatorname{div}_x \psi_{z}
+ \phi_{x_1z}\,\operatorname{div}_x \psi
- (vF)_{x_1z}\,(\tilde v)_{x_1} \\[0.5mm]
&\quad - \psi_{z}\cdot\nabla_x \phi_{x_1}
- \mathbf{u}_{x_1}\cdot\nabla_x \phi_{z}
- \psi_{x_1z}\cdot\nabla_x \phi
+ v\, \operatorname{div}_{x'} \psi_{x_1z} \\[2mm]
T &:= v \big(\mu \Delta_{x'} \psi_{1z} + (\mu + \lambda) \operatorname{div}_{x'} \psi_{x_1z}'\big)
+ \phi_{z}\Bigl(\mu \Delta_x \psi_1
+ (\mu+\lambda)\operatorname{div}_x \psi_{x_1}\Bigr) \\[0.5mm]
& \quad - \partial_t \psi_{1z}
- \mathbf{u}\cdot\nabla_x \psi_{1z}
- \Bigl(v\bigl(p'(v)-p'(\tilde v)\bigr)\Bigr)_{z}\,(\tilde v)_{x_1} \\[0.5mm]
&\quad - (vF)_{z}\,\tilde u_{1x_1}
- \bigl(vp'(v)\bigr)_{z}\, \phi_{x_1}
- \psi_{z}\cdot\nabla_x \psi_1
\end{aligned}
\end{equation}
Multiplying $(2\mu + \lambda) \phi_{x_1z}$ to $(\refeq{eq:215})_1$, $\phi_{x_1z}$ to $(\refeq{eq:215})_2$, and adding them, we can get
\begin{equation}
\begin{aligned}\nonumber
&\frac{(2\mu+\lambda)}{2}\,\frac{d}{dt}\int_{\Omega}|\phi_{x_1z}|^2 \, dx
 - \big((2\mu+\lambda)\,\frac{u_{-}}{2}\big)\, \int_{\mathbb{T}^{2}} |\phi_{x_1z}\big|_{\{x_1 = 0\}}|^2\, dx'
- \int_{\Omega} v p'(v)\,|\phi_{x_1z}|^2 \, dx  = \sum_{j=1}^3 K_j .
\end{aligned}
\end{equation}
where each $K_j$ denotes the term whose integrand is given below.
\begin{equation}
\begin{aligned}\nonumber
K_1 &:= -\frac{1}{2}(2\mu+\lambda)\int_{\Omega} \,\operatorname{div}_x\mathbf{u}\,|\phi_{x_1z}|^2 \, dx \\[2mm]
K_2 &:= (2\mu + \lambda) \int_{\Omega} R\, \phi_{x_1z} \, dx, \quad
K_3 :=  \int_{\Omega} T\, \phi_{x_1z} \, dx
\end{aligned}
\end{equation} 
$K_1$ can be treated by exactly same way as in \eqref{eq:div}
\beq
\begin{aligned}\nonumber
|K_1|
&\le C(\varepsilon + \delta^2)\big( \|\phi_{x_1z}\|^2 + \|\nabla_x^2 \operatorname{div} \psi\|^2 \big)
\end{aligned}
\eeq
To estimate $K_2$ and $K_3$, let us first rewrite $R$ and $T$ as follows.
\beq
\begin{aligned} \nonumber
R &=
 \phi_{z}\,\operatorname{div}_x \psi_{x_1}
+ \phi_{x_1}\,\operatorname{div}_x \psi_{z} + \operatorname{div}_x \psi_z \tilde v_{x_1}
+ \phi_{x_1z}\,\operatorname{div}_x \psi
- (vF)_{x_1z}\,(\tilde v)_{x_1} \\[2mm]
&\quad - \psi_{z}\cdot\nabla_x \phi_{x_1}
- \psi_{x_1}\cdot\nabla_x \phi_{z} - \nabla_x \phi_z \cdot \mathbf{\tilde u}_{x_1}
- \psi_{x_1z}\cdot\nabla_x \phi
+ v\, \operatorname{div}_{x'} \psi_{x_1z} \\[2mm]
T &= v \big(\mu \Delta_{x'} \psi_{1z} + (\mu + \lambda) \operatorname{div}_{x'} \psi_{x_1z}'\big)
+ \phi_{z}\Bigl(\mu \Delta_x \psi_1
+ (\mu+\lambda)\operatorname{div}_x \psi_{x_1}\Bigr) \\[0.5mm]
& \quad - \partial_t \psi_{1z}
- \mathbf{u}\cdot\nabla_x \psi_{1z}
- \phi_z S \tilde v_{x_1} - vS_z \tilde v_{x_1} - (vF)_{z}\,\tilde u_{1x_1}
- p'(v) \phi_z \phi_{x_1} - vS_z \phi_{x_1} - \psi_{z}\cdot\nabla_x \psi_1
\end{aligned}
\eeq
Therefore, by Lemmas~\ref{lem:shock-est}, \ref{GNIsevera}, and \ref{L: estiF}, we can bound the $L^2$-norms of $R$ and $S$ as follows.
\beq
\begin{aligned}\nonumber
\|R\| \le & \, C\bigg(\|\phi_z\|_{H^1} \|\operatorname{div}_x \psi_{x_1}\|_{H^1} + \|\phi_{x_1}\|_{H^1} \|\operatorname{div}_x \psi_z\|_{H^1} + \|\tilde v_{x_1}\|_{L^\infty} \|\operatorname{div}_x \psi_z\|\\
&\quad + \|\operatorname{div}_x \psi\|_{H^2} \|\phi_{x_1z}\| 
 + \|\tilde v_{x_1}\|_{L^\infty} \|(vF)_{x_1z}\| + \|\psi_z\|_{H^2} \|\nabla_x \phi_{x_1}\| \\
&\quad+ \|\psi_{x_1}\|_{H^2} \|\nabla_x \phi_z\| + \|\mathbf{\tilde u}_{x_1}\|_{L^\infty} \|\nabla_x \phi_z\| + \|\psi_{x_1z} \|_{H^1} \|\nabla_x \phi\|_{H^1} + \|\operatorname{div}_{x'} \psi_{x_1z}\|\bigg) \\
&\quad+ C\,\|\operatorname{div}_{x'} \psi_{x_1z}\| \\[2mm]
\le& \, C(\varepsilon + \delta^2) \big(\|\nabla_x \phi\|_{H^1} + \|\operatorname{div}_x \psi_z\| 
 + \|\nabla_x^3 \psi\| +  \|\psi_{x_1z} \| \big) + C\,\|\operatorname{div}_{x'} \psi_{x_1z}\| \\[4mm]
\|S\| \le & \, C\bigg(\|\psi_{1zzz}\| + \|\psi'_{x_1zz}\| + \|\phi_z\|_{H^1}\|\nabla_x^2 \psi\|_{H^1} + \|\psi_{zt}\|\\
&\quad + \|\nabla_x \psi_{1z}\| + \|S\tilde v_{x_1}\|_{L^\infty} \|\phi_z\| + \|\tilde v_{x_1}\|_{L^\infty} \|S_z\| \\
&\quad+ \|\phi_z\|_{H^1} \|\phi_{x_1}\|_{H^1} + \|S_z\|_{H^1} \|\phi_{x_1}\|_{H^1} + \|\psi_z\|_{H^1} \|\nabla_x \psi_1 \|_{H^1} \bigg) \\[2mm]
\le & \, C(\varepsilon + \delta^2) \big(\|\nabla_x^2 \psi\|_{H^1}  + \|\phi_z\|_{H^1} + \|S_z\|_{H^1} + \|\psi_z\|_{H^1} \big) \\[2mm]
&\quad + C\, \big(\|\psi_{1zzz}\| + \|\psi'_{x_1zz}\| + \|\psi_{zt}\| + \|\nabla_x \psi_{1z}\| \big)
\end{aligned}
\eeq
The required $L^2$-estimates for $\operatorname{div}_{x'}\psi_{x_1z}$, $\psi_{1zzz}$, $\psi'_{x_1zz}$, $\psi_{zt}$, and $\nabla_x\psi_z$ follow from \eqref{eq:divxzzg}, Lemma~\ref{L:22}, and Lemma~\ref{L:psiH2}. Hence, $K_2$ and $K_3$ can be estimated as before. By the Young's inequality,
\beq
\begin{aligned}\nonumber
|K_2| \le C\|R\|\|\phi_{x_1z}\| &\le C(\varepsilon + \chi + \delta^2) \big(\|\nabla_x \phi\|_{H^1} + \|\operatorname{div}_x \psi_z\| 
 + \|\nabla_x^3 \psi\| +  \|\psi_{x_1z} \| \big) \\[0.5mm]
&\quad+ C_{\chi}\,\|\operatorname{div}_{x'} \psi_{x_1z}\| \\[2mm]
|K_3| \le C\|S\|\|\psi_{x_1z}\| &\le C(\varepsilon + \chi + \delta^2) \big(\|\nabla_x^2 \psi\|_{H^1}  + \|\phi_z\|_{H^1} + \|S_z\|_{H^1} + \|\psi_z\|_{H^1} \big) \\
&\quad + C_{\chi}\, \big(\|\psi_{1zzz}\| + \|\psi'_{x_1zz}\| + \|\psi_{zt}\| + \|\nabla_x \psi_{1z}\| \big)
\end{aligned}
\eeq
Hence, we can obtain the following.
\begin{align}
\sum_{i=1}^3 |K_i| \le C(\varepsilon, \delta^2, \chi) \Big(\|\nabla_x^2 \phi\|^2 + \|\nabla_x \psi\|_{H^2}^2  \Big) + C\Big(\|\phi_{x_1z}\|^2 + \|\nabla_x \psi_{1z}\|^2 + \|\partial_t \psi_{1z}\|^2 + \|\nabla_x \psi_{zz}\|^2 \Big) \label{eq:216}
\end{align}
$\bullet$ \textbf{$L^2$- estimate of $\partial_{z}^2 \phi$ :}
Taking $\partial_{z}$ to the second or third component of $\eqref{eq:peq}_2$, we can get the following.
\begin{equation}
\begin{aligned} \label{eq:313}
v p'(v) \phi_{zz} - \mu v \psi'_{x_1x_1z}
&= \mu v\,\Delta_{x'} \psi'_{z}
+ (\mu+\lambda)\,v\, \operatorname{div}_x \psi_{zz}
+ \phi_{z}\Bigl(\mu \Delta_x \psi'
+ (\mu+\lambda)\operatorname{div}_x \psi_{z}\Bigr) \\
&\quad
- \bigl(vp'(v)\bigr)_{z}\phi_{z}
- \psi_{z}\cdot\nabla_x \psi'
- \partial_t \psi'_{z}
- \mathbf{u}\cdot\nabla_x \psi'_{z} \\[2mm]
& := R
\end{aligned}
\end{equation}
By periodicity in the transverse variables, transverse derivatives can be freely transferred between \(\psi'\) and \(\phi\) in the cross term \(\psi'_{x_1x_1z}\phi_{zz}\). Hence, for any \(\chi>0\), we obtain
\beq
\begin{aligned} \nonumber
\int_{\Omega} \phi_{zz} \psi'_{x_1x_1z} \, dx &= -\int_{\mathbb{T}^2} \phi_{zz} \psi'_{x_1z} \big|_{x_1 = 0} \, dx' - \int_{\Omega} \phi_{x_1zz} \psi'_{x_1z} \, dx \\[2mm]
&= -\int_{\mathbb{T}^2} \phi_{zz} \psi'_{x_1z} \big|_{x_1 = 0} \, dx' + \int_{\Omega} \phi_{x_1z} \psi'_{x_1zz} \, dx \\[2mm]
&\le \|\phi_{zz}\big|_{x_1 = 0}\|^2 + \|\psi'_{x_1z} \big|_{x_1 = 0}\|^2 + \|\phi_{x_1z}\|^2 + \|\psi'_{x_1zz}\|^2 \\[2mm]
&\le \|\phi_{zz}\big|_{x_1 = 0}\|^2 + C_{\chi} \|\psi'_{x_1z}\|^2 + \chi \|\psi'_{x_1x_1z}\|^2 + \|\phi_{x_1z}\|^2 + \|\psi'_{x_1zz}\|^2
\end{aligned}
\eeq
To handle the boundary term in the last line, we use \eqref{eq:treatbd}.

By Lemmas \ref{GNIsevera} and \ref{L: estiF}, and the relation \eqref{eq:baro1}, the right-hand side of \eqref{eq:313} is bounded as follows.
\beq
\begin{aligned} \nonumber
\|R\| &\le C\bigg[\|\psi'_{zt}\| + \|\nabla_x \psi'_{z}\| + \|\psi_{z}\|_{H^1} \|\nabla_x \psi'\|_{H^1} + \|\big(vp'(v)\big)_{z}\|_{H^1} \|\phi_{z}\|_{H^1} \\[1mm]
&\qquad+ \|\Delta_{x'} \psi'_{z}\| + \|\operatorname{div}_x \psi_{zz}\| + \|\phi_{z}\|_{H^1} \|\nabla_x^2 \psi\|_{H^1} \bigg]\\[2mm]
&= C\bigg[\|\psi'_{zt}\| + \|\nabla_x \psi'_{z}\| + \|\psi_{z}\|_{H^1} \|\nabla_x \psi'\|_{H^1} + \|\big(\gamma^2 v^{-\gamma-1} \phi_z)\|_{H^1} \|\phi_{z}\|_{H^1}\\[1mm]
&\qquad + \|\Delta_{x'} \psi'_{z}\| + \|\operatorname{div}_x \psi_{zz}\| + \|\phi_{z}\|_{H^1} \|\nabla_x^2 \psi\|_{H^1} \bigg] \\[2mm]
&\le C\varepsilon\big(\|\psi_z\|_{H^2} + \|\phi_z\|_{H^1}\big) + C\big(\|\psi'_{zt}\| + \|\nabla_x \psi'_{z}\| + \|\Delta_{x'} \psi'_{z}\| + \|\operatorname{div}_x \psi_{zz}\| \big)
\end{aligned}
\eeq
Therefore, by squaring both sides of \eqref{eq:313} and integrating, and using the results obtained above, we obtain the following estimate
\beq
\begin{aligned} \nonumber
\|\phi_{zz}\|^2 + \|\psi'_{x_1x_1z}\|^2 &= 2\mu \int_{\Omega} v^2 p'(v) \phi_{zz} \psi'_{x_1x_1z} \, dx + \int_{\Omega} R^2 \, dx \\[2mm]
&\le C \big(\|\phi_{zz}\big|_{x_1 = 0}\|^2 + C_{\chi} \|\psi'_{x_1z}\|^2 + \chi \|\psi'_{x_1x_1z}\|^2 + \|\phi_{x_1z}\|^2 + \|\psi'_{x_1zz}\|^2 \big)\\
&\quad +C\varepsilon^2\big(\|\psi_z\|_{H^2}^2 + \|\phi_z\|_{H^1}^2\big) + C\big(\|\psi'_{zt}\|^2 + \|\nabla_x \psi'_{z}\|^2 + \|\nabla_x \psi_{zz}\|^2 \big)
\label{eq:218}
\end{aligned}
\eeq
Therefore, choosing $\chi > 0$ sufficiently small, we obtain the following $L^2$-estimates.
\beq
\begin{aligned}\label{eq:316}
\|\phi_{zz}\|^2 + \|\psi'_{x_1x_1z}\|^2 &\le C_{\chi} \big(\|\psi'_{zt}\|^2 + \|\nabla_x^2 \psi\|^2 +  \|\nabla_x \psi_{zz}\|^2 + \|\phi_{zz}\big|_{x_1 = 0}\|_{L^2(\mathbb{T}^2)}^2 \big)\\[1mm]
&\quad + C(\varepsilon + \chi) \big(\|\nabla_x^3 \psi\|^2 + \|\nabla_x^2 \phi\|^2 \big)
\end{aligned}
\eeq

\noindent $\bullet$ \textbf{$L^2$-norm estimate of $\partial_{x_1}^2 \phi$ :}
As before, taking $\partial_{x_1}^2$ to $\eqref{eq:peq}_1$ and $\partial_{x_1}$ to the first component of $\eqref{eq:peq}_2$, we obtain the following.
\begin{equation}
\begin{aligned} \label{eq:219}
\partial_t \phi_{x_1x_1}
+ \mathbf{u}\cdot \nabla_x \phi_{x_1x_1}
&= v\,\operatorname{div}_x \psi_{x_1x_1} + R
\\[1mm]
(\mu+\lambda)\,v\,\operatorname{div}_x \psi_{x_1x_1} - v p'(v)\,\phi_{x_1x_1} &= T
\end{aligned}
\end{equation}
where
\begin{equation}
\begin{aligned}\nonumber
R &:= \dot X(t)\,(\partial_{x_1}^3\tilde v)
- (vF)\,(\partial_{x_1}^3\tilde v)
- 2\mathbf{u}_{x_1}\cdot\nabla_x \phi_{x_1}
- \mathbf{u}_{x_1x_1}\cdot\nabla_x \phi \\[1mm]
&\quad - 2(vF)_{x_1}\,\tilde v_{x_1x_1}
- (vF)_{x_1x_1}\,(\tilde v)_{x_1}
+ 2v_{x_1}\,\operatorname{div}_x \psi_{x_1}
+ v_{x_1x_1}\,\operatorname{div}_x \psi \\[2mm]
T &:= 
\partial_t \psi_{1x_1}
+ \mathbf{u}\cdot \nabla_x \psi_{1x_1}
+ v\bigl( p'(v) - p'(\tilde v) \bigr)\,(\tilde v)_{x_1x_1}
\\[1mm]
&\quad - \dot X(t)\,(\tilde u_1)_{x_1x_1}
+ (vF)\,(\tilde u_1)_{x_1x_1}
-\mu v\,\Delta_x \psi_{1x_1} 
+ \mathbf{u}_{x_1}\cdot \nabla_x \psi_{1}
+ (vF)_{x_1}\,(\tilde u_1)_{x_1}
\\[1mm]
&\quad + (vp'(v))_{x_1}\,\phi_{x_1}
+ \Bigl(v\bigl(p'(v)-p'(\tilde v)\bigr)\Bigr)_{x_1}\,(\tilde v)_{x_1}
 - \mu v_{x_1}\,\Delta_x \psi_1
- (\mu+\lambda)\,v_{x_1}\,\operatorname{div}_x \psi_{x_1}
\end{aligned}
\end{equation}
Similarly, multiplying $(\refeq{eq:219})_1$ by $(\mu + \lambda) \phi_{x_1x_1}$ and $(\refeq{eq:219})_2$ by $\phi_{x_1x_1}$, and then adding them together, we get the following:
\begin{equation}
\begin{aligned}\nonumber
&\frac{(\mu+\lambda)}{2}\frac{d}{dt} \int_{\Omega} |\phi_{x_1x_1}|^2 \, dx 
 -\big((\mu +\lambda)\frac{u_{-}}{2}\big) \int_{\mathbb{T}^{2}} |\phi_{x_1x_1}\big|_{\{x_1 = 0\}}|^2 \, dx'
- \int_{\Omega} v p'(v)\,|\phi_{x_1x_1}|^2 \, dx\\[2mm]
&\quad
=(\mu+\lambda)\int_{\Omega} \phi_{x_1x_1}\, R \, dx
+\int_{\Omega} \phi_{x_1x_1}\, T \, dx
- (\mu + \lambda) \int_{\Omega} \operatorname{div}_x \mathbf{u} |\phi_{x_1x_1}|^2 \, dx \\[2mm]
&\quad := M_1 + M_2 + M_3.
\end{aligned}
\end{equation}
To estimate $M_1$ and $M_2$, let us first estimate $R$ and $T$. Using the relations $\mathbf{u} = \psi + \mathbf{\tilde u}$ and $v = \phi + \tilde v$, we can rewrite $R$ and $T$ as follows:
\begin{equation}
\begin{aligned}\nonumber
R &= \dot X(t)\,(\partial_{x_1}^3\tilde v)
- (vF)\,(\partial_{x_1}^3\tilde v)
- 2\psi_{x_1}\cdot\nabla_x \phi_{x_1} - 2 \mathbf{\tilde u}_{x_1} \cdot \nabla_x \phi_{x_1}
- \psi_{x_1x_1}\cdot\nabla_x \phi - \mathbf{\tilde u}_{x_1x_1} \cdot \nabla_x \phi \\[1mm]
&\quad - 2(vF)_{x_1}\,\tilde v_{x_1x_1}
- (vF)_{x_1x_1}\,\tilde v_{x_1}
+ 2\phi_{x_1}\,\operatorname{div}_x \psi_{x_1} + 2\tilde v_{x_1}\,\operatorname{div}_x \psi_{x_1}
+ \phi_{x_1x_1}\,\operatorname{div}_x \psi + \tilde v_{x_1x_1}\,\operatorname{div}_x \psi
\end{aligned}
\eeq
and
\beq
\begin{aligned}\nonumber
T &:= 
\partial_t \psi_{1x_1}
+ \mathbf{u}\cdot \nabla_x \psi_{1x_1}
+ vS\,\tilde v_{x_1x_1}
 - \dot X\,\tilde u_{1x_1x_1}
+ (vF)\,\tilde u_{1x_1x_1}
-\mu v\,\Delta_x \psi_{1x_1} 
\\[1mm]
&\quad
+ \psi_{x_1}\cdot \nabla_x \psi_{1} + \mathbf{\tilde u}_{x_1}\cdot \nabla_x \psi_{1}
+ (vF)_{x_1}\,\tilde u_{x_1}
 + (vp'(v))_{x_1}\,\phi_{x_1}
+ \phi_{x_1}S\,\tilde v_{x_1} + \tilde v_{x_1}S\,\tilde v_{x_1} + vS_{x_1}\,\tilde v_{x_1}
\\[1mm]
&\quad - \mu \phi_{x_1} \,\Delta_x \psi_1 -\mu \tilde v_{x_1}\,\Delta_x \psi_1
- (\mu+\lambda)\,\phi_{x_1}\,\operatorname{div}_x \psi_{x_1} 
- (\mu+\lambda)\,\tilde v_{x_1}\,\operatorname{div}_x \psi_{x_1}.
\end{aligned}
\eeq
By the Lemmas \ref{infty interpolation inequality lemma} and \ref{L: estiF}, we can estimate this as follows.
\begin{equation}
\begin{aligned}\nonumber
\|R\| &\le C\bigg(|\dot X| \|\partial_{x_1}^3 \tilde v \| + \|(vF)(\partial_{x_1}^3 \tilde v)\| + \|\psi_{x_1}\|_{L^\infty}\|\nabla_x \phi_{x_1}\| + \|\mathbf{\tilde u}_{x_1}\|_{L^\infty}\|\nabla_x \phi_{x_1} \| \\[1mm]
&\qquad + \|\psi_{x_1x_1}\|_{H^1}\|\nabla_x \phi\|_{H^1} + \|\mathbf{\tilde u}_{x_1x_1}\|_{L^\infty} \|\nabla_x \phi_{x_1} \| + \|\tilde v_{x_1x_1}\|_{L^\infty} \|(vF)_{x_1}\| \\[1mm]
&\qquad + \|\tilde v_{x_1}\|_{L^\infty} \|(vF)_{x_1x_1}\| + \|\phi_{x_1}\|_{H^1} \|\operatorname{div}_x \psi_{x_1}\|_{H^1} + \|\tilde v_{x_1}\|_{L^\infty} \|\operatorname{div}_x \psi_{x_1}\|\\[1mm]
&\qquad + \|\phi_{x_1x_1}\| \|\operatorname{div}_x \psi\|_{L^\infty} + \|\tilde v_{x_1x_1}\|_{L^\infty} \|\operatorname{div}_x \psi\| \bigg) \\[2mm]
& \le C \bigg(\delta^{7/2} |\dot X| + \delta^2 \|(vF)\tilde v_{x_1}\| + \varepsilon \|\nabla_x^2 \psi_{x_1}\| + \varepsilon \|\nabla_x \phi_{x_1}\| + \delta^2 \|\nabla_x \phi_{x_1}\|\\[1mm]
&\qquad + \varepsilon\|\nabla_x \phi\|_{H^1} + \delta^3 \|\nabla_x \phi_{x_1}\| + \delta^3 \|(vF)_{x_1}\| + \delta^2 \|(vF)_{x_1x_1}\| + \varepsilon\|\phi_{x_1}\|_{H^1}\\[1mm]
&\qquad + \delta^2 \|\operatorname{div}_x \psi_{x_1}\| + \varepsilon \|\phi_{x_1x_1}\| + \varepsilon \|\operatorname{div}_x \psi\|_{H^2} + \delta^3 \|\operatorname{div}_x \psi\|\bigg) \\[2mm]
& \le C \delta^{7/2} |\dot X| + C(\varepsilon + \delta^2) \big((\mathcal{G}^s + \delta \mathcal{G}_3)^{1/2} + \|\nabla_x \phi\|_{H^1} + \|\nabla_x \psi\|_{H^2} \big)
\end{aligned}
\eeq
and,
\beq
\begin{aligned} \nonumber
\|T\|
&\le C\bigg(\|\partial_t \psi_{1x_1}\| + \|\nabla_x \psi_{1x_1}\| + \delta \|S\,\tilde v_{x_1}\|
 + \delta^{7/2}|\dot X| + \delta \|(vF)\,\tilde u_{x_1}\| + \|\Delta_x \psi_{1x_1}\| 
\\[1mm]
&\qquad
+ \|\psi_{x_1}\|_{H^1} \|\nabla_x \psi_{1}\|_{H^1} + \delta^2 \|\nabla_x \psi_{1}\|
+ \delta^2 \|(vF)_{x_1}\|
 + \|(vp'(v))_{x_1}\|_{H^1}\|\phi_{x_1}\|_{H^1}\\[1mm]
&\qquad
+ \varepsilon \delta^2\|\phi_{x_1}\| + \delta^2 \|S\,\tilde v_{x_1}\| + \delta^2 \|S_{x_1}\|
+ \|\phi_{x_1}\|_{H^1} \|\Delta_x \psi_1\|_{H^1} \\[1mm]
&\qquad + \delta^2 \|\Delta_x \psi_1\|
+\|\phi_{x_1}\|_{H^1}\|\operatorname{div}_x \psi_{x_1}\|_{H^1} + \delta^2 \|\operatorname{div}_x \psi_{x_1} \| \bigg) \\[2mm]
&\le C\delta^{7/2} |\dot X| + C\big(\|\partial_t \psi_{1x_1}\| + \|\nabla_x \psi_{1x_1}\|+ \|\Delta_x \psi_{1x_1}\| \big) \\[1mm]
&\quad + C(\varepsilon + \delta^2) \big((\mathcal{G}^s + \delta \mathcal{G}_3)^{1/2} + \|\nabla_x \phi\|_{H^1} + \|\nabla_x \psi\|_{H^2} \big).
\end{aligned}
\end{equation}
Consequently, we can estimate $M_2$ and $M_3$ as follows:
\beq
\begin{aligned} \nonumber
|M_1| \le C\|\phi_{x_1x_1}\| \|R\| &\le  C \delta^{2} |\dot X|^2 + C(\varepsilon + \delta^2) \big((\mathcal{G}^s + \delta \mathcal{G}_3) + \|\nabla_x \phi\|_{H^1}^2 + \|\nabla_x \psi\|_{H^2}^2 \big), \\[2mm]
|M_2| \le C\|\psi_{x_1x_1}\| \|T\| 
&\le C\delta^{2} |\dot X|^2 + C\big(\|\partial_t \psi_{1x_1}\|^2 + \|\nabla_x \psi_{1x_1}\|^2+ \|\Delta_x \psi_{1x_1}\|^2 \big) \\[1mm]
&\quad + C(\varepsilon + \delta^2) \big((\mathcal{G}^s + \delta \mathcal{G}_3) + \|\nabla_x \phi\|_{H^1}^2 + \|\nabla_x \psi\|_{H^2}^2 \big).
\end{aligned}
\eeq
Estimates of $M_3$ can be derived via standard manner as in \eqref{eq:div}.
$$|M_3| \le C(\varepsilon + \delta^2)\big(\|\phi_{x_1x_1}\|^2 + \|\nabla_x^2 \operatorname{div}_x \psi\|^2 \big)$$
Combining \eqref{eq:214}, \eqref{eq:216}, and \eqref{eq:316} together with the above results, we obtain the Lemma.
\end{proof}
\subsection{$H^2$ estimate of $\mathbf{u} - \tilde{\mathbf{u}}$}Finally, through the computations below, we close the $H^2$ estimate for the $\mathbf{u}$-perturbation.
\begin{lemma}Under the \textit{a priori} assumption, there exists a constant
\(C>0\) independent of \(\varepsilon\), \(\delta\), \(\chi\) and \(T\), such that
for all \(t \in [0,T]\), it holds \label{L:24}
\[
\begin{aligned}
&\sup_{0 \le t \le T} \|(\phi,\psi)(t)\|_{H^2}^2
+ \int_0^t \Bigl(
 \delta |\dot X(s)|^2
+ \mathcal G^s(s) + \mathcal G_3(s) + \mathcal D(s)
+ \|\nabla_x \psi\|_{H^2}^2
+ \|\nabla_x^2 \phi\|^2 \\
&\hspace{45mm} + \|(\phi_t, \psi_t)\|^2
+ \|\nabla_x \phi\big|_{x_1=0}\|^2
+ \|\nabla_x^2 \phi\big|_{x_1=0}\|^2
\Bigr)\,ds \\[1mm]
&\qquad \le C\|(\phi_0,\psi_0)\|_{H^2}^2
+ C e^{-C\delta\beta} + \int_0^t \|\nabla_x \psi_t \|^2 \, ds.
\end{aligned}
\]
\end{lemma}
\begin{proof}
$\bullet$ \textbf{Estimate of $\nabla_x \psi_{z}$ :}
After taking $\partial_{z}$ to the momentum equation, we can get
\[
\begin{aligned}
&
-v\Big(\mu \Delta_x \psi_{z}
    + (\mu + \lambda) \nabla_x \operatorname{div}_x \psi_z \Big)
+ \partial_t \psi_{z}
= R,
\end{aligned}
\]
where $R$ is defined by
\begin{align}
R &:= 
\phi_{z}\Bigl( \mu \Delta_x \psi
+ (\mu + \lambda) \nabla_x \operatorname{div}_x \psi \Bigr)
- \psi_{z} \nabla_x \psi \nonumber\\[1mm]
&\qquad - \bigl(vp'(v)\bigr)_{z} \nabla_x \phi
 - vp'(v) \nabla_x \phi_{z} 
-\mathbf{u} \, \nabla_x \psi_{z} - (vS)_{z}\nabla_x \tilde v
 - (vF)_{z}\tilde{\mathbf u}_{x_1}.\nonumber
\end{align}
After multiplying both sides by $\rho \, \psi_{zt}$ and integrating, we get this form:
\[
\begin{aligned}
&\frac12 \frac{d}{dt}\int_{\Omega} 
\bigl(\mu|\nabla_x\psi_{z}|^2+(\mu+\lambda)|\operatorname{div}_x\psi_{z}|^2\bigr)
 \, dx
+ \int_{\Omega} \rho|\psi_{zt}|^2 \, dx 
 = J_1 + J_2 + J_3.
\end{aligned}
\]
Each $J_j$ is defined as follows:
\begin{equation}
\begin{aligned}\nonumber
&J_1 := \int_{\Omega} \Bigl(
\rho \phi_{z} \bigl( \mu \Delta_x \psi
+ (\mu + \lambda) \nabla_x \operatorname{div}_x \psi \bigr) \cdot \psi_{zt}
\Bigr) \, dx, \\[1mm]
&J_2 := - \int_{\Omega}  \rho \Bigl(
\psi_{z} \nabla_x \psi + \bigl(vp'(v)\bigr)_{z} \nabla_x \phi
 + vp'(v) \nabla_x \phi_{z} 
+ \mathbf{u} \, \nabla_x \psi_{z}
\Bigr)\cdot \psi_{zt} \,\,\, dx,  \\[1mm]
&J_3 := - \int_{\Omega} \rho\bigl( (vS)_{z}\nabla_x \tilde v
 + (vF)_{z}\tilde{\mathbf u}_{x_1}
\bigr) \cdot \psi_{zt} \,\,\, dx.
\end{aligned}
\end{equation}
Using the Lemma \ref{GNIsevera} and the \textit{a priori} assumption
$\|\phi_z\|_{H^1} \le C\e $, we obtain the following:
\begin{equation}
\begin{aligned} \nonumber
|J_1| 
&\le C\, \|\phi_{z}\|_{H^1}\, \|(\Delta_x \psi + \nabla_x \operatorname{div}_x \psi)\|_{H^1}\, \|\psi_{zt}\| 
\le C\varepsilon \Big(\|\psi_{zt}\|^2 + \|\Delta_x \psi\|_{H^1}^2 + \|\nabla_x \operatorname{div}_x \psi\|_{H^1}^2 \Big).
\end{aligned}
\end{equation}
For $J_2$, using Young's inequality \eqref{eq:baro1}, Lemmas \ref{GNIsevera}, and \ref{L: estiF}, for a sufficiently small \(\chi _{0}\), we can derive the following:
\begin{equation}
\begin{aligned}\nonumber
|J_2| 
&\le C \|\psi_{zt}\| \Big( \|\psi_{z}\|_{H^1} \|\nabla_x \psi\|_{H^1} + \|(vS)_z \|_{H^1} \|\nabla_x \phi\|_{H^1} \\
&\hspace{30mm} + \|(vp'(\tilde v))_z\|_{L^\infty} \|\nabla_x \phi\| + \|\nabla_x \phi_{z}\| + \|\nabla_x \psi_{z}\| \Big) \\[1mm]
&\le C(\varepsilon + \delta^2 + \chi) \Big(\|\nabla_x \psi\|_{H^1}^2 + \|\nabla_x \phi\|_{H^1}^2 \Big) + C_{\chi} \|\psi_{zt}\|^2.
\end{aligned}
\end{equation}
The term $J_3$ can be estimated as follows. Lemmas \ref{lem:shock-est}, \ref{L: estiF}, and \ref{L:psig} yield
\begin{equation}
\begin{aligned}\nonumber
|J_3| 
&\le C\delta^2 \Big(\|\phi_{z}\|^2 + \|\psi_{1z}\|^2 + \|\psi_{zt}\|^2 \Big)
\end{aligned}
\end{equation}
Therefore, we can obtain the following inequality:
\begin{align} \nonumber
\sum_{i=1}^3 |J_i| \le C(\varepsilon+\delta^2+ \chi) \Big(\|\nabla_x \psi\|_{H^2}^2 + \|\nabla_x \phi\|_{H^1}^2 \Big) + C_\chi\, \|\psi_{zt}\|^2. \label{eq:220}
\end{align}
$\bullet$ \textbf{Estimate of $\partial_{x_1}^2 \psi$:}
Recall \eqref{eq:new}:
$$\|\psi_{1x_1x_1}\|^2 \le C\big(\|\psi_{1t}\|^2 + \|\nabla_x \psi_1\|^2 + \|\nabla_x \psi_{z}\|^2 + \mathcal{D} \big) + C\delta \big(\,\mathcal{G}_3 + \mathcal{G}^s\,\big) + C\delta^2 |\dot X|^2,$$
and \eqref{eq:19}:
$$ \|\psi'_{x_1x_1}\|^2 \le C\big(\|\psi'_t\|^2 + \|\nabla_x \psi'\|^2 + \mathcal{D} + \|\nabla_x \psi_{z}\|^2 \big). $$
Combining the above two inequalities with the fact \eqref{dxbound}, we obtain the following:
\beq
\begin{aligned} \nonumber
\sup_{0 \le s \le T}\|\psi_{x_1x_1}(s)\|^2 \,
\le\;& \sup_{0 \le s \le T} C \Big(\|\psi_t(s)\|^2
 + \|\nabla_x \psi(s)\|^2
 + \|\nabla_x \phi(s)\|^2 \\
&\hspace{16mm}
 + \|\nabla_x \psi_{z}(s)\|^2
 + \|\operatorname{div}_x \psi_{z}(s)\|^2 + \delta^2 \|\phi(s)\|^2\Big).
\end{aligned}
\eeq
$\bullet$ \textbf{$L^2$-estimate of $\psi_{1x_1x_1z}$:}
Differentiating $\eqref{eq:peq}_2$ by ${z}$ and rewriting it. We can obtain the following:
\beq
\begin{aligned}
(2\mu + \lambda)v\, \psi_{1x_1x_1z}
&=\partial_t \psi_{1z}
+ \mathbf{u}\cdot\nabla_x \psi_{1z}
+ v p'(v)\phi_{x_1z}
+ (vS)_{z}(\tilde v)_{x_1}
\nonumber\\[1mm]
&\quad + (vF)_{z}\,(\tilde u_{x_1})
- \mu v\,\Delta_{x'} \psi_{1z}
- (\mu+\lambda)\,v\,\operatorname{div}_{x'} \psi_{x_1z}'
\\[1mm]
&\quad
- \phi_{z}\Bigl(\mu \Delta_x \psi_1
+ (\mu+\lambda)\operatorname{div}_x \psi_{x_1}\Bigr)
+ \bigl(vp'(v)\bigr)_{z} \phi_{x_1}
+ \psi_{z}\cdot\nabla_x \psi_1.
\end{aligned}
\eeq
Multiplying the above equation by $\frac{1}{(2\mu+\lambda)v}\psi_{1x_1x_1z}$, integrating over the spatial domain, and applying Young's inequality, Lemma~\ref{lem:shock-est}, and the \textit{a priori} assumption, we obtain
\beq
\begin{aligned}\label{eq:demo}
\int_0^t \|\psi_{1x_1x_1z}\|^2\,ds
&\le\; C \int_0^t \Bigl(
 \|\partial_t \psi_{1z}\|^2
+ \|\nabla_x \psi_{1z}\|^2
+ \|\phi_{x_1z}\|^2
+ (\varepsilon + \delta^2)\, \mathcal{D}(s) 
+ \|\nabla_x \psi_{zz}\|^2 \\
&\hspace{16mm}
+ \varepsilon\bigl(
\|\Delta_x \psi_1\|_{H^1}^2
+ \|\operatorname{div}_x \psi_{x_1}\|_{H^1}^2
+ \|\nabla_x \phi_{x_1}\|^2
+ \|\nabla_x \psi_1\|_{H^1}^2
\bigr)
\Bigr) \, ds. 
\end{aligned}
\eeq
$\bullet$ \textbf{$L^2$-estimate of $\partial_{x_1}^3 \psi$:}
We proceed with a similar argument. Differentiating the momentum equation with respect to $x_1$, we obtain the following:
\begin{equation}
\begin{aligned}\nonumber
\mu v\,\psi'_{x_1x_1x_1}
=&\;\partial_t \psi'_{x_1}
+ \mathbf{u}\cdot\nabla_x \psi'_{x_1}
+ \psi_{x_1}\cdot\nabla_x \psi'
+ \mathbf{\tilde u}_{x_1}\cdot\nabla_x \psi' \\[1.5mm]
&\; + v p'(v)\,\phi_{x_1z}
+ (vS)_{x_1}\phi_{z} 
+ \bigl(\tilde v p'(\tilde v)\bigr)_{x_1}\phi_{z}  \\[1.5mm]
&\; - \mu v\,\Delta_{x'} \psi'_{x_1}
- \mu \phi_{x_1}\bigl(\Delta_x \psi'\bigr)
- \mu \tilde v_{x_1} \,\bigl(\Delta_x \psi'\bigr) \\[1.5mm]
&\; - (\mu+\lambda)v\,\operatorname{div}_x \psi_{x_1z}
- (\mu+\lambda)\phi_{x_1}\operatorname{div}_x \psi_{z} 
- (\mu+\lambda)\tilde v_{x_1} \,\operatorname{div}_x \psi_{z}.
\end{aligned}
\end{equation}
Multiplying both sides by $\frac{1}{\mu v}\psi'_{x_1x_1x_1}$ and arguing as in \eqref{eq:demo}, we obtain
\beq
\begin{aligned}
\|\psi'_{x_1x_1x_1}\|^2 &\le C\Big(\|\partial_t \psi'_{x_1}\|^2 + \|\nabla_x \psi'_{x_1}\|^2 +  \|\phi_{x_1z}\|^2 + \|\Delta_{x'} \psi'_{x_1}\|^2  + \|\operatorname{div}_x \psi_{x_1z}\|^2 \Big) \\
&\quad + C(\varepsilon + \delta^2) \Big(\|\nabla_x \psi'\|_{H^1}^2 + \|\phi_{z}\|_{H^1}^2 + \|\Delta_x \psi'\|_{H^1}^2 + \|\operatorname{div}_x \psi_{z}\|^2\Big). \nonumber
\end{aligned}
\eeq
To obtain a good term for $\psi_{1x_1x_1x_1}$, we use the same system again, but focus on the other component.
\begin{equation}
\begin{aligned}\nonumber
\partial_t \psi_{1x_1}
&+ \mathbf{u}\cdot \nabla_x \psi_{1x_1}
+ v p'(v)\,\phi_{x_1x_1}
+ vS\,\tilde v_{x_1x_1}
- \dot X(t)\,\tilde u_{x_1x_1}
+ (vF)\,\tilde u
_{x_1x_1}
\\[2mm]
&= \mu v\,\Delta_x \psi_{1x_1}
+ (\mu+\lambda)v\,\operatorname{div}_x \psi_{x_1x_1}
- \mathbf{u}_{x_1}\cdot \nabla_x \psi_{1}
- (vF)_{x_1}\,\tilde u_{x_1}
\\
&\quad - (vp'(v))_{x_1}\,\phi_{x_1}
- (vS)_{x_1}\,\tilde v_{x_1}
+ \mu v_{x_1}\,\Delta_x \psi_1
+ (\mu+\lambda)\,v_{x_1}\,\operatorname{div}_x \psi_{x_1}
\end{aligned}
\end{equation}
Multiplying both sides by $\frac{1}{(2\mu + \lambda)v} \psi_{1x_1x_1x_1}$ and using the same method as before, we obtain the following estimate.
\beq
\begin{aligned}
\|\psi_{1x_1x_1x_1}\|^2 &\le C\Big(\|\partial_t \psi_{1x_1}\|^2 + \|\nabla_x \psi_{1x_1}\|^2 +  \|\phi_{x_1x_1}\|^2 + \|\psi'_{x_1x_1x_1}\|^2  + \|\operatorname{div}_x \psi_{x_1x_1}\|^2 \Big) \\
&\quad + C(\varepsilon + \delta^2) \Big(\|\nabla_x \psi_1\|_{H^1}^2 + \|\phi_{x_1}\|_{H^1}^2 + \|\Delta_x \psi_1\|_{H^1}^2 + \|\operatorname{div}_x \psi_{x_1}\|^2 + \mathcal{G}_3 + \mathcal{G}^s + \mathcal{D} \Big)\nonumber
\end{aligned}
\eeq
Combining all the preceding estimates yields the desired result.
\end{proof}

\section{Proof of Proposition 4.1}
Choosing sufficiently small $\chi, \e, \delta > 0$ and combining the results of {Lemma \ref{L:22}} and {Lemma \ref{L:24}}, we obtain the following:
\beq
\begin{aligned} \label{eq:225}
&\sup_{0 \le t \le T}\big( \|(\phi,\psi)(t)\|_{H^2}^2 + \|\partial_t(\phi, \psi)(t) \|^2 \big) \\
&\quad
+ \int_0^T \Bigl(
 \delta |\dot X(t)|^2
+ \mathcal{G}^s(t) + \mathcal{D}(t) + \mathcal D_v(t) + \mathcal D_u(t)
 + \|(\phi_t, \psi_t)\|^2
 + \|\nabla_x \psi_t \|^2 
 + \mathcal{B}(t)
\Bigr)\,dt \\[2mm]
&\qquad \le C\|(\phi_0,\psi_0)\|_{H^2}^2 + C \|(\phi_t , \psi_t) (0)\|^2
+ C e^{-C\delta\beta} + C\delta^2 \varepsilon^2,
\end{aligned}
\eeq
where
\begin{align*} 
&\mathcal{B} = \|\psi_{1x_1} \big|_{x_1 = 0} \|_{L^2({\mathbb{T}^2})}^2
+ \|\phi\big|_{x_1 = 0} \|_{L^2({\mathbb{T}^2})}^2
+ \|\phi_t \big|_{x_1 = 0} \|_{L^2({\mathbb{T}^2})}^2
+ \|\nabla_x \phi\big|_{x_1=0}\|_{L^2({\mathbb{T}^2})}^2
+ \|\nabla_x^2 \phi\big|_{x_1=0}\|_{L^2({\mathbb{T}^2})}^2, \\[1mm]
& \mathcal D_v(t) = \int_\Omega \|\nabla_x^2 \phi\|^2\, dx, \quad \mathcal D_u(t) = \int_\Omega \big(\|\nabla_x \psi\|^2  +\|\nabla_x^2 \psi\| ^2 + \|\nabla_x^3 \psi\|^2 \big)\, dx.
\end{align*} 
We should handle the term $\|\partial_t (\phi, \psi)(0)\|$ on the right-hand side and obtain a good term for $\nabla_x \phi$.\medskip

\noindent$\bullet$ \textbf{Controlling $\|\partial_t (\phi_0, \psi_0)\|^2$:}
By substituting $t=0$ into (\refeq{eq:peq}) and applying the $L^2$-Minkowski inequality to both sides, we obtain the following:
\beq
\left\{
\begin{aligned}
\|\partial_t \phi(0)\|^2 &\le
 \|{u}_- \phi_{x_1}(0)\|^2
 + \|\dot{X}(0)\tilde v_{x_1}(0)\|^2
 + \|(vF)(0)\,\tilde v_{x_1}(0)\|^2
 + \|v_-\psi_{1x_1}(0)\|^2,
\\[2mm]
\|\partial_t \psi(0)\|^2 &\le
  \|{u}_- \,\nabla_x \psi_1(0)\|^2
 + \|v_-p'(v_-) \nabla_x \phi(0)\|^2
 + \|(vS)(0)\nabla_x \tilde v(0)\|^2\\[1mm]
 &\quad
 + \|(vF - \dot X)(0) \,\mathbf{\tilde u}_{x_1}(0)\|^2 + \|v_-\,\big(\mu \Delta_x \psi(0)
    + (\mu + \lambda) \nabla_x \operatorname{div}_x \psi(0)\big)\|^2.
\end{aligned}
\right.
\eeq
By Lemma~\ref{lem:shock-est}, we have
$$ \|(\tv_{x_1}, \mathbf \tu_{x_1} )\|^2 \le C\delta^3 e^{-C\delta \beta} $$
On the other hand, from the definitions of $F$ and $S$ in \eqref{eq:FS}, we obtain
$$ F \sim (\phi + \psi), \quad S \sim \phi $$
Combining these two observations with \eqref{dxbound}, we obtain the following estimate:
\begin{align}
\|\partial_t (\phi, \psi)(0)\|^2 \le C\|(\phi_0, \psi_0)\|_{H^2}^2 + C\delta^3 e^{-C\delta \beta}. \label{eq:226}
\end{align}

\noindent$\bullet$ \textbf{$L^2$-norm estimate of $\nabla_x \phi_t$:}
To investigate the time-asymptotic behavior, this term is essential, and it can be directly derived from the mass equation. By applying $\nabla_x$ to $(\refeq{eq:peq})_1$, we obtain the following:
$$ \nabla_x \phi_t = \operatorname{div}_x \psi \nabla_x v + v \, \nabla_x \operatorname{div}_x \psi - \nabla_x \big(\mathbf{u} \cdot \nabla_x \phi\big) + \dot X \nabla_x \tilde v_{x_1} - \tilde v_{x_1} \nabla_x (vF) - (vF)\nabla_x \tilde v_{x_1}. $$
By multiplying both sides by $\nabla_x \phi_t$, integrating over space, and applying the Lemmas 3.6 and \ref{L: estiF}, we obtain the following:
\beq
\begin{aligned}\label{eq:227}
\|\nabla_x \phi_t \|^2 &\le C \|\nabla_x \phi_t\| \Big(\|\operatorname{div}_x \psi\|_{H^1} \|\nabla_x \phi\|_{H^1} + \|\tilde v_{x_1} \|_{L^\infty} \|\operatorname{div}_x \psi\| + \|\nabla_x \operatorname{div}_x \psi\| \\
&\hspace{23mm}  + \|\nabla_x\big(\psi\cdot\nabla_x \phi \big) \| + \|( \tilde u \phi_{x_1} )_{x_1}\| + |\dot X| \|\nabla_x \tilde v_{x_1}\| \\
&\hspace{23mm} + \|\tilde v_{x_1}\|_{L^\infty} \|\nabla_x (vF)\| + \delta \|(vF)\tilde v_{x_1}^{1/2}\| \|\tilde v_{x_1}^{1/2}\|_{L^\infty} \Big) \\
&\le C\|\nabla_x^2 (\phi, \psi)\|^2 + C(\varepsilon+\delta+\chi) \Big(\|\nabla_x (\phi, \psi)\|^2 + \|\nabla_x \phi_t \|^2 \Big) \\
&\quad +C\delta^2 \Big( |\dot X|^2 + \mathcal{G}_3 + \mathcal{G}^s \Big).
\end{aligned}
\eeq
Choosing sufficiently small $\e,\, \delta,\, \chi > 0$ and by adding \eqref{eq:226} and \eqref{eq:227} to \eqref{eq:225}, we obtain the desired result.
\bibliographystyle{amsplain}
\bibliography{reference} 
\end{document}